\documentclass[a4paper,11pt]{article}
\usepackage[T1]{fontenc}
\usepackage[utf8]{inputenc}
\usepackage[english]{babel}
\usepackage{amsmath, amssymb, amsthm, amsfonts, mathtools, mathrsfs}	
\usepackage{enumerate, enumitem}
\usepackage{booktabs, caption, graphicx, subfig, wrapfig} 	
\usepackage{geometry}
\usepackage{xcolor}

\usepackage{orcidlink}

\usepackage{cite} 

\usepackage{upgreek}

\usepackage{bbm}

\newcommand{\R}{\mathbb{R}}
\newcommand{\N}{\mathbb{N}}

\newcommand{\C}{\mathbb{C}}

\numberwithin{equation}{section}

\theoremstyle{plain}
\newtheorem{theorem}{Theorem}[section] 
\newtheorem{proposition}[theorem]{Proposition} 
\newtheorem{lemma}[theorem]{Lemma}

\theoremstyle{definition}
\newtheorem{definition}[theorem]{Definition}
\newtheorem{remark}[theorem]{Remark}

\DeclarePairedDelimiter{\abs}{\lvert}{\rvert}
\DeclarePairedDelimiter{\norm}{\lVert}{\rVert}
\DeclarePairedDelimiter{\duality}{\langle}{\rangle}
\DeclareMathOperator{\sgn}{sign}

\DeclareMathOperator{\Tr}{Tr}

\DeclareMathOperator{\divergence}{div}
\DeclareMathOperator*{\argmin}{arg\,min}

\renewcommand{\div}{\divergence} 
\newcommand{\eps}{\varepsilon}
\renewcommand{\phi}{\varphi}
\renewcommand{\bar}{\overline}
\renewcommand{\vec}{\boldsymbol}

\def\genspazio #1#2#3#4#5{#1^{#2}(#5,#4;#3)}
\def\spazio #1#2#3{\genspazio {#1}{#2}{#3}T0}

\def\LT {\spazio L}
\def\HT {\spazio H}

\def\C #1#2{\mathcal{C}^{#1}([0,T];#2)}

\def\Lx #1{L^{#1}(\Omega)}
\def\Lt #1{L^{#1}(0,T)}
\def\Lqt #1{L^{#1}(Q_T)}
\def\Hx #1{H^{#1}(\Omega)}
\def\Wx #1{W^{#1}(\Omega)}
\def\Cx #1{\mathcal{C}^{#1}(\bar{\Omega})}
\def\Cqt #1{\mathcal{C}^{#1}(\bar{Q_T})}
\def\CzCu {\mathcal{C}^{0}(\bar{I}, \mathcal{C}^1(\bar{\Omega}))}

\def\Accorpa #1#2 #3 {\gdef #1{\eqref{#2}--\eqref{#3}}%
	\wlog{}\wlog{\string #1 -> #2 - #3}\wlog{}}

\def\ls{<}
\def\gs{>}
\def\mezzo {\frac{1}{2}}
\def\de {\mathrm{d}}
\def\ddt {\frac{\de}{\de t}}

\def\n {\vec{n}}
\def\hh {\mathbbm{h}}
\def\phib {\bar{\phi}}
\def\mub {\bar{\mu}}
\def\sigmab {\bar{\sigma}}
\def\ub {\bar{u}}
\def\vb {\bar{v}}
\def\Edr {E_{2,\rho}(I)}
\def\Eor {E_{0,\rho}(I)}
\def\Fr {F_\rho(I)}
\def\vpsi {\vec{\uppsi}}
\def\vzeta {\vec{\upzeta}}
\def\hv {\vec{h}}
\def\Acal {\mathcal{A}}
\def\Bcal {\mathcal{B}}
\def\Uad {\mathcal{U}_{\text{ad}}}
\def\Vad {\mathcal{V}_{\text{ad}}}
\def\Fiv {F^{(4)}}
\def\tp {t^+}
\def\trace {\mid \partial \Omega}
\def\Cbeta {\mathcal{C}^{\beta}(\Omega)}
\def\phiab {\abs{\phi}}
\def\intom {\int_\Omega}
\def\cdb {C_{\delta_\beta}}
\def\bstar {b^\ast}

\def\reaction {\sigma - \mu} 
\def\reactiontwo {\sigma_2 - \mu_2} 
\def\reactiondc {\sigma - \mu} 
\def\reactionbar {\bar{\sigma} - \bar{\mu}} 
\def\reactionlin {\rho - \eta} 
\def\reactiondiff {\theta - \zeta} 

\usepackage{hyperref}

\begin{document}
	

\begin{center}

\LARGE{\textbf{Maximal regularity and optimal control for a non-local Cahn-Hilliard tumour growth model}}

\vskip0.5cm


\large{\textsc{Matteo Fornoni$^1$} \orcidlink{0000-0002-9787-047X}} \\
\normalsize{e-mail: \texttt{matteo.fornoni01@universitadipavia.it}} \\
\vskip0.5cm



\small{$^1$Department of Mathematics "F. Casorati", Università degli Studi di Pavia, \\
	via Ferrata 5, 27100, Pavia, Italy}
\vskip0.5cm

\vskip0.5cm

\end{center}

\begin{abstract}\noindent
	We consider a non-local tumour growth model of phase-field type, describing the evolution of tumour cells through proliferation in presence of a nutrient. 
	The model consists of a coupled system, incorporating a non-local Cahn-Hilliard equation for the tumour phase variable and a reaction-diffusion equation for the nutrient. 
	First, we establish novel regularity results for such a model, by applying maximal regularity theory in weighted $L^p$ spaces. 
	This technique enables us to prove the local existence and uniqueness of a regular solution, including also chemotaxis effects. 
	By leveraging time-regularisation properties and global boundedness estimates, we further extend the solution to a global one.
	These results provide the foundation for addressing an optimal control problem, aimed at identifying a suitable therapy, guiding the tumour towards a predefined target. 
	Specifically, we prove the existence of an optimal therapy and, by studying the Fr\'echet-differentiability of the control-to-state operator and introducing the adjoint system, we derive first-order necessary optimality conditions.
	
	\vskip3mm
	
	
	\vskip3mm
	
	
\end{abstract}

\section{Introduction}


Mathematical modelling and analysis of tumour growth models is an expanding field that has been attracting much research in recent years. 
Indeed, it encompasses both modelling and numerical challenges, related to applicability with medical data (see \cite{AACGV2017, BC1997_freeboundary, BLM2008_cancer, FLNOW2019, Lorenzo2022_review, Lowengrub2010_review}), and analytical questions, related to well-posedness and regularity results (see \cite{FGR2015_TumGrowth, RSS2023, GLSS2016, GKT2022, PQV2014}).

In the present paper, we consider a diffuse interface tumour growth model, which is a non-local variant of the one firstly introduced in \cite{HZO2012}. 
In particular, it is a Cahn-Hilliard-based tumour growth model: for a general introduction to this type of models, we refer the reader to the recent review \cite{Fritz2023}.
The key feature of diffuse interface models is the presence of a phase-field variable $\phi$, representing the difference in volume fractions of tumour cells and healthy cells in a certain tissue. 
We assume that such variable can vary continuously between the tumour phase $\phi = 1$ and the healthy phase $\phi = -1$. 
Such approach, as opposed to a sharp-interface one, allows for an easier description of the tumour dynamics, both from an analytical and numerical point of view; additionally, topology changes in the tumour level set are also possible. 
In particular, the main object of our study is the following model:
\begin{alignat*}{2}
	& \partial_t \phi - \div (m(\phi) \nabla \mu) = P(\phi) (\sigma + \chi (1-\phi) - \mu) - \hh(\phi) u \qquad && \text{in } Q_T,  \\
	& \mu = AF'(\phi) + Ba \phi - BJ \ast \phi - \chi \sigma \qquad && \text{in } Q_T,  \\
	& \partial_t \sigma - \div (n(\phi) \nabla (\sigma + \chi (1 - \phi))) = - P(\phi) (\sigma + \chi (1-\phi) - \mu) + v \qquad && \text{in } Q_T, \\
	& m(\phi) \partial_{\n} \mu = n(\phi) \partial_{\n} (\sigma + \chi (1-\phi)) = 0 \qquad && \text{on } \Sigma_T, \\
	& \phi(0) = \phi_0, \quad \sigma(0) = \sigma_0 \qquad && \text{in } \Omega,
\end{alignat*}
where $Q_T=\Omega \times (0,T)$, $\Sigma_T = \partial \Omega \times (0,T)$ and $\Omega \subset \R^N$, $N=2,3$, is an open bounded sufficiently smooth domain, with exterior normal unit vector $\n$. 

We now briefly comment on the structure of the system.
Other than $\phi$, the second main variable of the system is $\sigma$, representing the concentration of a nutrient (e.g.~oxygen or glucose), which is assumed to be the main consumable source used by the tumour cells to proliferate.
As one can see, the phase-field variable $\phi$ satisfies a non-local Cahn-Hilliard equation, whereas $\sigma$ satisfies a reaction-diffusion equation, and the two equations are non-trivially coupled through reaction terms and cross-diffusion ones, due to chemotactic effects. 
The non-locality is given by the convolution term $Ba \phi - BJ \ast \phi$ in the equation for the chemical potential $\mu$ associated to the Cahn-Hilliard equation. 
Throughout the paper, we call $a = J \ast 1$ (see also the forthcoming hypothesis \ref{ass:j}).
Including non-local effects is of paramount importance, for instance when modelling cell-to-cell adhesion effects (see \cite{CLSW2011_nonlocalmodels}). 
The expression of $\mu$, as well as the form of the non-local terms, is deduced from the variational derivative of the Ginzburg-Landau free energy associated to the system: 
\begin{equation}
	\label{eq:free_energy}
	\begin{split}
		 \mathcal{E}(\phi,\sigma) & = \int_{\Omega} A F(\phi) \, \de x + \int_{\Omega} \int_{\Omega} \frac{B}{4} J(x - y)(\phi(x) - \phi(y))^2 \, \de x \, \de y \\ 
		 & \qquad + \int_{\Omega} \frac{1}{2} \abs{\sigma}^2 + \chi \sigma (1-\phi) \, \de x,
	\end{split}
\end{equation}
where $A,B \gs 0$ are parameters related to the width of the diffuse interface.
We note that, without reaction terms and sources, the system above can be seen as an $H^{-1}$-gradient flow of the previous free energy. 

With this in mind, the first term in the free energy drives the evolution of the phase-field towards the minima of the function $F$, which is typically taken as a double-well potential with equal minima in $-1$ and $1$. 
Typical choices for such potential are for instance 
\begin{align*}
	& F_{\text{reg}}(s) = \frac{1}{4} (1 - s^2)^2, \quad s \in \R, \\
	& F_{\text{sing}}(s) = \frac{\theta}{2} \left[ (1+s) \log (1+s) + (1-s) \log (1-s) \right] - \frac{\theta_0}{2} s^2, \quad s \in (-1,1), \quad 0 \ls \theta \ls \theta_0,
\end{align*}
From a modelling point of view, the best one would be the singular potential $F_{\text{sing}}$.
Indeed, its minima are strictly inside the interval $(-1,1)$, but sufficiently close to the extrema, and its derivative blows up in $\pm 1$. 
This still allows to restrict the evolution inside the physical interval $[-1,1]$, but also enables to show the existence of regular solutions, as long as a proper separation property is proved (see \cite{GGG2023} for a recent overview on the matter, as well as \cite{P2023}).
In a few words, this amounts to proving that $\phi$ stays strictly inside the interval $(-1,1)$ during its evolution, thus allowing free differentiation of the potential $F_{\text{sing}}$.
As such property is not always easy to prove, in many cases a suitable choice of the potential is the polynomial one $F_{\text{reg}}$, which, albeit not imposing the physical constraint, still typically guarantees good results in practice.  
Going back to the free energy \eqref{eq:free_energy}, the second term is the one responsible for the non-local character of the model. 
Indeed, given a sufficiently regular symmetric convolution kernel $J$, typically of Newton or Bessel type, it can be seen as a non-local approximation of the standard Dirichlet energy that one would find for the local Cahn-Hilliard equation. 
From the modelling viewpoint, this term wants to penalise steep transitions between the two phases, in a non-local way.
Indeed, if the kernel $J$ is suitably peaked around zero, one can also recover the local model through asymptotic procedures (see \cite{DRST2020, DST2021_1, DRST2023}).
Applications of the local Cahn-Hilliard equation to tumour growth problems have been widely studied, as one can see from the many articles already cited in this introduction, but the study of the non-local counterpart, where also some long-range interactions are considered, is just recently flourishing (see \cite{FLNOW2019, CLSW2011_nonlocalmodels, SS2021}).
Our paper fits exactly in this framework.
To conclude our commentary of the free energy \eqref{eq:free_energy}, the third term is responsible for the nutrient diffusion mechanism and the fourth one for the chemotaxis effect, which is related to the natural movement of tumour cells towards regions with higher nutrient concentrations.
Here we model chemotaxis through a cross-diffusion effect with intensity $\chi \ge 0$, but more sophisticated models are also available, using for instance the celebrated Keller-Segel equations (see \cite{RSS2023}). 

Now, going back to our system, we comment on the remaining parameters. 
The functions $m(\phi)$ and $n(\phi)$ are called mobilities and regulate the diffusion processes of the two variables. 
In some cases, $m(\phi)$ can be taken degenerate in $\pm 1$, like $m(\phi) = 1 - \phi^2$, to compensate the singularities of the potential $F_{\text{sing}}$ (see \cite{FGG2021, FLR2017}), but here we will work with the constant mobility case. 
The function $P(\phi)$, instead, is a proliferation function, calibrating the strength of the reaction terms, which, in turn, are written in this form due to chemical phenomenological laws (see \cite{HZO2012}). 
Generally, it can be taken of the form $P(s) = \max \{ 0, \min \{ \frac{1}{2}(1+s), 1 \}  \}$, for $s \in \R$, especially in case of young tumours, but other choices are also possible. 
To conclude, $u$ and $v$ are external sources that can be thought as medical therapies on the tumour: in particular, $u$ can be seen as radiotherapy acting directly on tumour cells and $v$ as a chemotherapy acting through the nutrient.
In this case, the function $\hh$ is a bounded function that can be used to distribute the radiotherapy through particular strategies.
In most cases, its expression would be similar to the one of $P$.


The main purpose of this work is to analyse the previously introduced model with constant mobilities and regular potential. 
Indeed, we now assume $m(\phi) = n(\phi) \equiv 1$ for simplicity and consider $F$ to be a regular potential defined on the whole $\R$, satisfying certain hypotheses for which $F_{\text{reg}}$ is certainly included (see the forthcoming hypotheses \ref{ass:fc0}--\ref{ass:fgrowth}). 
Hence, we actually consider the following system:
\begin{alignat}{2}
	& \partial_t \phi - \Delta \mu = P(\phi) (\sigma + \chi (1-\phi) - \mu) - \hh(\phi) u \qquad && \text{in } Q_T,  \label{eq:phi}\\
	& \mu = AF'(\phi) + Ba \phi - BJ \ast \phi - \chi \sigma \qquad && \text{in } Q_T,  \label{eq:mu} \\
	& \partial_t \sigma - \Delta \sigma + \chi \Delta \phi = - P(\phi) (\sigma + \chi (1-\phi) - \mu) + v \qquad && \text{in } Q_T, \label{eq:sigma}
\end{alignat}
together with the following homogeneous Neumann boundary conditions and initial conditions:
\begin{alignat}{2}
	& \partial_{\n} \mu = \partial_{\n} (\sigma - \chi \phi) = 0 \qquad && \text{on } \Sigma_T, \label{bc} \\
	& \phi(0) = \phi_0, \quad \sigma(0) = \sigma_0 \qquad && \text{in } \Omega. \label{ic}
\end{alignat}
Our first main result concerns the existence and uniqueness of highly regular solutions to \eqref{eq:phi}--\eqref{ic}. To achieve this, we use an approach based on maximal regularity theory in weighted $L^p$ spaces for parabolic equations with inhomogeneous Robin boundary conditions, as was recently done in \cite{FGG2021} for a non-local Cahn-Hilliard equation with singular potential and degenerate mobility. 
The theory we apply was developed by M. Meyries and R. Schnaubelt in the series of papers \cite{M2012, MS2012_spaces, MS2012_maxreg}, as well as the PhD thesis \cite{meyries_phd} of M. Meyries.

To the best of our knowledge, this is the first time that such results are applied to a Cahn-Hilliard system with regular potential and constant mobilities, especially in presence of reaction terms, chemotaxis, external sources and space-dependent terms. 
In particular, the same procedure also guarantees new regularity results for the standard non-local Cahn-Hilliard equation with constant mobility and regular potential (see Remark \ref{regularity:CH}).
Indeed, other than the application of maximal regularity theory, the crucial point is to prove that weak solutions are globally bounded through an Alikakos-Moser iteration technique.
On this point, we mention that this is an intrinsic problem related to the choice of a regular potential in the non-local Cahn-Hilliard equation.
Opting for a singular potential, instead, would replace such problem with the hurdle of proving the aforementioned separation property.
Great advances were made in \cite{P2023}, where such property was proved for the non-local Cahn-Hilliard equation in three space dimensions. 
Moreover, such results were later used in \cite{PS2023} to prove regularity results similar to ours in case of an advective non-local Cahn-Hilliard equation with constant mobility and singular potential.
Let us mention, however, that their same procedure would not work in case of our tumour growth system, since the same proof of the separation property cannot be easily adapted to the presence of reaction terms.
Finally, we observe that the Alikakos-Moser iteration scheme for our tumour growth system works only if we neglect the chemotaxis effect, i.e.~we put $\chi = 0$, since the low regularity of weak solutions is not enough to treat the cross-diffusion terms (see Proposition \ref{phi:linf}).
Thus, we are actually able to prove \emph{local} maximal regularity and uniqueness for the full model (see Theorem \ref{local:maxsol}), but we are able to extend the local maximal solution to a \emph{global} one only if $\chi = 0$ (see Theorem \ref{global:maxreg}).
We stress that the local maximal regularity result of Theorem \ref{local:maxsol} holds also in the standard non-weighted case, that would be $\rho =1$ in the setting of Section \ref{sect:maxreg}. 
However, one of the main ingredients for the extension to a global regular solution is the time-regularising effect of the weighted spaces, which allows to prove the key estimate contained in Lemma \ref{smooth:eff}.
For this reason, it is crucial to have maximal regularity in the time-weighted spaces to then recover a global regular solution even in the non-weighted case.
To conclude this section about regularity, we just mention that our analysis was deeply inspired by \cite{FGG2021}, but in this work we chose to deal with the different case of regular potential and constant mobility. 
However, we strongly believe that all results on the tumour growth model should also easily extend to the case of singular potential and degenerate mobility as in \cite{FGG2021}. 
In particular, also the application to the optimal control problem should be possible, by mimicking what was already done in \cite{FLS2021}.

Speaking of optimal control problems, we now introduce the second main goal of this paper, as well as the key reason why we strived to prove the existence of such regular solutions. 
Indeed, highly regular solutions of the tumour growth model are essentially needed if one wants to characterise the solutions to the optimal control problem through optimality conditions. 
In particular, the main step is to study in detail the differentiability properties of the control-to-state operator.
Due to these reasons, to prove first order necessary conditions, we are forced to work in the chemotaxis-free case, meaning that we put $\chi = 0$ and consider the system:
\begin{alignat}{2}
	& \partial_t \phi - \Delta \mu = P(\phi) (\sigma - \mu) - \hh(\phi) u \qquad && \text{in } Q_T,  \label{eq:phi2}\\
	& \mu = AF'(\phi) + Ba \phi - BJ \ast \phi \qquad && \text{in } Q_T,  \label{eq:mu2} \\
	& \partial_t \sigma - \Delta \sigma = - P(\phi) (\sigma - \mu) + v \qquad && \text{in } Q_T. \label{eq:sigma2} \\
	& \partial_{\n} \mu = \partial_{\n} \sigma = 0 \qquad && \text{on } \Sigma_T, \label{bc2} \\
	& \phi(0) = \phi_0, \quad \sigma(0) = \sigma_0 \qquad && \text{in } \Omega. \label{ic2}
\end{alignat}
Our aim is to study the following optimal distributed control problem:

\bigskip
\noindent(CP) \textit{Minimise the cost functional}
\begin{equation}
	\begin{split}
		\mathcal{J}(\phi, \sigma, u, v) & = \, \frac{\alpha_{\Omega}}{2} \int_{\Omega} |\phi(T) - \phi_{\Omega}|^2 \,\de x  + \frac{\alpha_Q}{2} \int_{0}^{T} \int_{\Omega} |\phi - \phi_Q|^2 \,\de x \,\de t \\
		& \quad + \frac{\beta_{\Omega}}{2} \int_{\Omega} |\sigma(T) - \sigma_{\Omega}|^2 \,\de x  + \frac{\beta_Q}{2} \int_{0}^{T} \int_{\Omega} |\sigma - \sigma_Q|^2 \,\de x \,\de t \\ 
		& \quad + \frac{\alpha_u}{2} \int_{0}^{T} \int_{\Omega} |u|^2 \,\de x \,\de t + \frac{\beta_v}{2} \int_{0}^{T} \int_{\Omega} |v|^2 \,\de x \,\de t,
	\end{split}
\end{equation}
\textit{subject to the control constraints}
\begin{equation} 
	\begin{split}
		& u \in \Uad := \{ u \in L^{\infty}(Q_T) \mid u_{\text{min}} \le u \le u_{\text{max}} \text{ a.e.~in } Q_T \}, \\
		& v \in \Vad := \{ v \in L^{\infty}(Q_T) \mid v_{\text{min}} \le v \le v_{\text{max}} \text{ a.e.~in } Q_T \},
	\end{split}
\end{equation}
\textit{and to the state system \eqref{eq:phi2}-\eqref{ic2}}.
\bigskip

\noindent Here $\alpha_\Omega, \alpha_Q, \beta_\Omega, \beta_Q, \alpha_u, \beta_v$ are non-negative parameters that can be used to select which targets have to be privileged. 
The function $\phi_\Omega$ is a final target for the tumour distribution, for instance one that could be suitable for surgery; whereas $\phi_Q$ is a possibly desired evolution. 
In the same way, $\sigma_\Omega$ and $\sigma_Q$ are respectively a final target and a desired evolution for the nutrient. 
Finally, the last two terms in the cost functional penalise large use of radiotherapy or chemotherapy, which could still harm the patient in the long run. 
The aim of the optimal control problem is, then, to find the best therapies $u$ and $v$, which can lead the evolution of the tumour to the desired targets.  
Similar control problems for tumour growth models were studied in \cite{CGLMRR2021, GLS2021_sparseoc, CSS2021_secondorder, CGRS2017, CRW2021}, but we can still count few contributions in the non-local case (see \cite{FLS2021, F2023_viscous, RSS2023}).
Regarding (CP), in Theorem \ref{thm:excont} we show the existence of an optimal pair $(\ub, \vb)$. 
Then, through careful analysis of the control-to-state operator and the introduction of the adjoint system, in Theorem \ref{thm:optcond} we prove the first-order necessary optimality conditions, which have to be satisfied by the optimal pair. 
We stress that proving necessary optimality conditions is the first step needed in order to justify numerical treatment of the optimal control problem. 
Indeed, the most commonly used algorithms are gradient descent schemes, where the descent direction is found through the optimality condition (see \cite[Section 3.7]{troltzsch} and references therein).

We finally mention that the present paper can be seen as a follow-up to \cite{F2023_viscous},  where the same optimal control problem for a similar model was analysed. 
The model in \cite{F2023_viscous} is essentially the same as \eqref{eq:phi}--\eqref{ic}, up to an extra viscous regularisation term $\tau \partial_t \phi$ in \eqref{eq:mu} with $\tau \gs 0$. 
The viscous term made it easier to prove high regularity results for the relaxed system, even with singular potentials and chemotaxis, however this called for an additional constraint on the control $u$, namely it was necessary to assume that $u$ was uniformly bounded in $H^1(0,T; L^2(\Omega))$.
Such additional constraint is definitely not a desired one in practice, as it makes it harder to compute the $L^2$-orthogonal projection onto $\Uad$, which is something that necessarily has to be done when solving numerically.
In the present case, not only are we able to solve the optimal control problem without the viscous relaxation, but we are also able to avoid asking additional regularity on $u$.
This is mainly due to our new strategy for proving high regularity for the solution of the tumour growth model, relying on maximal regularity theory.
We also mention that such theory was actually crucial in proving regularity results for our system in three dimensions.
Indeed, our initial attempts made use of more standard techniques, inspired by \cite{FGK2013}, but they only got us a regular solution in two dimension without chemotaxis, while also assuming $u \in H^1(0,T;L^2(\Omega))$.

The plan of the paper is the following. 
In Section \ref{sect:weak}, we state the notations and the main hypotheses on the parameters that will be used throughout the paper. 
Then, we state a first result about existence of weak solutions for \eqref{eq:phi}--\eqref{ic}, which directly follows from the results proven in \cite{FLR2017}. 
In Section \ref{sect:maxreg}, we start by introducing all the machinery of maximal regularity theory in weighted $L^p$ spaces and apply it to prove the first main result of the paper: Theorem \ref{local:maxsol} about local existence and uniqueness of maximal solutions, even for $\chi \gs 0$. 
Moreover, we prove a technical result (Lemma \ref{smooth:eff}) about a time-regularisation effect due to the use of weighted spaces, which is the key ingredient allowing the extension to a global solution. 
Section \ref{sect:global} is dedicated to proving strong well-posedness for the chemotaxis-free system \eqref{eq:phi2}--\eqref{ic2}. 
First, starting from weak solutions, we prove global boundedness through an Alikakos-Moser iteration scheme (Proposition \ref{phi:linf}) and a global H\"older regularity result (Proposition \ref{phisigma:holder}). 
Using these and the previous Lemma \ref{smooth:eff}, we hence prove the second main result of the paper about existence and uniqueness of global highly regular solutions (Theorems \ref{global:maxreg} and \ref{strong:sols}). 
Finally, in Theorem \ref{thm:contdepstrong} we prove a strong continuous dependence result from data, which is the first key step in the analysis of the optimal control problem.
In Section \ref{sect:oc}, we consequently study the optimal control problem (CP). 
Indeed, we first prove existence of a solution in Theorem \ref{thm:excont}. 
Then, we consider the linearised system and prove the Fr\'echet-differentiability of the control-to-state operator in Theorem \ref{thm:frechet}. 
In conclusion, we introduce the related adjoint system and deduce the first-order necessary optimality conditions in Theorem \ref{thm:optcond}, the third main result of the paper.
 
\section{Preliminaries and existence of weak solutions}
\label{sect:weak}

We now introduce some notation that will be used throughout the paper. We denote with $\Omega \subset \R^N$, $N=2,3$ an open bounded domain with boundary $\partial \Omega$ of class $\mathcal{C}^2$ and exterior normal unit vector $\vec{n}$, whereas $T \gs 0$ is a fixed final time. 
The $\mathcal{C}^2$ requirement for $\partial \Omega$ is needed for regularity estimates in Section \ref{sect:maxreg}, while for weak solutions one can just assume that $\partial \Omega$ is Lipschitz. 
For convenience, we also denote $Q_t = \Omega \times (0,t)$ and $\Sigma_t = \partial \Omega \times (0,t)$, for any $t \in (0,T]$.  

Next, we recall the usual conventions regarding the Hilbertian triplet used in this context. 
If we define
\[ H = L^2(\Omega), \quad V = H^1(\Omega), \quad W = \{ u \in H^2(\Omega) \mid \partial_{\n} u = 0 \text{ on } \partial \Omega \}, \]
then we have the continuous and dense embeddings:
\[ W \hookrightarrow V \hookrightarrow H \cong H^* \hookrightarrow V^* \hookrightarrow W^*. \]
We denote by $\duality{\cdot, \cdot}_V$ the duality pairing between $V^*$ and $V$ and by $(\cdot, \cdot)_H$ the scalar product in $H$.
Regarding Lebesgue and Sobolev spaces, we will use the notation $\norm{\cdot}_{\Lx p}$ for the $\Lx p$-norm and $\norm{\cdot}_{\Wx{k,p}}$ for the $\Wx {k,p}$-norm, with $k \in \N$ and $1 \le p \le \infty$. 
Moreover, we observe that, by elliptic regularity theory, an equivalent norm on $W$ is
\[ \norm{u}^2_W = \norm{u}^2_H + \norm{\Delta u}^2_H. \]
Finally, we recall the Riesz isomorphism $\mathcal{N}: V \to V^*$:
\[ \duality{\mathcal{N}u, v}_V := \int_{\Omega} \left( \nabla u \cdot \nabla v + uv \right) \, \de x \quad \forall u,v \in V. \]
It is well-known that for $u\in W$ we have $\mathcal{N}u = - \Delta u + u \in H$ and that the restriction of $\mathcal{N}$ to $W$ is an isomorphism from $W$ to $H$. 
Additionally, by the spectral theorem, there exists a sequence of eigenvalues $0 < \lambda_1 \le \lambda_2 \le \dots$, with $\lambda_j \to +\infty$, and a family of eigenfunctions $w_j \in W$ such that $\mathcal{N}w_j = \lambda_j w_j$, which forms an orthonormal basis in $H$ and an orthogonal basis in $V$. 
In particular, $w_1$ is constant.

Finally, we recall some useful inequalities that will be used throughout the paper: 
\begin{itemize}
	\item \emph{Gagliardo-Nirenberg inequality}. Let $\Omega \subset \R^N$ bounded Lipschitz, $m\in\N$, $1 \le r,q \le \infty$, $j\in\N$ with $0\le j \le m$ and $j/m \le \alpha \le 1$ such that
	\[ \frac{1}{p} = \frac{j}{N} + \left( \frac{1}{r} - \frac{m}{N} \right) \alpha + \frac{1-\alpha}{q}, \]
	then \[ \norm{D^j f}_{L^p(\Omega)} \le C \, \norm{f}^\alpha_{W^{m,r}(\Omega)} \norm{f}^{1-\alpha}_{L^q(\Omega)}. \]
	In particular, we recall the following versions with $N=2,3$:
	\begin{equation}
		\label{gn:ineq}
		\begin{split}
			& \norm{f}_{\Lx4} \le C \norm{f}^{1/2}_{\Hx1} \norm{f}^{1/2}_{\Lx2} \quad \text{ if } N = 2, \\
			& \norm{f}_{\Lx3} \le C \norm{f}^{1/2}_{\Hx1} \norm{f}^{1/2}_{\Lx2} \quad \text{ if } N = 3.
		\end{split}
	\end{equation}
	\item \emph{Agmon's inequality}. Let $\Omega \subset \R^N$ bounded Lipschitz, $0 \le k_1 < N/2 < k_2$ and $0 < \alpha <1$ such that $N/2 = \alpha k_1 + (1-\alpha) k_2$, then
	\[ \norm{f}_{L^\infty(\Omega)} \le C \, \norm{f}^\alpha_{H^{k_1}(\Omega)} \norm{f}^{1-\alpha}_{H^{k_2}(\Omega)}. \]
	In particular, we recall the following versions with $N=2,3$:
	\begin{equation}
		\label{agmon}
		\begin{split}
			& \norm{f}_{\Lx\infty} \le C \norm{f}^{1/2}_{\Hx1} \norm{f}^{1/2}_{\Lx2} \quad \text{ if } N = 2, \\
			& \norm{f}_{\Lx\infty} \le C \norm{f}^{1/2}_{\Hx2} \norm{f}^{1/2}_{\Hx1} \quad \text{ if } N = 3.
		\end{split}
	\end{equation}
\end{itemize}
Note that all constants $C > 0$ mentioned above depend only on the measures of the sets and the parameters, not on the actual functions. 

Now we introduce the structural assumptions on the parameters of our model \eqref{eq:phi}--\eqref{ic}, which we recall here for convenience:
\begin{alignat*}{2}
	& \partial_t \phi - \Delta \mu = P(\phi) (\sigma + \chi (1-\phi) - \mu) - \hh(\phi) u \qquad && \text{in } Q_T, \\
	& \mu = AF'(\phi) + Ba \phi - BJ \ast \phi - \chi \sigma \qquad && \text{in } Q_T, \\
	& \partial_t \sigma - \Delta \sigma + \chi \Delta \phi = - P(\phi) (\sigma + \chi (1-\phi) - \mu) + v \qquad && \text{in } Q_T, \\
	& \partial_{\n} \mu = \partial_{\n} (\sigma - \chi \phi) = 0 \qquad && \text{on } \Sigma_T, \\
	& \phi(0) = \phi_0, \quad \sigma(0) = \sigma_0 \qquad && \text{in } \Omega.
\end{alignat*}
We assume the following:
\begin{enumerate}[font = \bfseries, label = A\arabic*., ref = \bf{A\arabic*}]
	\item\label{ass:coeff} $A,B>0$ and $\chi \ge 0$.
	\item\label{ass:j} $J \in W^{1,1}_{\text{loc}}(\R^N)$ is a symmetric convolution kernel, namely $J(z) = J(-z)$ for any $z \in \R^N$. Moreover, we suppose that
	\[ a(x) := (J \ast 1)(x) = \int_{\Omega} J(x-y) \, dy \ge 0 \quad \text{a.e. } x \in \Omega \]
	and also that we have the bounds:
	\[	a^\ast := \sup_{x\in\Omega} \, \int_{\Omega} \abs{J(x-y)} \, dy < +\infty, \qquad
		b^\ast := \sup_{x\in\Omega} \, \int_{\Omega} \abs{\nabla J(x-y)} \, dy < +\infty.  \]
	\item\label{ass:fc0} $F \in \mathcal{C}^2(\R)$ and there exists $c_0 > \chi^2 \ge 0$ such that
		\[ A F''(s) + B a(x) \ge c_0 \quad \forall s \in \R \quad \text{a.e. } x \in \Omega. \] 
	\item\label{ass:fbelow} There exist $c_1 \in \R$ and $c_2 > \frac{\chi^2}{A}$ such that
		\[ F(s) \ge c_2 \abs{s}^2 - c_1 \quad \forall s \in \R. \]
	\item\label{ass:fgrowth} Assume that there exist $z \in (1,2]$, $c_3 >0$ and $c_4 \ge 0$ such that
		\[ \abs{F'(s)}^z \le c_3 F(s) + c_4 \quad \forall s \in \R. \]
	\item\label{ass:p} $P \in \mathcal{C}^0(\R)$ and there exist $c_5 >0$ and $q \in [1,\frac{4}{3}]$ such that 
		\[ 0 \le P(s) \le c_5 (1+\abs{s}^q) \quad \forall s \in \R. \]    
	\item\label{ass:h} $\hh \in \mathcal{C}^0(\R) \cap L^\infty(\R)$. 
	\item\label{ass:u} $u \in L^\infty(Q_T)$ and $v \in L^2(0,T; V^*)$.	
	\item\label{ass:initial} $\phi_0 \in H$ with $F(\phi_0) \in L^1(\Omega)$ and $\sigma_0 \in H$.
\end{enumerate}

\noindent Finally, we would like to stress that, in the following, we will extensively use the symbol $C>0$ to denote positive constants, which may change from line to line. 
They will depend only on $\Omega$, $T$, the parameters and on the norms of the fixed functions introduced in hypotheses \ref{ass:coeff}--\ref{ass:initial} and possible subsequent ones. 
Sometimes, we will also add subscripts on $C$ to highlight some particular dependences of these constants. 

We now state a first result about existence of weak solutions to our system \eqref{eq:phi}--\eqref{ic}. We would like to point out that existence of weak solutions under assumptions \ref{ass:coeff}--\ref{ass:initial} was already proved in \cite{FLR2017}, when $u \equiv v \equiv 0$. In our case, the presence of $v$ does not hinder the cited proof, since it can be easily treated. However, the term $- \hh(\phi) u$ requires some little changes in said proof, by following what was done in \cite[Remark 2.7]{F2023_viscous}. Below we give just a brief idea on how to modify the argument, by leaving most of the details to the interested reader.

\begin{theorem}
\label{thm:weaksols}
Under assumptions \emph{\ref{ass:coeff}--\ref{ass:initial}}, there exists a weak solution $(\phi, \mu, \sigma)$ to \eqref{eq:phi}--\eqref{ic}, such that 
\begin{align*}
	& \phi \in H^1(0,T;V^*) \cap \mathcal{C}^0([0,T],H) \cap L^2(0,T;V), \\
	& \mu \in L^2(0,T;V), \\
	& \sigma \in H^1(0,T;V^*) \cap \mathcal{C}^0([0,T],H) \cap L^2(0,T;V), 
\end{align*}
which satisfies 
\[ \phi(0) = \phi_0 \quad \text{and} \quad \sigma(0) = \sigma_0 \quad \text{in } H \]
and the following variational formulation for a.e. $t \in (0,T)$ and for any $\zeta \in V$:
\begin{align}
	& \duality{\phi_t, \zeta}_V + (\nabla \mu, \nabla \zeta)_H = (P(\phi)(\sigma + \chi(1-\phi) - \mu), \zeta)_H - (\hh(\phi) u, \zeta)_H, \label{varform:phi} \\
	& (\mu,\zeta)_H = (AF'(\phi) + Ba \phi - BJ \ast \phi - \chi \sigma, \zeta)_H, \label{varform:mu} \\
	& \duality{\sigma_t, \zeta}_V + (\nabla \sigma - \chi \nabla \phi, \nabla \zeta)_H = - (P(\phi)(\sigma + \chi(1-\phi) - \mu), \zeta)_H + \duality{v, \zeta}_V. \label{varform:sigma}
\end{align}
In particular, there exists a constant $C>0$, depending only on the parameters of the model and on the data $\phi_0$, $\sigma_0$, $u$ and $v$, such that: 
\begin{equation}
	\label{weaknorms:est}
	\begin{split}
	& \norm{\phi}_{H^1(0,T;V^*) \cap L^\infty(0,T,H) \cap L^2(0,T;V)} + \norm{\mu}_{ L^2(0,T;V)} \\
	& \quad + \norm{\sigma}_{H^1(0,T;V^*) \cap L^\infty(0,T,H) \cap L^2(0,T;V)} \le C.
	\end{split}
\end{equation}
\end{theorem}

\begin{proof}
	The proof follows exactly the argument of \cite[Theorem 2.1]{FLR2017}, with one main difference due to the presence of the source terms. The main energy estimate is done by testing \eqref{eq:phi} by $\mu$, \eqref{eq:mu} by $- \phi_t$ and \eqref{eq:sigma} by $\sigma + \chi (1-\phi)$. Therefore, the two extra terms to treat are $-(\hh(\phi) u, \mu)_H$ and $\duality{v, \sigma + \chi  (1-\phi)}_V$. The second one can be easily treated by duality properties and Young's inequality. Regarding the first one, instead, one can argue exactly as in the first part of \cite[Remark 2.7]{F2023_viscous}, by assuming $\tau = 0$. Observe that, in this case, one also has to slightly modify the estimate from below of the energy $E(t)$ as in \cite[Remark 2.7]{F2023_viscous}. In this way, the bound $F(\phi) \in \LT \infty {\Lx1}$ directly comes from the first energy estimate. 
	
	The rest of the procedure, also including the Galerkin approximation and the passage to the limit, is the same as \cite[Theorem 2.1]{FLR2017}, so we omit the details.
\end{proof}

\section{Local Maximal $L^p_\rho$ regularity}
\label{sect:maxreg}

As already stated in the introduction, in order to prove further regularity for the solutions of our system \eqref{eq:phi}-\eqref{ic}, we use an approach based on maximal regularity theory in weighted $L^p$ spaces. Before starting, we recall some notation and results introduced in the series of papers \cite{M2012, MS2012_spaces, MS2012_maxreg}. Then, we apply their results to prove local existence and uniqueness of maximal solutions. In our case, the difference from the previous works is that we are dealing with a system of equations containing also some space-dependent and time-dependent terms, as well as some external source terms. Therefore we have to make some adjustments, but the core of the argument is similar to the one of \cite{FGG2021}. 

\subsection{Functional framework}

We start by introducing the function spaces that we are going to use and by reformulating our problem \eqref{eq:phi}-\eqref{ic} in a more suitable way. 
Following \cite{MS2012_spaces}, for $p\in (1,+\infty)$, $\rho \in (1/p,1]$, $X$ real Banach space and $T\in (0,+\infty)$, we introduce the time-weighted spaces:
\begin{gather*}
		L^p_\rho (0,T;X) := \Big\{ f : (0,T) \to X \hbox{ strongly measurable such that } \\
		\qquad \qquad \qquad \qquad \qquad \norm{f}^p_{L^p_\rho(0,T;X)} = \int_0^T t^{p(1-\rho)} \norm{f(t)}_X^p \, \de t < +\infty \Big\}, \\
		W^{1,p}_\rho (0,T;X) := \left\{ f \in L^p_\rho (0,T;X) \mid \partial_t f \in L^p_\rho (0,T;X)  \right\}.
\end{gather*}
Observe that $\rho =1$ yields the unweighted case, which means that $L^p_1 = L^p$, moreover one can easily see that $L^p(0,T;X) \hookrightarrow L^p_\rho(0,T;X)$ for any $\rho \in (1/p,1]$. Then, in a standard way, one can define the spaces $W^{k,p}_\rho (0,T;X)$ for any $k \in \N$ and, by real interpolation, also the fractional order spaces $W^{s,p}_\rho (0,T;X)$ for any $s\in \R_+$, as done in \cite{MS2012_spaces}. 

\begin{remark}
	\label{rmk:weight}
	We observe that the temporal weight $t^{p(1-\rho)}$ has a regularising effect only for $t=0$. Namely, if $f \in L^p_\rho(0,T;\Lx p)$ with $\rho \in  (1/p,1]$, then $f \in L^p_1(s,T; \Lx p)$ for any $s \gs 0$. Indeed:
	\[ +\infty \gs \int_s^T t^{p(1-\rho)} \norm{f}^p_{\Lx p} \, \de t \ge \left( \min_{[s,T]} t^{p(1-\rho)} \right) \int_s^T \norm{f}^p_{\Lx p} \, \de t \implies \int_s^T \norm{f}^p_{\Lx p} \, \de t \ls + \infty.  \]
	However, this is enough to prove some useful time-regularisation estimates (see Lemma \ref{smooth:eff}).
\end{remark}

From now on, we assume that the bounded domain $\Omega \subset \R^N$, $N=2,3$, has boundary of class $\mathcal{C}^2$. Next, we introduce the actual spaces we work with for our maximal regularity theory. Let $I=(0,T)$ be a finite time-interval, $p \in (N+2, +\infty)$ and $\rho \in (1/p,1]$. We define the maximal regularity class 
\[ \Edr := W^{1,p}_\rho (I; L^p(\Omega)) \cap L^p_\rho (I; W^{2,p}(\Omega)),  \]
the boundary class
\[ \Fr := W^{\frac{1}{2} - \frac{1}{2p}, \, p}_{\rho} (I; L^p(\partial \Omega)) \cap L^p_\rho (I; W^{1-\frac{1}{p}, \, p} (\partial \Omega)), \]
as well as the starting space
\[ \Eor = L^p_\rho (I; L^p (\Omega)).   \]
Note that, for the boundary class, we are dealing with Neumann boundary conditions, therefore it is expected that the spaces have one degree less of regularity with respect to Dirichlet trace spaces.
We recall that, regarding these spaces, in \cite{MS2012_spaces} the following embedding results were proved: 
\begin{equation}
	\label{E2rho:emb}
	\Edr \hookrightarrow \mathcal{C}^0(\overline{I}, W^{2\left( \rho - \frac{1}{p} \right), \, p} (\Omega)) \hookrightarrow \mathcal{C}^0(\overline{I}, \mathcal{C}^1(\overline{\Omega})), 
\end{equation}
where the last embedding holds if and only if $2 \left( \rho - \frac{1}{p} \right) > 1 + \frac{N}{p}$. Moreover, if $2 \left( \rho - \frac{1}{p} \right) > 1 + \frac{1}{p}$, we also have that 
\begin{equation}
	\label{Frho:emb}
	\Fr \hookrightarrow \mathcal{C}^0(\overline{I}, W^{2\left( \rho - \frac{1}{p} \right)-1-\frac{1}{p}, \, p} (\partial \Omega)) \hookrightarrow \mathcal{C}^0(\overline{I}, \mathcal{C}^0(\partial \Omega)), 
\end{equation}
where again the last embedding holds if and only if $2 \left( \rho - \frac{1}{p} \right) > 1 + \frac{N}{p}$.

Now, it is convenient to rewrite our system \eqref{eq:phi}-\eqref{ic} in the following abstract form:
\begin{alignat}{2}
	\label{eq:abs}
	& \partial_t \vpsi(x,t) + \mathcal{A}\left( x, t, \vpsi(x,t) \right) = \vec{v}(x,t) \qquad && \text{in } Q_T, \nonumber \\
	& \mathcal{B} \left( x, \vpsi(x,t) \right) = \vec{0} \qquad && \text{on } \Sigma_T, \\
	&  \vpsi(x,0) = \vpsi_0(x) \qquad && \text{in } \Omega, \nonumber
\end{alignat}
where
\[ \vpsi(x,t) = \binom{\phi(x,t)}{\sigma(x,t)}, \quad \vpsi_0(x) = \binom{\phi_0(x)}{\sigma_0(x)}, \quad  \mathcal{A}\left(x, t, \vpsi \right) = \binom{\mathcal{A}_1(x,t,\phi,\sigma)}{\mathcal{A}_2(x,\phi,\sigma)}, \]
with components given by
\begin{align*}
	& \mathcal{A}_1(x,t,\phi,\sigma) = - \div ((AF''(\phi) + Ba(x)) \nabla \phi)  + \chi \Delta \sigma - \div (B \nabla a(x) \, \phi - B \nabla J \ast \phi)\\
	& \qquad \qquad \qquad \quad - P(\phi)(\sigma + \chi (1-\phi) - A F'(\phi) - B a(x) \,\phi + B J \ast \phi + \chi \sigma) + \hh(\phi) u(x,t), \\	
	& \mathcal{A}_2(x,\phi,\sigma) = - \Delta \sigma + \chi\Delta \phi \\ 
	& \qquad \qquad \qquad \quad  + P(\phi)(\sigma + \chi (1-\phi) - A F'(\phi) - B a(x) \, \phi + B J \ast \phi + \chi\sigma) 
\end{align*}
and
	\[  \mathcal{B} \left( x, \vpsi \right) = 
	\binom{ \nabla \phi_{\mid \partial \Omega} \cdot \n + l(x,\phi_{\mid \partial \Omega}) \left( B \phi_{\mid \partial \Omega} \nabla a(x) \cdot \n - B(\nabla J \ast \phi)_{\mid \partial \Omega} \cdot \n - \chi \nabla \sigma_{\mid \partial \Omega} \cdot \n \right)}
	{\nabla \sigma_{\mid \partial \Omega} \cdot \n - \chi\nabla \phi_{\mid \partial \Omega} \cdot \n}, \]
with
\[ 0 \le l(x,s) := \frac{1}{AF''(s) + Ba(x)} \le \frac{1}{c_0} \quad \text{for a.e. } x \in \Omega \text{ and any } s\in \R, \]
and finally 
\[ \quad \vec{v}(x,t) = \binom{0}{v(x,t)}.  \]

\begin{remark}
	\label{xt:dependence}
	We want to stress that our operators $\mathcal{A}$ and $\mathcal{B}$ depend also on $x$ and $t$, differently from what was done in \cite{M2012} and \cite{FGG2021}. However, the dependence on $x$ is only through the function $a \in W^{1,\infty}(\Omega)$ (actually, we will have $a \in W^{2,q}(\Omega)$, for any $q \gs 1$, with additional hypotheses on $J$; see Remark \ref{admissible} in the next subsection). The dependence on $t$, instead, is only through the function $u(x,t) \in \Lqt\infty$, which appears only as a lower order term. Moreover, we also have a source term, which we will assume to be $\vec{v} \in L^\infty(Q_T)^2$. Consequently, the high regularity of these terms allows us to proceed with similar arguments to the ones used in the cited papers, up to some adjustments.
\end{remark}

\begin{remark}
	\label{function:l}
	Observe that, by hypothesis \ref{ass:fc0}, the function $l(x,s)$ is well-defined for any $s\in\R$ and for a.e. $x \in \Omega$. Moreover, if $F \in \mathcal{C}^4(\R)$, we also have that 
	\[ l(x, \cdot) \in \mathcal{C}^2(\R) \quad \text{for a.e. } x \in \Omega. \] 
	We additionally remark that we chose to rewrite the first boundary condition in \eqref{bc} by using the function $l(x, \phi_{\mid \partial \Omega})$ in order to highlight the structure of an inhomogeneous Robin boundary condition and to keep consistency with the notation used in \cite{FGG2021}.
\end{remark}

Finally, we introduce the sought regularity for the solutions of \eqref{eq:abs}:

\begin{definition}
	\label{max:sols}
	Assume that the initial data belong to the following space:
	\[ \vpsi_0 \in M^{s,p} := \{ \vpsi \in (W^{s,p}(\Omega))^2 \mid \mathcal{B}(x, \vpsi) = 0 \text{ a.e. on } \Omega \}, \]
	with
	\[ p \in (N+2, +\infty), \quad \rho \in \left( \frac{1}{2} + \frac{N+2}{2p}, 1 \right], \quad \text{ and } s = 2 \left( \rho - \frac{1}{p} \right) > 1 + \frac{N}{p}. \]
	We say that $\vpsi$ is a \emph{maximal solution} to \eqref{eq:abs} on the interval $I=(0,T)$ if it satisfies \eqref{eq:abs} almost everywhere in $Q_T$ and 
	\[ \vpsi \in (\Edr)^2 \cap \mathcal{C}^0([0,T); M^{s,p}). \] 
\end{definition}

\begin{remark}
	\label{rmk:zerotrace}
	One could also be interested in studying existence of maximal solutions for all times, i.e. for $T=+\infty$. Indeed, this can easily be included in our theory by considering weighted spaces with zero temporal trace in $t=0$. More details about this matter can be found in the already cited papers \cite{M2012, MS2012_maxreg, FGG2021}. In particular, in \cite{M2012} it was shown that, for this kind of zero temporal trace spaces, the constants of the embeddings \ref{E2rho:emb} and \ref{Frho:emb} are independent of $T>0$. However, for simplicity we will stick with finite-time evolution.
\end{remark}

\subsection{Local-in-time existence}

The aim of this subsection is to establish existence and uniqueness of maximal solutions to \eqref{eq:abs}, in the sense of Definition \ref{max:sols}, at least locally in time, by adapting the results proved in \cite{MS2012_maxreg, M2012}.
In order to prove this, we shall use maximal $L^p_\rho$-regularity results for the linearised problem associated to \eqref{eq:abs} and then apply the Banach contraction principle to get the same regularity also in the non-linear case. We observe that, as it is said in \cite[Remark 3.6]{M2012}, as long as the operators $\mathcal{A}$ and $\mathcal{B}$ are of class $\mathcal{C}^1$ and a version of the maximal $L^p_\rho$-regularity holds for the corresponding linearised problem, the proof of local existence and uniqueness is actually independent of the concrete form of the operators. Therefore, by looking at \cite[Theorem 2.1]{MS2012_maxreg} for maximal regularity for linear parabolic systems, we can include also our $(x,t)$-dependent operators in this setting.

We now need to assume stronger hypotheses on the parameters of our system \eqref{eq:phi}-\eqref{ic}, on top of the previous ones. In particular, we would need to assume that $J \in  W^{2,1}_{\text{loc}}(\R^N)$, but this hypothesis is incompatible with widely used convolution kernels, such as those of Newton or Bessel type. 
However, following  \cite[Definition 1]{BRB2011}, we can still introduce a suitable class of kernels, which includes the ones mentioned before and satisfies our needs. 
Indeed, we recall the following definition:
\begin{definition}
	\label{def:admissible}
	A convolution kernel $J \in W^{1,1}_{\text{loc}}(\R^N)$ is \emph{admissible} if it satisfies the following conditions: 
	\begin{itemize}
		\item $J \in \mathcal{C}^3(\R^N \setminus \{0\})$.
		\item $J$ is radially symmetric and non-increasing, i.e. $J(\cdot) = \tilde{J}(\abs{\cdot})$ for a non-increasing function $\tilde{J} : \R_+ \to \R$.
		\item There exists $R_0$ such that $r \mapsto \tilde{J}''(r)$ and $r \mapsto \tilde{J}'(r)/r$ are monotone on $(0,R_0)$.
		\item There exists $C_N>0$ such that $\abs{D^3 J(x)} \le C_N \abs{x}^{-N-1}$ for any $x \in \R^3 \setminus \{0\}$.
	\end{itemize}
\end{definition}
\noindent Now, we can assume the following: 
\begin{enumerate}[font = \bfseries, label = B\arabic*., ref = \bf{B\arabic*}]
	\item\label{ass:j2} $J \in W^{2,1}_{\text{loc}}(\R^N)$ or $J$ is \emph{admissible} in the sense of Definition \ref{def:admissible}.
	\item\label{ass:f2} $F \in \mathcal{C}^4(\R)$.
	\item\label{ass:p2} $P \in \mathcal{C}^1(\R)$.
	\item\label{ass:h2} $\hh \in \mathcal{C}^1(\R)$.
	\item\label{ass:u2} $u, v \in L^\infty(Q_T)$.
	\item\label{ass:initial2} $\vpsi_0 = \displaystyle\binom{\phi_0}{\sigma_0} \in M^{s,p}$ as in Definition \mbox{\ref{max:sols}\strut}.
\end{enumerate}

\begin{remark}
	\label{admissible}
	We recall that if $J$ satisfies \ref{ass:j2}, then, by \cite[Lemma 2]{BRB2011}, for any $p \in (1,+\infty)$ there exists a constant $b_p>0$ such that:
	\[ \norm{\nabla (\nabla J \ast f)}_{L^p(\Omega)^{3\times 3}} \le b_p \norm{f}_{L^p(\Omega)} \quad \forall f \in L^p(\Omega). \]
	Moreover, by Young's inequality for convolutions, this implies that $a \in W^{2,p}(\Omega)$ for any $p \in (1, +\infty)$. In this way we can include typical choices for $J$, such as Newton or Bessel type potentials, which satisfy Definition \ref{def:admissible}.
\end{remark}

Before proving the main result, we state and prove two technical lemmas about the regularity of the operators $\mathcal{A}$ and $\mathcal{B}$. 

\begin{lemma}
\label{a:c1}
	Assume hypotheses \emph{\ref{ass:coeff}-\ref{ass:h}} and \emph{\ref{ass:j2}-\ref{ass:initial2}}. Set $I = (0,T)$, $T>0$, and let $p\in (N+2, +\infty)$ and $\rho \in \left( \frac{1}{2} + \frac{N+2}{2p}, 1 \right]$. 
	
	Then for a.e. $(x,t) \in Q_T$
	\[ \mathcal{A}(x,t,\cdot) \in \mathcal{C}^1( \Edr^2 ; \Eor^2 ), \]
	and, for $\vpsi \in \Edr^2$, we have that for any $\hv = (h, k)^\top \in \Edr^2$ and for a.e. $(x,t) \in Q_T$
	\[ \mathcal{A}'(x, t, \vpsi) \hv = \binom{\mathcal{A}'_1(x,t,\phi,\sigma) \hv}{\mathcal{A}'_2(x,\phi,\sigma) \hv},\] 
	where:
	\begin{align*}
		& \mathcal{A}'_1(x,t,\phi,\sigma) \hv = 
		- \div \left( (AF''(\phi) + Ba(x)) \nabla h + A F'''(\phi) \nabla \phi \, h \right) 
		+ \chi\Delta k \\ 
		& \qquad - \div (B \nabla a(x) \, h - B \nabla J \ast h) - P'(\phi) h \sigma - P(\phi) k - \chi (1-\phi) P'(\phi) h \\ 
		& \qquad + \chi P(\phi) h + A P'(\phi) F'(\phi) h + A P(\phi) F''(\phi) h + Ba(x) P(\phi) h + Ba(x) \, P'(\phi) \phi \, h \\ 
		& \qquad - B P(\phi) (J \ast h) - B P'(\phi) (J \ast \phi) h - \chi P'(\phi) h \sigma  - \chi P(\phi) k + \hh'(\phi) h \, u(x,t), \\
		& \mathcal{A}'_2(x,\phi,\sigma) \hv = - \Delta k  + \chi\Delta h + P'(\phi) h \sigma + P(\phi) k + \chi (1-\phi) P'(\phi) h - \chi P(\phi) h \\
		& \qquad - A P'(\phi) F'(\phi) h - A P(\phi) F''(\phi) h - Ba(x) P(\phi) h - Ba(x) \, P'(\phi) \phi \, h \\  
		& \qquad + B P(\phi) (J \ast h) + B P'(\phi) (J \ast \phi) h + \chi P'(\phi) h \sigma  + \chi P(\phi) k.
	\end{align*}
	In particular, for $R>0$ given, there exists a continuous function $\eps: [0,+\infty) \to [0, +\infty)$, $\eps(0) = 0$, such that 
	\begin{equation}
		\label{differentiability:a}
		\norm{ \mathcal{A}(\cdot, \cdot,\vpsi + \hv) - \mathcal{A}(\cdot, \cdot, \vpsi) - \mathcal{A}'(\cdot, \cdot, \vpsi)\hv }_{\Eor^2} \le \eps( \norm{\hv}_{\Edr^2} ) \norm{\hv}_{\Edr^2},
	\end{equation}
	for any $\vpsi, \hv \in \Edr^2$ such that
	\begin{equation}
		\label{normbound:R}
		\norm{\vpsi}_{\mathcal{C}^0(\overline{I}, \mathcal{C}^1(\overline{\Omega}))^2}, \norm{\vpsi}_{\Edr^2}, \norm{\hv}_{\Edr^2} \le R.
	\end{equation}
\end{lemma}

\begin{proof}
	We take inspiration from \cite[Lemma 3.2]{FGG2021}. 
	During the course of the proof, we will extensively use the following two facts:
	\begin{itemize}
		\item $\Wx {1,p}$ is a \emph{Banach algebra} if $p \gs N$; in particular the following property holds: 
		\[ \norm{f g}_{\Wx {1,p}} \le \norm{f}_{\Wx {1,p}} \norm{g}_{\Lx \infty} + \norm{f}_{\Lx \infty} \norm{g}_{\Wx {1,p}} \quad \forall f,g \in \Wx {1,p}. \] 
		\item If $g: \R^m \to \R^n$ is a $\mathcal{C}^1$ function, then $g: \Cqt0^m \to \Cqt0^n$, $m,n \in \N$, is $\mathcal{C}^1$ as a Nemytskii operator. In particular, for any $R \gs 0$ there exists $\eps : [0,+\infty) \to [0, +\infty)$ non-decreasing, continuous and with $\eps(0)=0$ such that for any $\vpsi, \hv \in \Cqt0^m$ 
		\begin{equation*}
			\norm{g(\vpsi + \hv) - g(\vpsi) - Dg(\vpsi) \hv}_{\Cqt0^n} \le \eps( \norm{\hv}_{\Cqt0^m} ) \norm{\hv}_{\Cqt0^m}.
		\end{equation*}
		When applying this property, we always write $\eps$, even if the actual function may change from line to line. In particular, we generally use this inequality in conjunction with the embedding $\Edr \hookrightarrow \Cqt0$, that is: 
		\begin{equation}
			\label{eps:funct}
			\norm{g(\vpsi + \hv) - g(\vpsi) - Dg(\vpsi) \hv}_{\Cqt0^n} \le \eps( \norm{\hv}_{\Edr^m} ) \norm{\hv}_{\Edr^m}.
		\end{equation}
	\end{itemize}

	Before starting, observe that, since $\phi \in \Edr \hookrightarrow \Cqt0$ and $F, P, \hh$ are regular functions by \ref{ass:f2}, \ref{ass:p2} and \ref{ass:h2}, all the terms involving derivatives of $F$, $P$, $\hh$ and their products, evaluated at $\phi$, are bounded in $\Cqt0$. Therefore, it is easy to see that $\mathcal{A}(x,t,\cdot): \Edr^2 \to \Eor^2$ is well-defined, by also using hypothesis \ref{ass:j2} and Remark \ref{admissible}. Moreover, it is clear that in order to show that $\mathcal{A}(x,t,\cdot)$ is $\mathcal{C}^1$, we just need to show the differentiability estimate \eqref{differentiability:a} and that $\mathcal{A}'(x,t,\cdot)$ is continuous. 
	
	So we start by proving \eqref{differentiability:a}. By noticing that for linear terms the following difference is identically equal to $0$, we are left to estimate:  
	{\allowdisplaybreaks
	\begin{align*}
		& \norm{ \mathcal{A}(\cdot, \cdot, \vpsi + \hv) - \mathcal{A}(\cdot, \cdot, \vpsi) - \mathcal{A}'(\cdot, \cdot, \vpsi)\hv }_{\Eor^2} \\
		& \le \norm{ \div ( A F''(\phi + h) \nabla(\phi + h) ) - \div( AF''(\phi) \nabla \phi ) - \div(AF''(\phi) \nabla h + A F'''(\phi) \nabla \phi \, h) }_{\Eor} \\
		& \quad + 2 \norm{(1+\chi)P(\phi + h)(\sigma + k) - (1+\chi)P(\phi)\sigma - (1+\chi) P'(\phi)\sigma \, h - (1+\chi) P(\phi) k }_{\Eor} \\ 
		& \quad + 2 \norm{\chi P(\phi + h)(1 - \phi - h) - \chi P(\phi)(1-\phi) - \chi P'(\phi) (1-\phi) h - \chi P(\phi) h }_{\Eor} \\
		& \quad + 2 \norm{ A P(\phi + h) F'(\phi + h) - A P(\phi) F'(\phi) - AP'(\phi)F'(\phi) h - A P(\phi) F''(\phi) h }_{\Eor} \\
		& \quad + 2 \norm{ Ba P(\phi + h) (\phi + h) - Ba P(\phi) \phi - Ba P(\phi) h - Ba P'(\phi) \phi h }_{\Eor} \\
		& \quad + 2 \norm{ B P(\phi + h) (J \ast (\phi + h)) - BP(\phi)(J \ast \phi) - BP(\phi)(J \ast h) - BP'(\phi)(J \ast \phi) h }_{\Eor} \\
		& \quad + \norm{\hh(\phi + h) u - \hh(\phi) u - \hh'(\phi) h \, u}_{\Eor} \\
		& := I_1 + 2I_2 + 2I_3 + 2I_4 + 2I_5 + 2I_6 + I_7.
	\end{align*} }
	We now proceed term by term. For the first one, we explicitly compute the divergence operator and rearrange the terms, so that 
	\begin{align*}
		I_1 & \le \norm{ AF'''(\phi + h) \nabla (\phi + h) \cdot \nabla (\phi + h) - A F'''(\phi) \nabla \phi \cdot \nabla \phi \\
		& \qquad - 2 A F'''(\phi) \nabla \phi \cdot \nabla h - A F^{(4)}(\phi) \nabla \phi \cdot \nabla \phi h }_{\Eor} \\
		& \quad + \norm{ A F''(\phi + h) \Delta (\phi + h) - A F''(\phi) \Delta \phi - A F''(\phi) \Delta h - A F'''(\phi) \Delta \phi h }_{\Eor} \\
		& \le \norm{ A (F'''(\phi + h ) - F'''(\phi) - F^{(4)}(\phi) h ) \nabla \phi \cdot \nabla \phi }_{\Eor} \\
		& \quad + \norm{ 2A (F'''(\phi + h) - F'''(\phi)) \nabla \phi \cdot \nabla h }_{\Eor} 
						+ \norm{ A F'''(\phi + h ) \nabla h \cdot \nabla h }_{\Eor} \\ 
		& \quad + \norm{ A ( F''(\phi + h) - F''(\phi) - F'''(\phi) h ) \Delta \phi }_{\Eor} 
						+ \norm{ A (F''(\phi + h) - F''(\phi)) \Delta h }_{\Eor} .
	\end{align*}
	At this point, we exploit estimate \eqref{eps:funct}, hypothesis \ref{ass:f2}, which implies that $F''$ and $F'''$ are locally Lipschitz and the fact that $\phi, h \in \Edr$, together with the embeddings \eqref{E2rho:emb}, to infer that
	\begin{align*}
		I_1 & \le A \norm{ F'''(\phi + h ) - F'''(\phi) - F^{(4)}(\phi) h }_{\Cqt0} \norm{ \phi }_{\CzCu} \norm{ \nabla \phi }_{\Eor} \\
		& \quad + 2A \norm{ F'''(\phi + h) - F'''(\phi) }_{\Cqt0} \norm{ \phi }_{\CzCu} \norm{\nabla h }_{\Eor} \\
		& \quad	 + A \norm{ F'''(\phi + h) }_\infty \norm{h}_{\CzCu} \norm{\nabla h}_{\Eor} \\
		& \quad + A  \norm{ F''(\phi + h ) - F''(\phi) - F'''(\phi) h }_{\Cqt0} \norm{ \Delta \phi}_{\Eor} \\
		& \quad	 + A \norm{ F''(\phi + h) - F''(\phi) }_{\Cqt0} \norm{\Delta h}_{\Eor} \\
		& \le ( \norm{\phi}_{\Edr} + \norm{\phi}^2_{\Edr} ) \, \eps( \norm{h}_{\Edr} ) \norm{h}_{\Edr} + C ( 1 + \norm{\phi}_{\Edr}) \norm{h}^2_{\Edr}.
	\end{align*}
	For the second term, since $P \in C^1$ and therefore locally Lipschitz, we similarly deduce that
	\begin{align*}
		I_2 & \le (1+\chi) \norm{P(\phi + h) - P(\phi) - P'(\phi) h}_{\Cqt0} \norm{\sigma}_{\Eor} \\ 
		& \quad + (1+\chi) \norm{P(\phi + h) - P(\phi)}_{\Cqt0} \norm{k}_{\Eor} \\
		& \le \norm{\sigma}_{\Eor} \, \eps( \norm{h}_{\Edr} ) \norm{h}_{\Edr} + C \norm{h}_{\Edr} \norm{k}_{\Eor}.
	\end{align*}
	As one can readily see, under our hypotheses the terms $I_3$, $I_4$ and $I_5$ can be treated in the same way to get:
	\[ I_3 + I_4 + I_5 \le C (1 + \norm{\phi}_{\Eor}) \, \eps( \norm{h}_{\Edr} ) \norm{h}_{\Edr} + C \norm{h}^2_{\Edr},  \]
	where we used \ref{ass:j} for the terms involving $a$. 
	Moreover, by using the local Lipschitz continuity of $P$, \eqref{eps:funct}, \ref{ass:j} and Young's inequality for convolutions, we similarly deduce that:
	\begin{align*}
		I_6 & \le B \norm{P(\phi + h) - P(\phi) - P'(\phi) h}_{\Cqt0} \norm{J \ast \phi}_{\Eor} \\ 
		& \quad + B \norm{P(\phi + h) - P(\phi)}_{\Cqt0} \norm{J \ast h}_{\Eor} \\
		& \le C \norm{\phi}_{\Eor} \, \eps( \norm{h}_{\Edr} ) \norm{h}_{\Edr} + C \norm{h}_{\Edr} \norm{h}_{\Eor}.
	\end{align*}
	Finally, by using \eqref{eps:funct} and hypotheses \ref{ass:h2} and \ref{ass:u2}, we can also see that:
	\[ I_7 \le \norm{u}_{\Lqt\infty} \norm{\hh(\phi + h) - \hh(\phi) - \hh'(\phi) h}_{\Eor} \le C \eps( \norm{h}_{\Edr} ) \norm{h}_{\Edr}. \]
	Then, by putting all together and using \eqref{normbound:R}, as well as the trivial embedding $\Edr \hookrightarrow \Eor$, we obtain that for any $\vpsi, \hv \in \Edr^2$
	\begin{align*}
		& \norm{ \mathcal{A}(\cdot, \cdot, \vpsi + \hv) - \mathcal{A}(\cdot, \cdot, \vpsi) - \mathcal{A}'(\cdot, \cdot, \vpsi)\hv }_{\Eor^2} \\
		& \quad \le (1 + \norm{\vpsi}_{\Edr^2} + \norm{\vpsi}^2_{\Edr^2} ) \, \eps( \norm{\hv}_{\Edr^2} ) \norm{\hv}_{\Edr^2} + C(1 + \norm{\vpsi}_{\Edr^2} ) \norm{\hv}^2_{\Edr^2} \\
		& \quad \le C_R \norm{\hv}_{\Edr^2} \left( \eps(\norm{\hv}_{\Edr^2}) + \norm{\hv}_{\Edr^2} \right),
	\end{align*}
	which implies the Fr\'echet-differentiability of $\Acal(x, t, \cdot) : \Edr^2 \to \Eor^2$, as $\norm{\hv}_{\Edr^2} \to 0$ for a.e.~$(x,t) \in Q_T$. 
	
	Now, to show that actually $\Acal(x, t, \cdot) \in \mathcal{C}^1(\Edr^2, \Eor^2)$ for a.e.~$(x,t) \in Q_T$, it remains to prove that $\Acal'(x, t, \cdot): \Edr^2 \to \mathcal{L}(\Edr^2, \Eor^2)$ is continuous. Indeed, given $\vpsi_1, \vpsi_2 \in \Edr^2$, for any $\hv \in \Edr^2$, we estimate:
	\begin{align*}
		& \norm{ \Acal'(\cdot, \cdot, \vpsi_1) \hv - \Acal'(\cdot, \cdot, \vpsi_2) \hv }_{\Eor^2} \\
		& \le \norm{ \div ( AF''(\phi_1) \nabla h - A F''(\phi_2) \nabla h ) }_{\Eor} 
		+ \norm{ \div ( AF'''(\phi_1) \nabla \phi_1 h - AF'''(\phi_2) \nabla \phi_2 h ) }_{\Eor} \\
		& \quad + 2 \norm{ (1+\chi) (P'(\phi_1) \sigma_1 h - P'(\phi_2) \sigma_2 h) }_{\Eor}  
		+ 2 \norm{ (1+\chi) ( P(\phi_1) k - P(\phi_2) k ) }_{\Eor} \\ 
		& \quad + 2 \norm{ \chi (1-\phi_1) P'(\phi_1) h - \chi (1 - \phi_2) P'(\phi_2) h }_{\Eor} 
		+ 2 \norm{ \chi P(\phi_1) h - \chi P(\phi_2) h }_{\Eor} \\ 
		& \quad + 2 \norm{ A (PF')'(\phi_1) h - A (PF')'(\phi_2) h }_{\Eor} 
		+ 2 \norm{ Ba P(\phi_1) - Ba P(\phi_2) }_{\Eor} \\
		& \quad + 2 \norm{ Ba P'(\phi_1) \phi_1 h - Ba P'(\phi_2) \phi_2 h }_{\Eor} 
		+ 2 \norm{ B P(\phi_1) (J \ast h) - B P(\phi_2) (J \ast h) }_{\Eor} \\
		& \quad + 2 \norm{ B P'(\phi_1) (J \ast \phi_1) h - B P'(\phi_2) (J \ast \phi_2) h }_{\Eor} + \norm{u}_{\Lqt\infty} \norm{\hh'(\phi_1) h - \hh'(\phi_2) h}_{\Eor} \\
		& := I_1 + I_2 + 2 I_3 + 2 I_4 + 2 I_5 + 2 I_6 + 2 I_7 + 2 I_8 + 2 I _9 + 2 I_{10} + 2 I_{11} + I_{12},
	\end{align*}
	where we note again that all the linear terms simplify when taking the difference. We argue term by term as before. Starting from the first one, by computing the divergence, adding and subtracting some terms and using \ref{ass:f2} and \eqref{E2rho:emb}, we infer that 
	\begin{align*}
		I_1 & \le A \norm{ F'''(\phi_1) (\nabla \phi_1 - \nabla \phi_2) \cdot \nabla h }_{\Eor} 
		+ A \norm{ (F'''(\phi_1) - F'''(\phi_2)) \nabla \phi_2 \cdot \nabla h }_{\Eor} \\
		& \quad + A \norm{ (F''(\phi_1) - F''(\phi_2)) \Delta h }_{\Eor} \\ 
		& \le A \norm{F'''(\phi_1)}_{\Cqt0} \norm{\nabla \phi_1 - \nabla \phi_2}_{\Eor} \norm{h}_{\CzCu} \\
		& \quad + A \norm{F'''(\phi_1) - F'''(\phi_2)}_{\Cqt0} \norm{\nabla \phi_2}_{\Eor} \norm{h}_{\CzCu} \\
		& \quad + A \norm{ F''(\phi_1) - F''(\phi_2) }_{\Cqt0} \norm{\Delta h}_{\Eor} \\
		& \le C ( 1 + \norm{\phi_2}_{\Edr} ) \norm{h}_{\Edr} \norm{\phi_1 - \phi_2}_{\Edr}.
	\end{align*}
	For $I_2$, we compute again the divergence and then, up to adding and subtracting some terms and using again \ref{ass:f2} and \eqref{E2rho:emb}, we get: 
	{\allowdisplaybreaks
	\begin{align*}
		I_2 & \le A \Big( \norm{ \Fiv(\phi_1) \nabla \phi_1 \cdot (\nabla \phi_1 - \nabla \phi_2) h }_{\Eor} 
		+ \norm{\Fiv(\phi_2) \nabla \phi_2 \cdot (\nabla \phi_1 - \nabla \phi_2) h }_{\Eor} \\
		& \quad + \norm{ (\Fiv(\phi_1) - \Fiv(\phi_2)) \nabla \phi_1 \cdot \nabla \phi_2 \, h }_{\Eor} 
		+ \norm{ F'''(\phi_1) (\Delta \phi_1 - \Delta \phi_2) h }_{\Eor} \\ 
		& \quad + \norm{ (F'''(\phi_1) - F'''(\phi_2)) \Delta \phi_2 h }_{\Eor} 
		+ \norm{(F'''(\phi_1) - F'''(\phi_2)) \nabla \phi_1  \cdot \nabla h}_{\Eor} \\
		& \quad + \norm{F'''(\phi_2) (\nabla \phi_1 - \nabla \phi_2) \cdot \nabla h}_{\Eor} \Big) \\
		& \le A \norm{h}_{\Cqt0} \Big( \norm{\Fiv(\phi_1)}_{\Cqt0} \norm{\phi_1}_{\CzCu} \norm{\nabla \phi_1 - \nabla \phi_2}_{\Eor} \\
		& \quad + \norm{\Fiv(\phi_2)}_{\Cqt0} \norm{\phi_2}_{\CzCu} \norm{\nabla \phi_1 - \nabla \phi_2}_{\Eor} \\
		& \quad + \norm{\Fiv(\phi_1) - \Fiv(\phi_2)}_{\Cqt0} \norm{\phi_1}_{\CzCu} \norm{\nabla \phi_2}_{\Eor} \\
		& \quad + \norm{F'''(\phi_1)}_{\Cqt0} \norm{\Delta \phi_1 - \Delta \phi_2}_{\Eor} 
		+ \norm{F'''(\phi_1) - F'''(\phi_2)}_{\Cqt0} \norm{\Delta \phi_2}_{\Eor} \Big) \\ 
		& \quad + \norm{h}_{\CzCu} \norm{F'''(\phi_1) - F'''(\phi_2)}_{\Cqt0} \norm{\nabla \phi_1}_{\Eor} \\
		& \quad + \norm{h}_{\CzCu}  \norm{F''(\phi_2)}_{\Cqt0} \norm{\nabla \phi_1 - \nabla \phi_2}_{\Eor} \\
		& \le C \left( 1 + \norm{\phi_1}_{\Edr} + \norm{\phi_2}_{\Edr} \right) \norm{h}_{\Edr} \norm{\phi_1 - \phi_2}_{\Edr} \\
		& \quad + C \norm{\phi_1}_{\Edr} \norm{\phi_2}_{\Edr} \norm{\Fiv(\phi_1) - \Fiv(\phi_2)}_{\Cqt0} \norm{h}_{\Edr}.
	\end{align*}
	}
	By similar strategies and by using \ref{ass:p2}, we can also deduce that
	\begin{align*}
		I_3 & \le (1+\chi) \norm{h}_{\Cqt0} \left( \norm{P'(\phi_1)}_{\Cqt0} \norm{\sigma_1 - \sigma_2}_{\Eor} + \norm{P'(\phi_1) - P'(\phi_2)}_{\Cqt0} \norm{\sigma_2}_{\Eor} \right) \\
		& \le C \norm{h}_{\Edr} \norm{\sigma_1 - \sigma_2}_{\Edr} + C \norm{\sigma_2}_{\Edr} \norm{h}_{\Edr} \norm{P'(\phi_1) - P'(\phi_2)}_{\Cqt0}. 
	\end{align*}
	In a similar way, under our hypotheses, we can also estimate all the other terms as
	\begin{align*}
		& I_4 + I_5 + I_6 + I_7 + I_8 + I _9 + I_{10} + I_{11} + I_{12} \le \\
		& \le C \left( 1 + \norm{\phi_1}_{\Edr} + \norm{\phi_2}_{\Edr} \right) \norm{h}_{\Edr} \norm{P'(\phi_1) - P'(\phi_2)}_{\Cqt0} \\ 
		& \quad + C \norm{h}_{\Edr} \norm{ (PF')'(\phi_1) - (PF')'(\phi_2) }_{\Cqt0} \\
		& \quad + C \norm{h}_{\Edr} \norm{ \hh'(\phi_1) - \hh'(\phi_2) }_{\Cqt0} + C \norm{h}_{\Edr} \norm{\phi_1 - \phi_2}_{\Edr},
	\end{align*}
	where, for $I_{10}$ and $I_{11}$ we also used Young's inequality for convolutions.
	Then, by putting all together and using \eqref{normbound:R}, we see that 
	\begin{align*}
		& \norm{ \Acal'(\cdot, \vpsi_1) - \Acal'(\cdot, \vpsi_2)}_{ \mathcal{L}(\Edr^2, \Eor^2) } = \sup_{ \norm{\hv}_{\Edr^2} = 1 } \norm{ \Acal'(\cdot, \vpsi_1) \hv - \Acal'(\cdot, \vpsi_2) \hv }_{\Eor^2} \\
		& \quad \le C_R \Big( \norm{ \vpsi_1 - \vpsi_2 }_{\Edr^2} + \norm{\Fiv(\phi_1) - \Fiv(\phi_2)}_{\Cqt0} + \norm{P'(\phi_1) - P'(\phi_2)}_{\Cqt0}  \\
		& \qquad + \norm{ (PF')'(\phi_1) - (PF')'(\phi_2) }_{\Cqt0} + \norm{ \hh'(\phi_1) - \hh'(\phi_2) }_{\Cqt0} \Big) \to 0 
	\end{align*}
	as $\vpsi_1 \to \vpsi_2$ in $\Edr^2$ (and therefore also in $\Cqt0^2$ by \eqref{E2rho:emb}), since by \ref{ass:f2}, \ref{ass:p2} and \ref{ass:h2} we know that $F \in \mathcal{C}^4(\R)$, $P \in \mathcal{C}^1(\R)$ and $\hh \in \mathcal{C}^1(\R)$. Consequently, we have shown that $\Acal(x,t,\cdot) : \Edr \to \Eor$ is Fr\'echet-differentiable with continuity for a.e.~$(x,t)\in Q_T$. This concludes the proof of Lemma \ref{a:c1}.
\end{proof}	

\begin{lemma}
	\label{b:c1}
	Assume hypotheses \emph{\ref{ass:coeff}-\ref{ass:h}} and \emph{\ref{ass:j2}-\ref{ass:initial2}}. Set $I = (0,T)$, $T>0$, and let $p\in (N+2, +\infty)$ and $\rho \in \left( \frac{1}{2} + \frac{N+2}{2p}, 1 \right]$.
	
	Then for a.e. $x \in \Omega$
	\[ \mathcal{B}(x, \cdot) \in \mathcal{C}^1( (\Edr)^2 ; (\Fr)^2 ), \]
	and, for $\vpsi \in (\Edr)^2$, we have that for any $\hv = (h, k)^\top \in (\Edr)^2$
	\[ \mathcal{B}'(x, \vpsi) \hv = \binom{\mathcal{B}'_1(x,\phi,\sigma) \hv}{\mathcal{B}'_2(\phi,\sigma) \hv},\] 
	where:
	\begin{align*}
		& \mathcal{B}'_1(x,\phi,\sigma) \hv = \nabla {h}_{\mid \partial \Omega} \cdot \n + l(x,\phi_{\mid \partial \Omega}) ( B {h}_{\mid \partial \Omega} \nabla a(x) \cdot \n - B (\nabla J \ast {h})_{\mid \partial \Omega} \cdot \n - \chi\nabla {k}_{\mid \partial \Omega} \cdot \n) \\ 
		& \qquad \qquad \qquad \, \, \, + l'(x,\phi_{\mid \partial \Omega}) ( B \phi_{\mid \partial \Omega} \nabla a(x) \cdot \n - B (\nabla J \ast \phi)_{\mid \partial \Omega} \cdot \n  - \chi\nabla \sigma_{\mid \partial \Omega} \cdot \n) {h}_{\mid \partial \Omega}, \\
		& \mathcal{B}'_2(\phi,\sigma) \hv = \nabla {k}_{\mid \partial \Omega} \cdot \n - \chi\nabla {h}_{\mid \partial \Omega} \cdot \n.
	\end{align*}
	In particular, for $R>0$ given, there exists a continuous function $\eps: [0,+\infty) \to [0, +\infty)$, $\eps(0) = 0$, such that
	\[ \norm{ \mathcal{B}(\cdot, \vpsi + \hv) - \mathcal{B}(\cdot, \vpsi) - \mathcal{B}'(\cdot, \vpsi)\hv }_{(\Eor)^2} \le \eps( \norm{\hv}_{(\Edr)^2} ) \norm{\hv}_{(\Edr)^2},  \]
	for any $\vpsi, \hv \in (\Edr)^2$ such that
	\[ \norm{\vpsi}_{(\mathcal{C}(\overline{I}, \mathcal{C}^1(\overline{\Omega})))^2}, \norm{\vpsi}_{(\Edr)^2}, \norm{\hv}_{(\Edr)^2} \le R.  \]
\end{lemma}

\begin{proof}
	We immediately observe that $\Bcal_2$ is linear, so there is nothing to prove regarding the second component. The only difficult part is the first component $\Bcal_1$, but in this case the proof is essentially identical to the one of \cite[Lemma 3.3]{FGG2021}. Indeed, our function $l(x,\cdot)$, even if depending on $x$, has exactly the same properties as theirs, i.e.~it is of class $\mathcal{C}^2$ for a.e.~$x\in\Omega$ by Remark \ref{function:l}. Moreover, the new terms depending on the presence of $a(x)$ and $\chi$ are essentially linear and can be treated in a very similar way to the convolution term, or actually in an easier way. For these reasons, we avoid going into the technical details here and we refer the interested reader to \cite[Lemma 3.3]{FGG2021}. 
\end{proof}	

At this point, we can state and prove our main result about existence and uniqueness of local maximal solutions. 

\begin{theorem}
	\label{local:maxsol}
	Let $s = 2 \left( \rho - \frac{1}{p} \right) > 1 + \frac{N}{p}$, with $\rho \in \left( \mezzo + \frac{N+2}{2p}, 1 \right]$ and $p\in (N+2, +\infty)$. Assume also hypotheses \emph{\ref{ass:coeff}-\ref{ass:h}} and \emph{\ref{ass:j2}-\ref{ass:initial2}}.
	
	Then, for any $\vpsi_0 \in M^{s,p}$, there exists a time $t^+ = t^+ (\vpsi_0, \vec{v}) > 0$ such that \eqref{eq:abs} has a unique maximal solution in the sense of Definition $\ref{max:sols}$ on $I=(0,t^+)$. Moreover, the maximal time-interval of existence $I=[0,t^+)$ is such that either $t^+ = T$ or $t^+ \in (0,T)$ with $\lim_{t\to t^+} \norm{\vpsi(t)}_{(W^{s,p}(\Omega))^2}  = + \infty$. 
\end{theorem}

\begin{proof}
	We consider the linearised problem associated to \eqref{eq:abs} and first show that it enjoys maximal $L^p_\rho$-regularity for any $\vpsi \in (\Edr)^2$. 
	
	Let $\vpsi \in (\Edr)^2$ be given and let $\vec{f} \in (\Eor)^2$, $\vec{g} \in (\Fr)^2$ and $\vec{\upxi}_0 \in M^{s,p}$, then consider the linearised problem: 
	\begin{alignat}{2}
		\label{eq:abslin}
		& \partial_t \vec{\upxi} + \mathcal{A}'\left( x, t, \vpsi \right) \vec{\upxi} = \vec{f} \qquad && \text{in } Q_T, \nonumber \\
		& \mathcal{B}' \left( x, \vpsi \right) \vec{\upxi} = \vec{g} \qquad && \text{on } \Sigma_T, \\
		&  \vec{\upxi}(x,0) = \vec{\upxi}_0(x) \qquad && \text{in } \Omega. \nonumber
	\end{alignat}
	We claim that there exists a unique maximal solution $\vec{\upxi}$ to problem \eqref{eq:abslin} such that
	\begin{equation}
		\label{abslin:est}
		\norm{\vec{\upxi}}_{(\Edr)^2} \le C \left( \norm{\vec{f}}_{(\Eor)^2}  + \norm{\vec{g}}_{(\Fr)^2} + \norm{\vec{\upxi}_0}_{(W^{s,p}(\Omega))^2} \right).
	\end{equation}
	To show this, we want to prove that \eqref{eq:abslin} enjoys maximal $L^p_\rho$-regularity, by applying \cite[Theorem 2.1]{MS2012_maxreg}, therefore we need to check its hypotheses. First, we observe that we can write:
	\[ (\mathcal{A}'(x,t,\vpsi), \mathcal{B}'(x,\vpsi)) = (\tilde{\mathcal{A}}(x,\vpsi), \tilde{\mathcal{B}}(x,\vpsi)) + \text{ lower order terms, } \]
	where the principal parts have the following form:
	\[ \tilde{\mathcal{A}}(x,\vpsi) \hv = - \div \left( \begin{pmatrix}
		AF''(\phi) + Ba(x) & -\chi\\
		-\chi& 1
	\end{pmatrix} \binom{\nabla h}{\nabla k} \right), \]
	\[ \tilde{\mathcal{B}}(x,\vpsi) \hv = \left( \begin{pmatrix}
		1 & -\chi\, l(x,\phi_{\mid \partial \Omega}) \\
		-\chi& 1
	\end{pmatrix} \binom{\nabla h \cdot \n}{\nabla k \cdot \n} \right). \]
	Next, regarding the regularity of the top-order coefficients, i.e. those of the two $2\times 2$ matrices above, we note that they all belong to $\mathcal{BUC}(\overline{I} \times \overline{\Omega}; \R^{2\times 2})$, since $\phi \in \Edr \hookrightarrow \mathcal{C}^0(\overline{I}, \mathcal{C}^1(\overline{\Omega}))$ is fixed and $a \in W^{2,p} \hookrightarrow \mathcal{C}^1(\overline{\Omega})$ for $p \gs N + 2$. Finally, the last conditions we need to check are normal ellipticity for $\tilde{\mathcal{A}}$ and the Lopatinskii-Shapiro type condition for $(\tilde{\mathcal{A}}, \tilde{\mathcal{B}})$. Both conditions follow easily, even in our $x$-dependent case, from \cite[Theorem 4.4 and Example 4.5]{A1990}, thanks to hypothesis \ref{ass:fc0}. Observe that, to apply \cite[Theorem 4.4]{A1990}, it may be useful to rewrite $\tilde{\mathcal{B}}$ by multiplying the first row by $AF''(\phi) + Ba(x) = l(x,\phi)^{-1}$. Then, by \cite[Theorem 2.1]{MS2012_maxreg}, there exists a unique maximal solution $\vec{\xi} \in (\Edr)^2$ to \eqref{eq:abslin} satisfying \eqref{abslin:est}. Note that the time-dependent part of the operator $\mathcal{A}$ is hidden in the lower order terms, therefore it does not hinder the ellipticity properties. Moreover, the statement of \cite[Theorem 2.1]{MS2012_maxreg} also allows time-dependent coefficients, as long as its hypotheses are satisfied.
	
	At this point, since we have a unique maximal solution for the linearised system, we can employ the Banach contraction principle, exactly as in \cite[Lemma 3.2 and Lemma 3.3]{M2012}, to find the required maximal solution $\vpsi$ to \eqref{eq:abs}. As already noted before, this procedure is independent of the concrete form of the operators, as long as they are $\mathcal{C}^1$ and the maximal regularity holds for the linearised system. For the sake of brevity, we just give the main steps here, leaving most of the calculations to the detailed proof of \cite[Proposition 4.3.2]{meyries_phd}. First, given $\vpsi_0 \in M^{s,p}$, we can fix an extension $\vpsi_* \in \Edr$ such that $\vpsi_*(\cdot, 0) = \vpsi_0$, which exists by \cite[Lemma 4.3]{MS2012_spaces}. Next, we consider the linear problem 
	\begin{alignat}{2}
		\label{eq:banach1}
		& \partial_t \vec{\upxi} + \mathcal{A}'\left( x, t, \vpsi_* \right) \vec{\upxi} = \mathcal{A}'\left( x, t, \vpsi_* \right) \vpsi_* - \mathcal{A}(x,t, \vpsi_*) + \vec{v} \qquad && \text{in } Q_T, \nonumber \\
		& \mathcal{B}' \left( x, \vpsi_* \right) \vec{\upxi} = \mathcal{B}'\left( x, \vpsi_* \right) \vpsi_* - \mathcal{B}(x, \vpsi_*) \qquad && \text{on } \Sigma_T, \\
		&  \vec{\upxi}(x,0) = \vpsi_0(x) \qquad && \text{in } \Omega, \nonumber
	\end{alignat}
	and observe that both the right-hand sides $ \mathcal{A}'\left( x, t, \vpsi_* \right) \vpsi_* - \mathcal{A}(x,t, \vpsi_*)  + \vec{v} \in \Eor$ and $  \mathcal{B}' \left( x, \vpsi_* \right) \vec{\upxi} = \mathcal{B}'\left( x, \vpsi_* \right) \vpsi_* - \mathcal{B}(x, \vpsi_*) \in \Fr$ belong to the right spaces, due to Lemmas \ref{a:c1} and \ref{b:c1}, and also, since $\mathcal{B}(\vpsi_0) = 0$, the compatibility condition $\mathcal{B}'(x, \vpsi_0) \vpsi_0 = \mathcal{B}'(x, \vpsi_0) \vpsi_0 - \mathcal{B}(x,\vpsi_0)$ on $\partial \Omega$ holds. Then maximal $L^p_\rho$ regularity for \eqref{eq:abslin}, as above, yields a unique solution $\vec{\upxi}_* \in \Edr$ of \eqref{eq:banach1}.
	Now, using this $\vec{\upxi}_*$, we define for $\lambda, \tau \in (0,1]$ the set 
	\[ \Sigma(\lambda, \tau) := \left\{ \vpsi \in E_{2,\rho}(0,\tau) \mid \norm{\vpsi - \vec{\upxi}_*}_{E_{2,\rho}(0,\tau)} \le \lambda, \, \vpsi(\cdot,0) = \vpsi_0  \right\},  \]
	which is closed in $E_{2,\rho}(0,\tau)$. At this point, for any $\vpsi \in \Sigma(\lambda, \tau)$, we consider the linear system
	\begin{align}
		\label{eq:banach2}
		& \partial_t \vec{\upxi} + \mathcal{A}'\left( x, t, \vpsi_* \right) \vec{\upxi} = \mathcal{A}'\left( x, t, \vpsi_* \right) \vpsi - \mathcal{A}(x,t, \vpsi) + \vec{v} & \text{in } \Omega \times (0,\tau), \nonumber \\
		& \mathcal{B}' \left( x, \vpsi_* \right) \vec{\upxi} = \mathcal{B}'\left( x, \vpsi_* \right) \vpsi - \mathcal{B}(x, \vpsi) & \text{on } \partial \Omega \times (0,\tau), \\
		&  \vec{\upxi}(x,0) = \vpsi_0(x) & \text{in } \Omega. \nonumber
	\end{align}
	Exactly as above, also this linear system admits a unique solution $\vec{\upxi} = \mathcal{S}(\vpsi) \in E_{2,\rho}(0,\tau)$ and this defines a map $\mathcal{S}: \Sigma(\lambda, \tau) \to E_{2,\rho}(0,\tau)$. Clearly, $\vpsi \in \Sigma(\lambda, \tau)$ solves our starting system \eqref{eq:abs} if and only if it is a fixed point of $\mathcal{S}$ in $\Sigma(\lambda, \tau)$.  Since, for given $\lambda$, each solution of \eqref{eq:abs} in $E_{2,\rho}(0,\tau)$ belongs to $\Sigma(\lambda, \tau)$ for sufficiently small $\tau$, our task is thus to show that the map $\mathcal{S}$ has a unique fixed point in $\Sigma(\lambda, \tau)$, provided that $\lambda$ and $\tau$ are sufficiently small. To this end we use the Banach contraction principle.
	Indeed, one has to show that $\mathcal{S}$ maps $\Sigma(\lambda, \tau)$ into itself for small $\lambda$ and $\tau$, as well as that it is a contraction on $\Sigma(\lambda, \tau)$. These two properties can be verified exactly as in second and third steps of the proof of \cite[Proposition 4.3.2]{meyries_phd}, by relying on the inequalities proven in Lemmas \ref{a:c1} and \ref{b:c1} and on properties of zero-temporal-trace spaces (see Remark \ref{rmk:zerotrace}) to keep constants independent of $\tau$.
	The existence of a maximal existence time $t^+ = t^+(\vpsi_0, \vec{v})$ and a maximal solution in $\mathcal{C}([0; t^+);  W^{s,p}(\Omega; \R^N))$ then follows from standard arguments.
\end{proof}

We have shown that, under hypotheses \ref{ass:coeff}-\ref{ass:h} and \ref{ass:j2}--\ref{ass:initial2}, our system \eqref{eq:phi}-\eqref{ic} admits a unique maximal solution with regularity 
\[ \phi, \sigma \in E_{2,\rho} (0,t^+) := W^{1,p}_\rho (0,t^+; L^p(\Omega)) \cap L^p_\rho (0,t^+; W^{2,p}(\Omega)), \] 
with $\rho \in \left( \mezzo + \frac{N+2}{2p}, 1 \right]$ and $p\in (N+2, +\infty)$, up to a certain maximal time $t^+=t^+(\phi_0,\sigma_0)$. In the next section, we will extend this regularity to the whole interval $(0,T)$, however, in order to do this, it will be necessary to put $\chi= 0$. Before going into the details, we need another technical lemma about a $\mathcal{C}^\beta - W^{2-2/p,p}$ smoothing effect for the solutions to our problem, which will be useful in the next section. Note that this lemma can still be proven even if $\chi > 0$.

\begin{lemma}
	\label{smooth:eff}
	Let $\rho =1$, $p \in (N+2, +\infty)$ and $\vpsi_0 \in M^{2-2/p,p}$. Assume also hypotheses \emph{\ref{ass:coeff}-\ref{ass:h}} and \emph{\ref{ass:j2}-\ref{ass:u2}}. 
	
	Let $\vpsi$ be a maximal solution to \eqref{eq:abs} on $(0,t^+)$ in the sense of Definition \emph{\ref{max:sols}} and let $t_1, t_2 \in (0,t^+)$, with $t_2 > t_1$ and $\tau := t_2 - t_1$. Then, for any $\beta \in (0,1)$, there exists a constant $C>0$, depending only on the parameters of the system and $\tau$, $p$ and $\delta_\beta = \norm{\vpsi}_{(\mathcal{C}([t_1, t_2]; \mathcal{C}^\beta(\Omega))^2}$, such that
	\[ \norm{\vpsi(t_2)}_{(W^{2-2/p,p}(\Omega))^2} \le C ( 1 + \norm{\vpsi(t_1)}_{(\mathcal{C}^\beta(\Omega))^2} ). \]
\end{lemma} 

\begin{remark}
	The smoothing effect provided by this lemma will be crucial in establishing global regular solutions and it is actually the reason why we need to consider \emph{weighted} spaces.
	Indeed, the regularisation effect of the time-weight is the crucial ingredient of the argument. This will be clear in the last part of the proof. 
\end{remark}

\begin{proof}
	We mostly follow the argument of \cite[Lemma 3.5]{FGG2021}, which in turn is inspired by the one of \cite[Lemma 4.1]{M2012}. 
	
	Now fix $t_1 \ls t_2 \in (0,\tp)$ and $\tau = t_2 - t_1$. Given $\vpsi$ maximal solution to \eqref{eq:abs}, we define $\vzeta (t) := \vpsi (t + t_1)$ for any $t \in (0, \tau)$. Then, since $\vpsi \in E_{2,1}(0,\tp)$, we have that $\vzeta \in E_{2,1}(0,\tau)$. Moreover, we recall that $\vpsi \in E_{2,1}(0,\tp) \hookrightarrow \mathcal{C}^0(0, \tp; W^{2-2/p, p}(\Omega))$ by \eqref{E2rho:emb} with $\rho = 1$, hence this implies that
		\begin{equation}
		\label{thmsmooth:normt2}
		\norm{\vpsi(t_2)}_{\Wx {2-2/p,p}} = \norm{\vzeta(\tau)}_{\Wx {2-2/p,p}} \le C_\tau \norm{\vzeta}_{E_{2,1}(0,\tau)} \le C_\tau \norm{\vzeta}_{E_{2,\rho}(0,\tau)},
	\end{equation}
	where the last inequality holds for any $\rho \in (1/p, 1]$, due to Remark \ref{rmk:weight}. Note that here we passed from the unweighted case to the weighted one, in order to exploit the time-regularisation effect.
	Now, observe that, by definition, $\vzeta = \binom{\xi}{\theta}$ satisfies the following system of partial differential equations: 
	\begin{align*}
		& \xi_t - (AF''(\xi) + Ba) \Delta \xi + \chi\Delta \theta \\
		& \quad = A F'''(\xi) \nabla \xi \cdot \nabla \xi + 2B \nabla a \cdot \nabla \xi + B \Delta a \xi - \div (B \nabla J \ast \xi) \\
		& \qquad + P(\xi) (\theta + \chi (1-\xi) - AF'(\xi) - Ba \xi + B J \ast \xi + \chi\theta ) - \hh(\xi) u && \hbox{in $Q_\tau$}, \\
		& \theta_t - \Delta \theta + \chi\Delta \xi = - P(\xi) (\theta + \chi (1-\xi) - AF'(\xi) - Ba \xi + B J \ast \xi + \chi\theta ) + v && \hbox{in $Q_\tau$}, \\
		& \nabla \xi_{\trace} \cdot \n = - l(x, \xi_{\trace}) (B \xi_{\trace} \nabla a \cdot \n - B (\nabla J \ast \xi)_{\trace} \cdot \n - \chi \nabla  \theta_{\trace} \cdot \n ) && \hbox{on $\Sigma_\tau$}, \\
		& \nabla \theta_{\trace} = \chi\nabla \xi_{\trace} && \hbox{on $\Sigma_\tau$}, \\
		& \xi(0) = \phi(t_1), \quad \theta(0) = \sigma(t_1) && \hbox{in $\Omega$}.		
	\end{align*}
	At this point, we assume $\vzeta$ fixed and consider the following non-homogeneous \emph{linear} parabolic problem for $\vec{w} = \binom{w}{z}$: 
	\begin{align*}
		& w_t - (AF''(\xi) + Ba) \Delta w + \chi\Delta z \\
		& \quad = A F'''(\xi) \nabla \xi \cdot \nabla \xi + 2B \nabla a \cdot \nabla \xi + B \Delta a \xi  - \div (B \nabla J \ast \xi) \\
		& \qquad + P(\xi) (\theta + \chi (1-\xi) - AF'(\xi) - Ba \xi + B J \ast \xi + \chi\theta ) - \hh(\xi) u && \hbox{in $Q_\tau$}, \\
		& z_t - \Delta z + \chi\Delta w = - P(\xi) (\theta + \chi (1-\xi) - AF'(\xi) - Ba \xi + B J \ast \xi + \chi\theta ) + v && \hbox{in $Q_\tau$}, \\
		& \nabla w_{\trace} \cdot \n - \chi l(x, \xi_{\trace}) \nabla z_{\trace} \cdot \n = - l(x, \xi_{\trace}) (B \xi_{\trace} \nabla a \cdot \n - B (\nabla J \ast \xi)_{\trace} \cdot \n) && \hbox{on $\Sigma_\tau$}, \\
		& \nabla z_{\trace} \cdot \n - \chi \nabla w_{\trace} \cdot \n = 0 && \hbox{on $\Sigma_\tau$}, \\
		& w(0) = \phi(t_1), \quad z(0) = \sigma(t_1) && \hbox{in $\Omega$}, 
	\end{align*}
	where we observe that, since $l(x,\xi)^{-1} = AF''(\xi) + Ba(x) \ge c_0 \gs \chi^2 \ge 0$ by \ref{ass:fc0}, the differential operator 
	\[ - \left( \begin{pmatrix}
		AF''(\xi) + Ba(x) & -\chi\\
		-\chi & 1
	\end{pmatrix} \binom{\Delta w}{\Delta z} \right)  \]
	and the boundary operator
	\[  \left( \begin{pmatrix}
		1 & -\chi\, l(x,\xi_{\trace}) \\
		-\chi & 1
	\end{pmatrix} \binom{\nabla w \cdot \n}{\nabla z \cdot \n} \right) \]
	 are normally elliptic and satisfy the Lopatinskii-Shapiro type condition. Again, both conditions can be verified by applying \cite[Theorem 4.4 and Example 4.5]{A1990}. Moreover, by the embedding \eqref{E2rho:emb}, it also follows that $\norm{l(\cdot, \xi(\cdot))^{-1}}_{\mathcal{BUC}(\bar{Q_T})} \le C$, where $C$ depends on the fact that $\norm{\phi}_{\mathcal{C}^0([t_1,t_2]; \mathcal{C}^0(\bar{\Omega}))} \le \delta_\beta$ for any fixed $\beta \in (0,1)$. Now, we know that $\vec{v}=\vzeta$ is a solution of this problem, then we can apply the linear maximal parabolic regularity result \cite[Theorem 2.1]{MS2012_maxreg} to infer that there exists a constant $C \gs 0$, depending on the parameters of the system, on $\tau$ and $\delta_\beta$, such that the norm of $\vzeta$ is bounded by the corresponding norms of the right-hand side, the boundary data and the initial data, namely:
	\begin{equation}
	\label{thmsmooth:mainest}
	\begin{split}
		& \norm{\vzeta}_{\Edr} \le C \Big( \norm{F'''(\xi) \nabla \xi \cdot \nabla \xi}_{\Eor} + \norm{\nabla a \cdot \nabla \xi}_{\Eor} + \norm{\Delta a \xi}_{\Eor} \\
		& \qquad + \norm{\div (B \nabla J \ast \xi) }_{\Eor} + \norm{\hh(\xi) u}_{\Eor} + \norm{v}_{\Eor}  \\ 
		& \qquad + \norm{P(\xi) (\theta + \chi (1-\xi) - AF'(\xi) - Ba \xi + B J \ast \xi + \chi\theta )}_{\Eor} \Big) \\
		& \quad + C \Big( \norm{l(\cdot, \xi) (B \xi_{\trace} \nabla a \cdot \n)}_{\Fr} + \norm{l(\cdot, \xi)( B (\nabla J \ast \xi)_{\trace} \cdot \n)}_{\Fr} \Big) \\ 
		& \quad + C \norm{\vpsi(t_1)}_{\Wx{2(\rho - 1/p),p}^2},
	\end{split}
	\end{equation}
	where, from here onward, we use the notation $I := (0,\tau)$. 
	Observe that \emph{linear} maximal regularity, as in  \cite[Theorem 2.1]{MS2012_maxreg}, works for any $\rho \in (1/p, 1]$. This is important, since in the end we need some freedom in choosing $\rho$ small enough.
	We are now left to estimate each term on the right-hand side of \eqref{thmsmooth:mainest}.
	
	Here we need a generalisation of Gagliardo-Nirenberg's inequality for fractional Sobolev spaces, that was stated and proved in \cite{A1985}. Indeed, we apply \cite[Proposition 4.1]{A1985} with $\theta = 1/2$, $p = 2q \in (1, +\infty)$, $p_1 = r \in (1,+\infty)$, $p_2 =q$, $s_1 = \alpha$ and $s_2 = \gamma \gs 0$ such that 
	\[ 
	1 - \frac{N}{2q} \ls \mezzo \left( \alpha - \frac{N}{r} \right) + \mezzo \left( \gamma - \frac{N}{q} \right), \]
	to have the following inequality
	\begin{equation}
		\label{ineq:gnfrac_amann}
		\norm{u}_{\Wx{1,2q}} \le C \norm{u}^{1/2}_{\Wx{\alpha,r}} \norm{u}^{1/2}_{\Wx{\gamma,q}},
	\end{equation}
	where we can choose $q \in (1,+\infty)$, $r >> 1$, $\alpha \in (0, \beta)$ and $\gamma \ls 2$. Furthermore, we also need another version of Gagliardo-Nirenberg's inequality for fractional Sobolev spaces, this time from \cite{BM2018}. Indeed, we can apply \cite[Theorem 1]{BM2018} with $s_1 = \alpha \in (0,\beta)$, $s_2 = 2$, $s = \gamma \in (1,2)$, $p = p_1 = p_2 = q$ and $\delta \in (0,1)$ such that $\gamma = \alpha \delta + 2 (1 - \delta)$, to obtain the inequality 
	 \begin{equation}
	 	\label{ineq:gnfrac_brezis}
	 	\norm{u}_{\Wx{\gamma,q}} \le C \norm{u}^{\delta}_{\Wx{\alpha,q}} \norm{u}^{1-\delta}_{\Wx{2,q}}.
	 \end{equation}
	
	Then, we can begin estimating the terms on the right-hand side of \eqref{thmsmooth:mainest}. Starting from the first one, we can infer that
	\begin{align*}
		& \norm{F'''(\xi) \nabla \xi \cdot \nabla \xi}_{\Eor}^p = \int_0^\tau t^{p(1-\rho)} \norm{F'''(\xi) \abs{\nabla \xi}^2}^p_{\Lx p} \, \de t \\
		& \quad \le \norm{F'''(\xi)}^p_{\Cqt0} \int_0^\tau t^{p(1-\rho)} \norm{\nabla \xi}^{2p}_{\Lx{2p}} \, \de t \le \cdb  \int_0^\tau t^{p(1-\rho)}  \norm{\xi}^{2p}_{\Wx{1,2p}} \, \de t. 
	\end{align*}
	Then, to estimate $\norm{\xi}^{2p}_{\Wx{1,2p}}$, we first use \eqref{ineq:gnfrac_amann} with $q=p$, then the embedding $\Cbeta \hookrightarrow \Wx{\alpha,r}$ for $\alpha \in (0,\beta)$ and $r \in (1,+\infty)$ and finally \eqref{ineq:gnfrac_brezis} with $p=q$, together with Young's inequality with exponents $1/(1-\delta)$ and $1/\delta$ and again the embedding $\Cbeta \hookrightarrow \Wx{\alpha,p}$. Indeed, we have that
	\begin{align*}
		\norm{F'''(\xi) \nabla \xi \cdot \nabla \xi}_{\Eor}^p & \le \int_0^\tau t^{p(1-\rho)} \norm{\xi}^{2p}_{\Wx{1,2p}} \, \de t \\
		& \le  \int_0^\tau t^{p(1-\rho)} \norm{\xi}^p_{\Wx{\alpha,r}} \norm{\xi}^p_{\Wx{\gamma,p}} \, \de t \\
		& \le  C \int_0^\tau t^{p(1-\rho)} \norm{\xi}^p_{\Cbeta} \norm{\xi}^p_{\Wx{\gamma,p}} \, \de t \\
		& \le  \cdb \int_0^\tau t^{p(1-\rho)} \norm{\xi}^{p(1-\delta)}_{\Wx{2,p}} \norm{\xi}^{p\delta}_{\Wx{\alpha,p}} \, \de t \\
		& \le \eps \int_0^\tau t^{p(1-\rho)} \norm{\xi}^{p}_{\Wx{2,p}} \, \de t + C_\eps \int_0^\tau t^{p(1-\rho)} \norm{\xi}^{p}_{\Cbeta} \, \de t \\
		& \le \eps \norm{\xi}^p_{\Edr} + C_\eps,
	\end{align*}
	for some $\eps \gs 0$, depending on the parameters of the system, $\tau$, $p$ and $\delta_\beta$, to be chosen later.
	Going on with the second term of \eqref{thmsmooth:mainest}, by using Gagliardo-Nirenberg and Young's inequalities to say that 
	\[ \norm{\nabla \xi}^p_{\Lx p} \le \eps \norm{\xi}^p_{\Wx{2,p}} + C_\eps \norm{\xi}^p_{\Lx p}, \]
	we can estimate
	\begin{align*}
		\norm{\nabla a \cdot \nabla \xi}^p_{\Eor} & \le \norm{\nabla a}^p_\infty \norm{\nabla \xi}^p_{\Eor} \le (b^*)^p \int_0^\tau t^{p(1-\rho)} \norm{\nabla \xi}^p_{\Lx p} \, \de t \\
		& \le \eps \int_0^\tau t^{p(1-\rho)} \norm{\xi}^p_{\Wx{2,p}} \, \de t + C_\eps \int_0^\tau t^{p(1-\rho)} \norm{\xi}^p_{\Lx p} \, \de t \\
		& \le \eps \norm{\xi}^p_{\Edr} + C_\eps \norm{\xi}^p_{\mathcal{C}^0(0,\tau; \Cbeta)} \le \eps \norm{\xi}^p_{\Edr} + C_\eps,
	\end{align*}
	where again $\eps$ is to be chosen later. 
	For the third and fourth term, we use hypothesis \ref{ass:j2}, together with Remark \ref{admissible}, and similar techniques to say that 
	\begin{align*}
		& \norm{\Delta a \, \xi}^p_{\Eor} + \norm{\div(\nabla J \ast \xi)}^p_{\Eor} \\
		& \quad \le \int_0^\tau t^{p(1-\rho)} \norm{\Delta a}^p_{\Lx {2p}} \norm{\xi}^p_{\Lx{2p}} \, \de t + 2 b_p^p \int_0^\tau t^{p(1-\rho)} \norm{\xi}^p_{\Lx p} \, \de t \\ 
		& \quad \le C \int_0^\tau t^{p(1-\rho)} \norm{\xi}^{p}_{\Wx {1,p}} \, \de t 
		\le \eps \int_0^\tau t^{p(1-\rho)} \norm{\xi}^p_{\Wx{2,p}} \, \de t + C_\eps \int_0^\tau t^{p(1-\rho)} \norm{\xi}^p_{\Lx p} \, \de t \\
		& \quad \le \eps \norm{\xi}^p_{\Edr} + C_\eps \norm{\xi}^p_{\mathcal{C}^0(0,\tau; \Cbeta)} \le \eps \norm{\xi}^p_{\Edr} + C_\eps,
	\end{align*}
	where we used that $a \in \Wx{2,2p}$ by Remark \ref{admissible} and the embedding $\Wx{1,p} \hookrightarrow \Lx{2p}$, which surely holds if $p \gs N + 2$.
	Then, we can easily see that for the fifth and sixth terms we have:
	\begin{align*}
		& \norm{\hh(\xi) u}^p_{\Eor} \le \norm{u}^p_{\Lqt\infty} \norm{\hh(\xi)}^p_{\Eor} \le C, \\
		& \norm{u}_{\Eor} \le \norm{u}_{\Lx\infty} \le C,
	\end{align*}
	since $u,v \in \Lqt\infty$ and $\hh \in L^\infty(\R)$.
	Moreover, for the seventh one, by similar methods, we can deduce that 
	{\allowdisplaybreaks
	\begin{align*}
		& \norm{P(\xi) (\theta + \chi (1-\xi) - AF'(\xi) - Ba \xi + B J \ast \xi + \chi\theta )}^p_{\Eor} \\
		& \quad \le \norm{P}_\infty^p \int_0^\tau t^{p(1-\rho)} \norm{\theta + \chi (1-\xi) - AF'(\xi) - Ba \xi + B J \ast \xi + \chi\theta}^p_{\Lx p} \, \de t \\
		& \quad \le C \int_0^\tau t^{p(1-\rho)} \left( \norm{\theta}^p_{\Lx p} + \norm{\xi}^p_{\Lx p} + \norm{F'(\xi)}^p_{\Lx\infty} + 1 \right) \, \de t \\
		& \quad \le C \int_0^\tau t^{p(1-\rho)} \norm{\vzeta}^p_{\Lx p} \, \de t + C_\tau \norm{F'(\xi)}^p_{\Cqt0} + C_\tau
		\le C \int_0^\tau t^{p(1-\rho)} \norm{\vzeta}^p_{\Lx p} \, \de t + \cdb \\
		& \quad \le C \int_0^\tau t^{p(1-\rho)} \norm{\vzeta}^{p/2}_{\Lx p} \norm{\vzeta}^{p/2}_{\Lx p} \, \de t + \cdb
		\le \eps \norm{\vzeta}^p_{\Edr} + C_\eps.
	\end{align*} }
	We are now left only with the boundary terms, for which we first have to recall some results about embeddings for trace spaces. Indeed, by \cite[Theorem 4.5]{MS2012_spaces}, we have that the trace operator $\Tr_\Omega: u \mapsto u_{\trace}$ is continuous between the spaces 
	\[ \Tr_\Omega: W^{1/2, p}_\rho (I, \Lx p) \cap L^p_\rho (I, \Wx{1,p}) \to \Fr, \]
	with an embedding constant depending on $\tau$. In this way, we can estimate the norm on the boundary space $\Fr$ with a norm on a functional space on the whole $\Omega$, i.e. 
	\[ \norm{f}_{\Fr} \le C \norm{f}_{W^{1/2, p}_\rho (I, \Lx p) \cap L^p_\rho (I, \Wx{1,p})}. \]
	Moreover, we recall that the function $l(x, \cdot)$ is in $\mathcal{C}^2(\R)$ for a.e.~$x \in \Omega$, therefore, since $\xi \in \Edr$ is uniformly bounded in $\Lqt\infty$, we have that 
	\[ \norm{l(\cdot, \xi)}_{\Lqt\infty} + \norm{l'(\cdot, \xi)}_{\Lqt\infty} \le \cdb. \]
	Then, we can estimate the eighth and ninth term in \eqref{thmsmooth:mainest} as 
	\begin{equation}
	\label{thmsmooth:boundary}
		\begin{split}
		& \norm{l(\cdot, \xi) (B \xi_{\trace} \nabla a \cdot \n)}^p_{\Fr} + \norm{l(\cdot, \xi)( B (\nabla J \ast \xi)_{\trace} \cdot \n)}^p_{\Fr} \\
		& \quad \le C(\tau) \Big( \norm{l(\cdot, \xi) (\xi \nabla a)}^p_{W^{1/2, p}_\rho (I, \Lx p)} + \norm{l(\cdot, \xi)(\nabla J \ast \xi)}^p_{W^{1/2, p}_\rho (I, \Lx p)} \\ 
		& \qquad \qquad + \norm{l(\cdot, \xi) (\xi \nabla a)}^p_{L^p_\rho (I, \Wx{1,p})} + \norm{l(\cdot, \xi)(\nabla J \ast \xi)}^p_{L^p_\rho (I, \Wx{1,p})} \Big).
		\end{split}
	\end{equation}
	For the first two terms, we use the intrinsic norm of fractional Sobolev spaces and hypothesis \ref{ass:j}, together with the mean value theorem and the bounds on $l(\cdot, \xi)$, to infer that
	\begingroup
	\allowdisplaybreaks
	\begin{align*}
		& \norm{l(\cdot, \xi) (\xi \nabla a)}^p_{W^{1/2, p}_\rho (I, \Lx p)} \\
		& \quad = \int_0^T \int_0^s \frac{t^{p(1-\rho)}}{(s-t)^{1 + (1/2)p}} \norm{ l(\cdot, \xi(s)) \xi(s) \nabla a - l(\cdot, \xi(t)) \xi(t) \nabla a }^p_{\Lx p} \, \de t \, \de s \\
		& \quad \le \int_0^T \int_0^s \frac{t^{p(1-\rho)}}{(s-t)^{1 + (1/2)p}} \norm{ ( l(\cdot, \xi(s)) - l(\cdot, \xi(t)) ) \xi(s) \nabla a}^p_{\Lx p} \, \de t \, \de s \\
		& \qquad + \int_0^T \int_0^s \frac{t^{p(1-\rho)}}{(s-t)^{1 + (1/2)p}} \norm{ l(\cdot, \xi(t)) ( \xi(s) - \xi(t) ) \nabla a }^p_{\Lx p} \, \de t \, \de s \\
		& \quad \le \int_0^T \int_0^s \frac{t^{p(1-\rho)}}{(s-t)^{1 + (1/2)p}} (\bstar)^p \norm{\xi}^p_{\Cbeta} \norm{ ( l(\cdot, \xi(s)) - l(\cdot, \xi(t)) ) }^p_{\Lx p} \, \de t \, \de s \\
		& \qquad + \int_0^T \int_0^s \frac{t^{p(1-\rho)}}{(s-t)^{1 + (1/2)p}} (\bstar)^p \norm{l(\cdot,\xi(t))}^p_{\Lx\infty} \norm{ \xi(s) - \xi(t) }^p_{\Lx p} \, \de t \, \de s \\
		& \quad \le \cdb \norm{\xi}^p_{W^{1/2, p}_\rho (I, \Lx p)},
	\end{align*}
	\endgroup
	where, for the last inequality, we used the local Lipschitz continuity of the function $l(x, \cdot)$. In a similar way, by using linearity of the convolution and Young's inequality for convolutions, together with hypothesis \ref{ass:j}, one can easily see that also
	\begin{align*}
		& \norm{l(\cdot, \xi) (\nabla J \ast \xi)}^p_{W^{1/2, p}_\rho (I, \Lx p)} \\
		& \quad = \int_0^T \int_0^s \frac{t^{p(1-\rho)}}{(s-t)^{1 + (1/2)p}} \norm{ l(\cdot, \xi(s)) \nabla J \ast \xi(s) - l(\cdot, \xi(t)) \nabla J \ast \xi(t) }^p_{\Lx p} \, \de t \, \de s \\
		& \quad \le \cdb \norm{\xi}^p_{W^{1/2, p}_\rho (I, \Lx p)}.
	\end{align*}
	Next, by recalling that $\Wx{1,p}$ is a Banach algebra if $p \gs N$ and by using \ref{ass:j2} to say that $a \in \Wx{2, q} \hookrightarrow \mathcal{C}^1(\overline{\Omega})$ for $q > N$, we can also estimate the last two terms with similar techniques, namely
	\begin{align*}
		& \norm{l(\cdot, \xi) (\xi \nabla a)}^p_{L^p_\rho (I, \Wx{1,p})} \\
		& \quad \le \norm{l(\cdot, \xi)}^p_{\mathcal{C}^0(0,\tau; \Cx0)} \norm{\xi \nabla a }^p_{L^p_\rho (I, \Wx{1,p})} + \norm{\xi \nabla a}^p_{\mathcal{C}^0(0,\tau; \Cx0)} \norm{l(\cdot, \xi)}^p_{L^p_\rho (I, \Wx{1,p})} \\
		& \quad \le \cdb \int_0^\tau t^{p(1-\rho)} \left( \norm{\xi \nabla a}^p_{\Lx p} + \norm{\nabla \xi \otimes \nabla a}^p_{\Lx p} + \norm{\xi}^p_{\Lx {2p}} \norm{D^2 a}^p_{\Lx {2p}} \right) \, \de t \\
		& \qquad + \cdb \int_0^\tau t^{p(1-\rho)} \left( \norm{l(\cdot, \xi)}^p_{\Lx p}  + \norm{ l'(\cdot, \xi) \nabla \xi }^p_{\Lx p} \right) \, \de t \\
		& \quad \le \cdb \int_0^\tau t^{p(1-\rho)} \norm{\xi}^p_{\Wx{1,p}} \, \de t + \cdb + \cdb \int_0^\tau t^{p(1-\rho)} \norm{\nabla \xi}^p_{\Lx p} \, \de t  \\
		& \quad \le \cdb \left( 1 + \norm{\xi}^p_{L^p_\rho (I, \Wx{1,p})} \right),
	\end{align*}
	where we used H\"older's inequality and the embedding $\Wx{1,p} \hookrightarrow \Lx{2p}$.
	Analogously, by using \ref{ass:j2}, Remark \ref{admissible}, Young's convolution inequality and the same reasoning as above, we also see that 
	\begin{align*}
		& \norm{l(\cdot, \xi) (\nabla J \ast \xi)}^p_{L^p_\rho (I, \Wx{1,p})} \\
		& \quad \le \norm{l(\cdot, \xi)}^p_{\mathcal{C}^0(0,\tau; \Cx0)} \norm{\nabla J \ast \xi}^p_{L^p_\rho (I, \Wx{1,p})} + \norm{\nabla J \ast \xi}^p_{\mathcal{C}^0(0,\tau; \Cx0)} \norm{l(\cdot, \xi)}^p_{L^p_\rho (I, \Wx{1,p})} \\
		& \quad \le \cdb \left( 1 + \norm{\xi}^p_{L^p_\rho (I, \Wx{1,p})} \right).
	\end{align*}
	Then, by putting all together, \eqref{thmsmooth:boundary} now becomes
	\begin{align*}
		& \norm{l(\cdot, \xi) (B \xi_{\trace} \nabla a \cdot \n)}^p_{\Fr} + \norm{l(\cdot, \xi)( B (\nabla J \ast \xi)_{\trace} \cdot \n)}^p_{\Fr} \\
		& \quad \le \cdb \left( 1 + \norm{\xi}^p_{W^{1/2, p}_\rho (I, \Lx p) \cap L^p_\rho (I, \Wx{1,p})} \right).
	\end{align*}
	At this point, observe that 
	\[ \Edr \hookrightarrow W^{1/2, p}_\rho (I, \Lx p) \cap L^p_\rho (I, \Wx{1,p}) \hookrightarrow \Eor, \]
	in particular, one can see that, by interpolation theory (cf. \cite{lunardi}), we actually have that 
	\[ W^{1/2, p}_\rho (I, \Lx p) \cap L^p_\rho (I, \Wx{1,p}) = [\Edr, \Eor]_{1/2}, \]
	then, by using also Young's inequality, it follows that 
	\[ \norm{\xi}^p_{W^{1/2, p}_\rho (I, \Lx p) \cap L^p_\rho (I, \Wx{1,p})} \le \eps \norm{\xi}^p_{\Edr} + C_\eps \norm{\xi}^p_{\Eor} \le \eps \norm{\xi}^p_{\Edr} + C_\eps(\tau, \delta_\beta). \]
	Now we can finally group together all the results we obtained in estimating the right-hand side of \eqref{thmsmooth:mainest}, to deduce that 
	\[ \norm{\vzeta}_{\Edr} \le \eps \norm{\vzeta}_{\Edr} + C_\eps + C \norm{\vpsi(t_1)}_{\Wx{2(\rho-1/p), p}^2}. \]
	Then, by choosing $\eps \gs 0$, depending on $\tau$, $p$, $\delta_\beta$ and the parameters of the system, small enough, we can infer that 
	\begin{equation}
		\label{thmsmooth:finalest}
		\norm{\vzeta}_{\Edr} \le C \left( 1 + \norm{\vpsi(t_1)}_{\Wx{2(\rho-1/p),p}^2} \right),
	\end{equation} 
	up to renaming the constants. 
	
	Hence, by combining \eqref{thmsmooth:finalest} with \eqref{thmsmooth:normt2}, we find that
	\[ \norm{\vpsi(t_2)}_{\Wx {2-2/p,p}} \le C \left( 1 + \norm{\vpsi(t_1)}_{\Wx{2(\rho-1/p),p}^2} \right). \]
	As already anticipated, here is where the regularising effect of the temporal weight $\rho$ comes into play. Indeed, we are now free to choose any $\rho \in (1/p,1]$ and we actually take $\rho \ls 1$ by asking that $\rho = 1/p + \eps$, for some $\eps \ls \min \{ 1 - 1/p, \beta/2 \}$. In this way, the fractional exponent $(2(\rho - 1/p))$ is strictly smaller than $\beta$. Therefore, we can use the embedding $\Cx{\beta} \hookrightarrow \Wx{\alpha,p}$, which holds for any $\alpha \in (0,\beta)$ and any $p \in (1,+\infty)$, to finally deduce the inequality 
	\[ \norm{\vpsi(t_2)}_{(W^{2-2/p,p}(\Omega))^2} \le C ( 1 + \norm{\vpsi(t_1)}_{(\mathcal{C}^\beta(\Omega))^2} ). \]
	This concludes the proof of Lemma \ref{smooth:eff}.
\end{proof}

\section{Strong global well-posedness}
\label{sect:global}

The next step in establishing global high regularity results for the solution to \eqref{eq:phi}-\eqref{ic}, is to extend the maximal regularity found in the previous section to the whole time-interval $(0,T)$. To do this, we first need an $L^\infty(Q_T)$-estimate for both $\phi$ and $\sigma$, which is achieved respectively by an Alikakos-Moser iterative scheme and by parabolic regularity theory. Here it is necessary to put the chemotaxis parameter $\chi$ equal to $0$, in order to make these arguments work. The main problem due to chemotaxis is the treatment of the cross-diffusion terms $-\chi \sigma$ in \eqref{eq:mu} and $+ \chi \Delta \phi$ in \eqref{eq:sigma}, which essentially prevent the $\Lqt\infty$-estimate, if starting only from the low regularity given by weak solutions. Indeed, we recall that the only global result available up to now is the one of Theorem \ref{thm:weaksols}; therefore we are forced to rely on that to prove global boundedness of the solutions. Subsequently, we prove H\"older-type estimates on the solutions, which, in conjunction to Lemma \ref{smooth:eff}, allow us to extend the maximal regularity to the whole $(0,T)$, if $\chi = 0$. This is then enough to prove a continuous dependence result in strong spaces, which will be the starting point of the subsequent study of the optimal control problem.

\subsection{Global maximal regularity}

From now on, we consider the system of equations with $\chi= 0$, namely the system \eqref{eq:phi2}--\eqref{ic2}, which we recall here for convenience:
\begin{alignat*}{2}
	& \partial_t \phi - \Delta \mu = P(\phi) (\sigma - \mu) - \hh(\phi) u \qquad && \text{in } Q_T,  \\
	& \mu = AF'(\phi) + Ba \phi - BJ \ast \phi \qquad && \text{in } Q_T, \\
	& \partial_t \sigma - \Delta \sigma = - P(\phi) (\sigma - \mu) + v \qquad && \text{in } Q_T. \\
	& \partial_{\n} \mu = \partial_{\n} \sigma = 0 \qquad && \text{on } \Sigma_T, \\
	& \phi(0) = \phi_0, \quad \sigma(0) = \sigma_0 \qquad && \text{in } \Omega.
\end{alignat*}
First, we state and prove the results about the $L^\infty(Q_T)$-estimates for $\phi$ and $\sigma$. Note that in both cases, we only need the weak regularity together with stronger assumptions on some data of the system. In particular, we have to further assume that: 
\begin{enumerate}[font = \bfseries, label = B\arabic*., ref = \bf{B\arabic*}]
	\setcounter{enumi}{6} 
	\item\label{ass:pinf} $P \in L^\infty(\R)$.
\end{enumerate}
With this extra hypothesis we can now prove the following result. 

\begin{proposition}
	\label{phi:linf}
	Assume hypotheses \emph{\ref{ass:coeff}-\ref{ass:initial}}, \emph{\ref{ass:pinf}} and let $\phi_0 \in L^\infty(\Omega)$. Let $(\phi, \mu, \sigma)$ be a weak solution to \eqref{eq:phi2}--\eqref{ic2}, with regularities given by Theorem \emph{\ref{thm:weaksols}}.
	
	Then, there exists a constant $C >0$, depending only on the data of the system, such that
	\[ \norm{\phi}_{L^\infty(Q_T)} \le C. \]
\end{proposition}

\begin{proof}
	{\allowdisplaybreaks
	We perform an Alikakos-Moser iteration scheme, by taking some inspiration from \cite[Theorem 2.1]{BH2005}. We start by testing \eqref{eq:phi2} by $\phi \abs{\phi}^{p-1}$, with $p \gs 1$. Note that, for this to be rigorous, one should have to consider truncated versions for the test function, i.e.~$\phi_\lambda \abs{\phi_\lambda}^{p-1}$ with $\phi_\lambda = \max \{ \min \{ \phi, \lambda \}, - \lambda \}$ for any $\lambda \gs 0$. In this way, $\phi_\lambda \abs{\phi_\lambda}^{p-1}$ would be eligible as a test function in the weak formulation. Then, one can proceed with all the estimates below and pass to the limit as $\lambda \to + \infty$ to get the result. Here, with the idea of not overburdening the exposition, we proceed formally. Then, we have:
	\[ \duality{\phi_t, \phi \phiab^{p-1}}_V + ( \nabla \mu, \nabla (\phi \phiab^{p-1}))_H = (P(\phi) (\sigma - \mu), \phi \phiab^{p-1} )_H - (\hh(\phi) u, \phi \phiab^{p-1})_H. \]
	From now on, we will call the reaction term $R := \sigma - \mu$ for simplicity and recall that by Theorem \ref{thm:weaksols} 
	\begin{equation}
		\label{moser:Rest}
		\norm{R}_{\LT 2 V} \le C.
	\end{equation}
	Then, we rewrite the previous identity by using the explicit expression of $\mu$ and we get:
	\begin{equation}
		\label{moser:eq1}
		\begin{split}
		& \duality{\phi_t, \phi \phiab^{p-1}}_V + ( (AF''(\phi) + Ba) \nabla \phi, \nabla ( \phi \phiab^{p-1} ) )_H \\
		& \quad = - B (\nabla a \, \phi, \nabla (\phi \phiab^{p-1}))_H + B (\nabla J \ast \phi, \nabla (\phi \phiab^{p-1}))_H \\
		& \qquad + (P(\phi) R, \phi \phiab^{p-1})_H - (\hh(\phi) u, \phi \phiab^{p-1})_H.
		\end{split}
	\end{equation}
	Before going on, we also recall the following useful identities, which hold for any function $f: Q_T \to \R$ sufficiently regular and any $p \gs 1$:
	\begin{align}
		& \ddt \abs{f}^{p+1} = (p+1) f_t \, f \abs{f}^{p-1}, \label{identity:p1} \\
		& \nabla ( f \abs{f}^{p-1} ) = p \abs{f}^{p-1} \nabla f, \label{identity:p2} \\
		& \nabla (\abs{f}^{\frac{p+1}{2}}) = \frac{p+1}{2} \abs{f}^{\frac{p-1}{2}} \sgn(f) \nabla f. \label{identity:p3}
	\end{align}
	Now we treat all the terms in \eqref{moser:eq1} one by one. Starting from the first one, by using \eqref{identity:p1} (which can be extended to the setting of Hilbert triplets), we deduce that 
	\[ \duality{\phi_t, \phi \phiab^{p-1}}_V = \frac{1}{p+1} \ddt \intom \phiab^{p+1} \, \de x.  \]
	Next, by using \eqref{identity:p2} and \eqref{identity:p3}, together with \ref{ass:fc0}, we deduce that 
	\begin{align*}
		& ( (AF''(\phi) + Ba) \nabla \phi, \nabla ( \phi \phiab^{p-1} ) )_H = p \intom (AF''(\phi) + Ba) \phiab^{p-1} \nabla \phi \cdot \nabla \phi \, \de x \\
		& \quad = p \intom (AF''(\phi) + Ba) \abs{\nabla \phi}^2 ( \phiab^{\frac{p-1}{2}} )^2 \, \de x \\
		& \quad \ge p c_0 \intom \abs{ \phiab^{\frac{p-1}{2}} \nabla \phi }^2 \, \de x = \frac{4 c_0 p}{(p+1)^2} \intom \abs{ \nabla \phiab^{\frac{p+1}{2}} }^2 \, \de x. 
	\end{align*}
	Hence, we are left to estimate the terms on the right-hand side of \eqref{moser:eq1}. Here we will repeatedly use the trivial identity $\phiab^{p} = \phiab^{\frac{p-1}{2}} \phiab^{\frac{p+1}{2}}$ to simplify some calculations. Indeed, by using \ref{ass:j}, \eqref{identity:p2}, \eqref{identity:p3} and Cauchy-Schwarz and Young's inequalities, we have that
	\begin{align*}
		& - B (\nabla a \, \phi, \nabla (\phi \phiab^{p-1}))_H \le B\bstar \abs{(\phi, \nabla (\phi \phiab^{p-1}))_H} \\
		& \quad \le B\bstar \intom p \phiab^p \abs{\nabla \phi} \, \de x = B\bstar p \intom ( \phiab^{\frac{p-1}{2}} \abs{\nabla \phi} ) \phiab^{\frac{p+1}{2}} \, \de x \\
		& \quad \le B\bstar p \left( \frac{4}{(p+1)^2} \intom \abs{ \nabla \phiab^{\frac{p+1}{2}} }^2 \, \de x \right)^{1/2} \left( \intom \phiab^{p+1} \, \de x \right)^{1/2} \\
		& \quad \le \mezzo \frac{c_0 p}{(p+1)^2} \intom \abs{ \nabla \phiab^{\frac{p+1}{2}} }^2 \, \de x + C (p+1) \intom \phiab^{p+1} \, \de x.
	\end{align*} 
	Then, for the convolution term, we use a similar strategy, by exploiting \eqref{identity:p2}, \eqref{identity:p3}, Cauchy-Schwarz and Young's inequalities, H\"older's inequality with $\frac{p-1}{p+1} + \frac{2}{p+1} = 1$ and Young's inequality for convolutions. Indeed, we can infer that 
	\begin{align*}
		& B (\nabla J \ast \phi, \nabla (\phi \phiab^{p-1}))_H = Bp \intom \abs{\nabla J \ast \phi} \phiab^{p-1} \abs{\nabla \phi} \, \de x \\
		& \quad = Bp \intom \left( \phiab^{\frac{p-1}{2}} \abs{\nabla \phi} \right) \left( \abs{\nabla J \ast \phi} \phiab^{\frac{p-1}{2}} \right) \, \de x \\
		& \quad \le \frac{c_0 p}{8} \intom  \left( \phiab^{\frac{p-1}{2}} \abs{\nabla \phi} \right)^2 \, \de x + C p \intom \phiab^{p-1} \abs{\nabla J \ast \phi}^2 \, \de x \\
		& \quad = \mezzo \frac{c_0 p}{(p+1)^2} \intom \abs{ \nabla \phiab^{\frac{p+1}{2}} }^2 \, \de x + C p \left( \intom \phiab^{p+1} \, \de x \right)^{\frac{p-1}{p+1}} \left( \intom \abs{\nabla J \ast \phi}^{p+1} \, \de x \right)^{\frac{2}{p+1}} \\
		& \quad \le \mezzo \frac{c_0 p}{(p+1)^2} \intom \abs{ \nabla \phiab^{\frac{p+1}{2}} }^2 \, \de x + C p  \left( \intom \phiab^{p+1} \, \de x \right)^{\frac{p-1}{p+1}} (\bstar)^2 \left( \intom \phiab^{p+1} \, \de x \right)^{\frac{2}{p+1}} \\
		& \quad \le \mezzo \frac{c_0 p}{(p+1)^2} \intom \abs{ \nabla \phiab^{\frac{p+1}{2}} }^2 \, \de x + C (p+1) \intom \phiab^{p+1} \, \de x. 
	\end{align*}
	Moreover, for the reaction term, we use \ref{ass:pinf} and a combination of the generalised H\"older's inequality with $\frac{1}{6} + \frac{1}{6} + \frac{2}{3} = 1$ and Young's inequality, yielding
	\begin{align*}
		& (P(\phi) R, \phi \phiab^{p-1})_H \le \intom \abs{P(\phi)} \abs{R} \phiab^p \, \de x \le \norm{P}_\infty \intom \abs{R} \phiab^{\frac{p+1}{2}} \phiab^{\frac{p-1}{2}} \, \de x \\
		& \quad \le \norm{R}_{\Lx 6} \left( \intom \phiab^{3(p+1)} \, \de x \right)^{\frac{1}{6}} \left( \intom \phiab^{\frac{3}{4} (p-1)} \, \de x \right)^{\frac{2}{3}} \\ 
		& \quad \le \eps \frac{c_0 p}{(p+1)^2} \left( \intom \phiab^{3(p+1)} \, \de x \right)^{\frac{1}{3}} + C_\eps \frac{(p+1)^2}{4 c_0 p} \norm{R}^2_{\Lx 6} \left( \intom \phiab^{\frac{3}{4} (p-1)} \, \de x \right)^{\frac{4}{3}} \\
		& \quad \le \eps \frac{c_0 p}{(p+1)^2} \left( \intom \phiab^{3(p+1)} \, \de x \right)^{\frac{1}{3}} + C_\eps (p+1) \norm{R}^2_{\Lx 6} \left( \intom \phiab^{\frac{3}{4} (p-1)} \, \de x \right)^{\frac{4}{3}}, 
	\end{align*}
	where we used Young's inequality with $\delta = \eps \frac{c_0 p}{(p+1)^2}$ and $\eps \gs 0$, independent of $p$, to be chosen later. 
	Finally, for the last term we argue exactly as above, by recalling that $u \in \Lqt\infty$ and using hypothesis \ref{ass:h}. Indeed, we have that: 
	\begin{align*}
		& (\hh(\phi) u, \phi \phiab^{p-1})_H \le \norm{h}_\infty \norm{u}_{\Lqt\infty} \intom \phiab^{\frac{p+1}{2}} \phiab^{\frac{p-1}{2}} \, \de x \\
		& \quad \le C \abs{\Omega}^{1/6} \left( \intom \phiab^{3(p+1)} \, \de x \right)^{\frac{1}{6}} \left( \intom \phiab^{\frac{3}{4} (p-1)} \, \de x \right)^{\frac{2}{3}} \\ 
		& \quad \le \eps \frac{c_0 p}{(p+1)^2} \left( \intom \phiab^{3(p+1)} \, \de x \right)^{\frac{1}{3}} + C_\eps (p+1) \left( \intom \phiab^{\frac{3}{4} (p-1)} \, \de x \right)^{\frac{4}{3}}, 
	\end{align*}
	where $\eps \gs 0$, independent of $p$, is again to be chosen later.
	Therefore, starting from \eqref{moser:eq1}, we arrived at
	\begin{equation}
		\label{moser:eq2}
		\begin{split}
		& \frac{1}{p+1} \ddt \intom \phiab^{p+1} \, \de x +  \frac{3 c_0 p}{(p+1)^2} \intom \abs{ \nabla \phiab^{\frac{p+1}{2}} }^2 \, \de x \\
		& \quad \le \eps \frac{c_0 p}{(p+1)^2} \left( \intom \phiab^{3(p+1)} \, \de x \right)^{\frac{1}{3}} + C (p+1) \intom \phiab^{p+1} \, \de x \\
		& \qquad + C_\eps (p+1) \left( 1 + \norm{R}^2_{\Lx 6} \right) \left( \intom \phiab^{\frac{3}{4} (p-1)} \, \de x \right)^{\frac{4}{3}}.
		\end{split}
	\end{equation} 
	At this point, we observe that, by the Sobolev embedding $V \hookrightarrow \Lx 6$, it follows that 
	\[ \norm{ \phiab^{\frac{p+1}{2}} }^2_V \ge \frac{1}{\tilde{C}} \norm{ \phiab^{\frac{p+1}{2}} }^2_{\Lx 6} = \frac{1}{\tilde{C}} \left( \intom \phiab^{3(p+1)} \, \de x \right)^{\frac{1}{3}}, \]
	therefore by the definition of the norm in $V$, we can write that 
	\[ \frac{3c_0 p}{(p+1)^2} \intom \abs{ \nabla \phiab^{\frac{p+1}{2}} }^2 \, \de x \ge \frac{3c_0 p}{\tilde{C}(p+1)^2} \left( \intom \phiab^{3(p+1)} \, \de x \right)^{\frac{1}{3}} - \underbrace{\frac{3c_0 p}{(p+1)^2}}_{\le C(p+1)} \left( \intom \phiab^{p+1} \, \de x \right). \]
	Then, by choosing $\eps = 1/\tilde{C}$, \eqref{moser:eq2} now becomes:
	\begin{equation}
		\label{moser:eq3}
		\begin{split}
			& \frac{1}{p+1} \ddt \intom \phiab^{p+1} \, \de x +  \frac{2 c_0 p}{\tilde{C}(p+1)^2} \left( \intom \phiab^{3(p+1)} \, \de x \right)^{\frac{1}{3}}  \\
			&\quad \le C(p+1)  \intom \phiab^{p+1} \, \de x + C (p+1) \left( 1 + \norm{R}^2_{\Lx 6} \right) \left( \intom \phiab^{\frac{3}{4} (p-1)} \, \de x \right)^{\frac{4}{3}}.
		\end{split}
	\end{equation}
	Next, we want to estimate further the two integrals on the right-hand side, in order to get to an inequality starting from which it is easier to perform an Alikakos-Moser iteration. Indeed, by using H\"older's inequality with $\frac{1}{6} + \frac{1}{6} + \frac{2}{3} = 1$ and Young's inequality with $\eps = \frac{c_0 p}{C \tilde{C} \abs{\Omega}^{1/6} (p+1)^3}$, we infer that
	\begin{align*}
		& C (p+1) \intom \phiab^{p+1} \, \de x = C (p+1) \intom \phiab^{\frac{p+1}{2}} \phiab^{\frac{p+1}{2}} \, \de x \\
		& \quad \le C (p+1) \abs{\Omega}^{\frac{1}{6}} \left( \intom \phiab^{3(p+1)} \, \de x \right)^{\frac{1}{6}} \left( \intom \phiab^{\frac{3}{4} (p+1)} \, \de x \right)^{\frac{2}{3}} \\
		& \quad \le \frac{c_0 p}{\tilde{C}(p+1)^2} \left( \intom \phiab^{3(p+1)} \, \de x \right)^{\frac{1}{3}} + \underbrace{C \frac{(p+1)^4}{p}}_{\le C (p+1)^3} \left( \intom \phiab^{\frac{3}{4} (p+1)} \, \de x \right)^{\frac{4}{3}}. 
	\end{align*}
	Moreover, for the second term, we use Young's inequality with exponents $\frac{p-1}{p+1} + \frac{2}{p+1} = 1$ and the fact that the function $x \mapsto x^{4/3}$ is convex. Therefore, we get:
	\begin{align*}
		&  C (p+1) \left( 1 + \norm{R}^2_{\Lx 6} \right) \left( \intom \phiab^{\frac{3}{4} (p-1)} \, \de x \right)^{\frac{4}{3}} \\
		& \quad \le C (p+1) \left( 1 + \norm{R}^2_{\Lx 6} \right) \left( \frac{p-1}{p+1} \intom \phiab^{\frac{3}{4} (p+1)} \, \de x + \frac{2}{p+1} \abs{\Omega} \right)^{\frac{4}{3}} \\
		& \quad \le C (p+1) \left( 1 + \norm{R}^2_{\Lx 6} \right) \left( \frac{p-1}{p+1} \left( \intom \phiab^{\frac{3}{4} (p+1)} \, \de x \right)^{\frac{4}{3}} + \frac{2}{p+1} \abs{\Omega}^{\frac{4}{3}} \right) \\
		& \quad \le C (p+1) \underbrace{ \frac{p-1}{p+1} }_{\le 1} \left( 1 + \norm{R}^2_{\Lx 6} \right)  \left( \intom \phiab^{\frac{3}{4} (p+1)} \, \de x \right)^{\frac{4}{3}} + C  \left( 1 + \norm{R}^2_{\Lx 6} \right).
	\end{align*} 
	For simplicity, we now call 
	\[ g(t) := 1 + \norm{R}^2_{\Lx 6} \in \Lt 1, \]
	which is integrable due to Theorem \ref{thm:weaksols}.
	Consequently, starting from \eqref{moser:eq3}, we deduce the following inequality: 
	\begin{align*}
		& \frac{1}{p+1} \ddt \intom \phiab^{p+1} \, \de x +  \frac{c_0 p}{\tilde{C}(p+1)^2} \left( \intom \phiab^{3(p+1)} \, \de x \right)^{\frac{1}{3}}  \\
		&\quad \le C \left( (p+1)^3 + (p+1) g(t)  \right) \left( \intom \phiab^{\frac{3}{4} (p+1)} \, \de x \right)^{\frac{4}{3}} + C g(t),
	\end{align*}
	which, by multiplying everything by $(p+1)$ becomes:
	\begin{equation}
		\label{moser:eq4}
		\begin{split}
			& \ddt \intom \phiab^{p+1} \, \de x + \frac{c_0 p}{\tilde{C}(p+1)} \left( \intom \phiab^{3(p+1)} \, \de x \right)^{\frac{1}{3}}  \\
			&\quad \le C \left( (p+1)^4 + (p+1)^2 g(t)  \right) \left( \intom \phiab^{\frac{3}{4} (p+1)} \, \de x \right)^{\frac{4}{3}} + C (p+1) g(t).
		\end{split}
	\end{equation}
	Now notice that the second term on the left-hand side is non-negative, therefore we can ignore it and consider the inequality:
	\begin{equation*}
		\ddt \intom \phiab^{p+1} \, \de x \le C \left( (p+1)^4 + (p+1)^2 g(t) \right) \left( \intom \phiab^{\frac{3}{4} (p+1)} \, \de x \right)^{\frac{4}{3}} + C (p+1) g(t). 
	\end{equation*}
	At this point, we can integrate on $(0,t)$, for any $t \in (0,T)$, and use \eqref{moser:Rest}, together with the embedding $V \hookrightarrow \Lx 6$, to deduce that
	{\allowdisplaybreaks
	\begin{align*}
		& \intom \abs{\phi(t)}^{p+1} \, \de x \le \intom \abs{\phi_0}^{p+1} \, \de x + C \int_0^T  \left( (p+1)^4 + (p+1)^2 g(t) \right) \left( \intom \phiab^{\frac{3}{4} (p+1)} \, \de x \right)^{\frac{4}{3}} \, \de t \\
		& \qquad + C (p+1) \int_0^T g(t) \, \de t  \\
		& \quad \le \abs{\Omega} \norm{\phi_0}^{p+1}_{\Lx\infty} + C (p+1) + C \underbrace{\left( T (p+1)^4 + (p+1)^2 \int_0^T g(t) \, \de t \right)}_{\le C (p+1)^4} \sup_{(0,T)} \left( \intom \phiab^{\frac{3}{4} (p+1)} \, \de x \right)^{\frac{4}{3}} \\
		& \quad \le \abs{\Omega} \norm{\phi_0}^{p+1}_{\Lx\infty}  + C (p+1) + C (p+1)^4 \sup_{(0,T)} \left( \intom \phiab^{\frac{3}{4} (p+1)} \, \de x \right)^{\frac{4}{3}} \\
		& \quad \le C (p+1)^4 \max \left\{ \norm{\phi_0}^{p+1}_{\Lx\infty}, 1, \sup_{(0,T)} \left( \intom \phiab^{\frac{3}{4} (p+1)} \, \de x \right)^{\frac{4}{3}} \right\} \\
		& \quad \le C (p+1)^4 \max \left\{ \max\{\norm{\phi_0}_{\Lx\infty}, 1\}^{p+1}, \sup_{(0,T)} \left( \intom \phiab^{\frac{3}{4} (p+1)} \, \de x \right)^{\frac{4}{3}} \right\}.
	\end{align*}
	Therefore, by taking the supremum on $(0,T)$ also on the left-hand side, we arrive at the inequality: 
	\begin{equation}
		\label{moser:eq5}
		\sup_{(0,T)} \intom \phiab^{p+1} \, \de x \le C (p+1)^4 \max \left\{ \max\{\norm{\phi_0}_{\Lx\infty}, 1\}^{p+1}, \sup_{(0,T)} \left( \intom \phiab^{\frac{3}{4} (p+1)} \, \de x \right)^{\frac{4}{3}} \right\},
	\end{equation}
	where the constant $C \gs 0$ depends only on $\Omega$, $T$, the parameters of the system and not on $p$.%
	}

	We can now start the iteration scheme, by taking a sequence $\{p_k\}_{k \in \N}$ such that $p_k \to + \infty$ as $k \to +\infty$, defined in the following way:
	\[ p_0 = 2, \quad p_{k+1} = \frac{4}{3} p_k \quad \forall k \in \N. \]
	Then, by using $p = p_{k+1} - 1 \gs 1$ in \eqref{moser:eq5} we get that 
	\[ \sup_{(0,T)} \intom \phiab^{p_{k+1}} \, \de x \le C (p_{k+1})^4 \max \left\{ \max\{\norm{\phi_0}_{\Lx\infty}, 1\}^{p_{k+1}}, \sup_{(0,T)} \left( \intom \phiab^{\frac{3}{4} p_{k+1}} \, \de x \right)^{\frac{4}{3}} \right\}. \] 
	Hence, by calling $C_0 = \max\{\norm{\phi_0}_\infty, 1\}$ and observing that $\frac{3}{4} p_{k+1} = p_k$, we further arrive at the inequality:
	\begin{equation}
		\label{moser:eq6}
		\sup_{(0,T)} \intom \phiab^{p_{k+1}} \, \de x \le C (p_{k+1})^4 \max \left\{ C_0^{p_{k+1}}, \sup_{(0,T)} \left( \intom \phiab^{p_k} \, \de x \right)^{\frac{4}{3}} \right\}.
	\end{equation}
	Finally, we can apply \cite[Lemma A.1]{L1994} with 
	\begin{align*}
		& \delta_0 = 2, \quad \delta_k = p_k \quad \forall k \in \N, \quad a = \frac{4}{3} \gs 1, \quad c = 0, \quad b = 4 \ge 0 \\
		& \gamma_k = \sup_{(0,T)} \intom \phiab^{p_k} \, \de x, \quad \gamma_0 = \sup_{(0,T)} \intom \phiab^2 \, \de x,  
	\end{align*}
	where $\gamma_0 \le C$ by Theorem \ref{thm:weaksols}, with $C$ depending only on the parameters of the system. Indeed, by applying \cite[Lemma A.1]{L1994}, we can infer that
	\[ \gamma_k^{\frac{1}{\delta_k}} = \left( \sup_{(0,T)} \intom \phiab^{p_k} \, \de x \right)^{\frac{1}{p_k}} = \sup_{(0,T)} \norm{\phi}_{\Lx {p_k}} \le \bar{C}, \]
	with $\bar{C}$ independent of $k$. Then, by sending $k \to +\infty$ and recalling that, if $\abs{\Omega} \ls +\infty$, $\norm{f}_p \to \norm{f}_\infty$ as $p \to + \infty$ for any $f$ measurable, we obtain that
	\[ \norm{\phi}_{\Lqt \infty} \le C. \]
	This concludes the proof of Proposition \ref{phi:linf}.  
	}
\end{proof}

\begin{remark}
	\label{mu:linf}
	Under the hypothesis of Proposition \ref{phi:linf}, it also follows that there exists a constant $C \gs 0$, depending only on the parameters of the system, such that 
	\[ \norm{ \mu }_{\Lqt \infty} \le C. \]
	Indeed, from equation \eqref{eq:mu2}, since $F \in \mathcal{C}^1$ and thus locally bounded, it follows that 
	\begin{align*}
		\norm{\mu}_{\Lqt \infty} & \le A \norm{F'(\phi)}_{\Lqt \infty} + B \norm{a}_{\Lx\infty} \norm{\phi}_{\Lqt\infty} + B \norm{J}_{\Lx 1} \norm{\phi}_{\Lqt \infty} \\
		& \le A \norm{F'(\phi)}_{\Lqt \infty} + 2 B a^* \norm{\phi}_{\Lqt\infty} \le C,
	\end{align*}
	where we also used Young's inequality for convolutions, together with hypothesis \ref{ass:j}, and Proposition \ref{phi:linf}.
\end{remark}

\begin{remark}
	\label{rmk:normsfp_infty}
	Now that $\phi \in \Lqt\infty$ under the hypotheses of Proposition \ref{phi:linf}, since $F \in \mathcal{C}^4(\R)$, $P \in \mathcal{C}^1(\R)$ and $\hh \in \mathcal{C}^1(\R)$, by local boundedness we can say that there exists a constant $C \gs 0$, depending only on the parameters of the system, such that
	\[ \norm{ F^{(i)}(\phi) }_{\Lqt\infty} + \norm{ P^{(j)}(\phi) }_{\Lqt\infty} +  \norm{ \hh^{(j)}(\phi) }_{\Lqt\infty} \le C \quad \text{for any $i=1,\dots,4$ and $j =0,1$.} \]
\end{remark}

\begin{proposition}
	\label{sigma:linf}
	Assume hypotheses \emph{\ref{ass:coeff}-\ref{ass:initial}} and \emph{\ref{ass:pinf}}. Let $\phi_0 \in L^\infty(\Omega)$, $\sigma_0 \in L^\infty(\Omega)$ and $v \in L^\infty(0,T;H)$. Let $(\phi, \mu, \sigma)$ be a weak solution to \eqref{eq:phi2}--\eqref{ic2}, with regularities given by Theorem \emph{\ref{thm:weaksols}} and Propositions \emph{\ref{phi:linf}}.
	
	Then, there exists a constant $C >0$, depending only on the data of the system, such that
	\[ \norm{\sigma}_{L^\infty(Q_T)} \le C. \]
\end{proposition}

\begin{proof}
	This is just an application of maximum principle for parabolic equations. Indeed, we rewrite equation \eqref{eq:sigma2} as 
	\[ \sigma_t - \Delta \sigma = f_\sigma, \quad \text{with } f_\sigma = P(\phi) (\sigma - \mu) + v. \]
	Then, we see that $f_\sigma \in \LT \infty H$ uniformly with respect to the parameters, since $\phi \in \Lqt \infty$ by Proposition \ref{phi:linf}, $\mu \in \Lqt \infty$ by Remark \ref{mu:linf}, $\sigma \in \LT \infty H$ by Theorem \ref{thm:weaksols} and $v \in \LT \infty H$ by hypothesis. Then, we can apply \cite[Theorem 7.1, p.~181]{ladyzhenskaja} with $q = \infty$ and $r = 2$ to conclude the proof. 
\end{proof}

Next, we state and prove global H\"older-type estimates for both $\phi$ and $\sigma$. 

\begin{proposition}
	\label{phisigma:holder}
	Assume \emph{\ref{ass:coeff}-\ref{ass:initial}} and \emph{\ref{ass:pinf}}. Let $\phi_0, \sigma_0 \in L^\infty(\Omega)$ and $v \in L^\infty(Q_T)$.
	Let $(\phi, \mu, \sigma)$ be a weak solution to \eqref{eq:phi2}--\eqref{ic2}, with regularities given by Theorem \emph{\ref{thm:weaksols}} and Propositions \emph{\ref{phi:linf}} and \emph{\ref{sigma:linf}}.
	
	Then, there exist $\beta \in (0,1)$ and a constant $C > 0$, depending only on the parameters of the system, such that
	\[ \abs{\phi (x,t) - \phi (y,s)} + \abs{\sigma (x,t) - \sigma (y,s)} \le C \left( \abs{x-y}^\beta + \abs{t-s}^{\frac{\beta}{2}} \right) \quad \forall (x,t), (y,s) \in \overline{\Omega} \times [0,T], \]
	which means that $\phi, \sigma \in \mathcal{C}^{\beta, \beta/2}(\overline{\Omega} \times [0,T])$ uniformly.
\end{proposition}

\begin{proof}
	The proof is inspired by \cite[Lemma 2]{FGG2016} and references therein, where the authors prove a similar estimate for a non-local Cahn-Hilliard equation with a convection term. The argument is heavily based on the results contained in \cite[Chapter II, Section 7]{ladyzhenskaja}.
	
	Let $R > 0$ be such that
	\[ \norm{\phi}_{L^\infty(Q_T)} \le R \quad \text{and} \quad \norm{\sigma}_{L^\infty(Q_T)} \le R.  \]
	We start by proving the H\"older estimate for $\phi$.
	With the idea of applying \cite[Chapter II, Theorem 7.1]{ladyzhenskaja}, we let $k \in [0,R]$ and $\zeta = \zeta (x,t) \in [0,1]$ be a continuous piecewise smooth function, supported on space-time cylinders defined as $Q_{t_0, t_0 + \tau}(x_0, \rho)  := B_\rho(x_0) \times (t_0, t_0 + \tau)$, where $B_\rho(x_0)$ is the open ball centred at $x_0$ of radius $\rho > 0$. Our aim is to prove an estimate like the one in \cite[Chapter II, Remark 7.2]{ladyzhenskaja}. Indeed, call $\phi_k^+ := \max \{ 0, \phi - k  \}$, multiply equation \eqref{eq:phi2} by $\zeta^2 \phi_k^+$ and integrate over $Q_{t_0,t} := \Omega \times (t_0,t)$, where $0 \le t_0 < t < t_0 + \tau \le T$, to get:
	\[ \int_{Q_{t_0,t}} \phi_t \, \zeta^2 \phi_k^+ \, \de x \, \de s + \int_{Q_{t_0,t}} \nabla \mu \cdot \nabla (\zeta^2 \phi_k^+ )  \, \de x \, \de s = \int_{Q_{t_0,t}} ( P(\phi) (\sigma - \mu) - \hh(\phi) u) \, \zeta^2 \phi_k^+ \, \de x \, \de s.  \]
	By computing $\nabla \mu$, we can rewrite the previous equality as:
	\begin{align*}
		& \int_{Q_{t_0,t}} \phi_t \, \zeta^2 \phi_k^+ \, \de x \, \de s + \int_{Q_{t_0,t}} (A F''(\phi) + Ba)\nabla \phi \cdot \nabla (\zeta^2 \phi_k^+ )  \, \de x \, \de s \\ 
		& \qquad + \int_{Q_{t_0,t}} (\nabla a \, \phi - \nabla J \ast \phi) \cdot \nabla (\zeta^2 \phi_k^+ )  \, \de x \, \de s =  \int_{Q_{t_0,t}} ( P(\phi) (\sigma - \mu) - \hh(\phi) u) \, \zeta^2 \phi_k^+ \, \de x \, \de s.
	\end{align*}
	Now, we consider each term one by one. By using the definition of $\phi_k^+$ and, in particular, the fact that $(\phi_k^+)_t = \phi_t$ on $\{ \phi > k \}$, we infer that
	\begin{align*}
		& \int_{Q_{t_0,t}} \phi_t \, \zeta^2 \phi_k^+ \, \de x \, \de s = \int_{Q_{t_0,t}} (\phi_k^+)_t \, \zeta^2 \phi_k^+ \, \de x \, \de s \\
		& \qquad = \int_{Q_{t_0,t}} \frac{1}{2} \frac{\de}{\de s} ( (\phi_k^+)^2 \zeta^2 ) \, \de x \, \de s - \int_{Q_{t_0,t}} (\phi_k^+)^2 \zeta \zeta_t \, \de x \, \de s \\
		& \qquad = \frac{1}{2} \intom [(\phi_k^+)^2 \zeta^2 ](t) \, \de x - \frac{1}{2} \intom [ (\phi_k^+)^2 \zeta^2 ](t_0) \, \de x - \int_{Q_{t_0,t}} (\phi_k^+)^2 \zeta \zeta_t \, \de x \, \de s.
	\end{align*}
	Next, by using the fact that $\nabla \phi_k^+ = \nabla \phi$ on $\{ \phi > k \}$ and \ref{ass:fc0}, we can estimate the second term from below in the following way: 
	\begin{align*}
		& \int_{Q_{t_0,t}} (A F''(\phi) + Ba)\nabla \phi \cdot \nabla (\zeta^2 \phi_k^+ )  \, \de x \, \de s = \int_{Q_{t_0,t}} (A F''(\phi) + Ba)\nabla \phi_k^+ \cdot \nabla (\zeta^2 \phi_k^+ ) \, \de x \, \de s \\
		& \qquad = \int_{Q_{t_0,t}} (A F''(\phi) + Ba) \abs{\nabla (\zeta \phi_k^+ )}^2 \, \de x \, \de s - \int_{Q_{t_0,t}} (A F''(\phi) + Ba) \abs{\nabla \zeta}^2 ( \phi_k^+ )^2 \, \de x \, \de s \\
		& \qquad \ge c_0 \int_{Q_{t_0,t}} \abs{\nabla (\zeta \phi_k^+ )}^2 \, \de x \, \de s - \underbrace{ \norm{AF''(\phi) + Ba}_{L^\infty(Q_T)} }_{\le C_R} \int_{Q_{t_0,t}} \abs{\nabla \zeta}^2 ( \phi_k^+ )^2 \, \de x \, \de s.
	\end{align*}
	Finally, we can also estimate the other two terms from above. Indeed, by using \ref{ass:j} and H\"older's and Young's inequalities, we have that
	\begin{align*}
		& \int_{Q_{t_0,t}} (\nabla a \, \phi - \nabla J \ast \phi) \cdot \nabla (\zeta^2 \phi_k^+ ) \, \de x \, \de s \le \norm{\nabla a \, \phi - \nabla J \ast \phi}_{L^\infty(Q_T)} \int_{Q_{t_0,t}} \abs{\nabla (\zeta^2 \phi_k^+ )} \, \de x \, \de s \\
		& \qquad \le C_R \int_{Q_{t_0,t}} \abs*{ \zeta^2 \nabla \phi_k^+ + 2 \zeta \nabla \zeta \phi_k^+ } \, \de x \, \de s = C_R \int_{Q_{t_0,t}} \abs*{ \zeta \nabla ( \zeta \phi_k^+) + \zeta \nabla \zeta \phi_k^+ } \, \de x \, \de s \\
		& \qquad \le \frac{c_0}{4} \int_{Q_{t_0,t}} \abs{\nabla (\zeta \phi_k^+ )}^2 \, \de x \, \de s + C_R \int_{Q_{t_0,t}} \abs{\zeta}^2 \, \de x \, \de s + C_R \int_{Q_{t_0,t}} \abs{\nabla \zeta}^2 ( \phi_k^+ )^2 \, \de x \, \de s,
	\end{align*}
	and, by using \ref{ass:h} and \ref{ass:pinf}, we get
	\begin{align*}
		& \int_{Q_{t_0,t}} ( P(\phi) (\sigma - \mu) - \hh(\phi) u ) \, \zeta^2 \phi_k^+ \, \de x \, \de s \\
		& \quad \le \left( \norm{P}_\infty \norm{\sigma - \mu}_{L^\infty(Q_T)} + \norm{\hh}_\infty  \norm{u}_{\Lqt\infty} \right) \int_{Q_{t_0,t}} \zeta^2 \phi_k^+ \, \de x \, \de s \\
		& \qquad \le C_R \norm{\phi_k^+}_{L^\infty(Q_T)} \int_{Q_{t_0,t}} \abs{\zeta}^2 \, \de x \, \de s \le C_R \int_{Q_{t_0,t}} \abs{\zeta}^2 \, \de x \, \de s.
	\end{align*}
	Then, by putting all together and by taking the supremum over $(t_0,t)$ on the right-hand side, we obtain:
	\begin{equation}
		\label{holder:est}
		\begin{split}
		& \frac{1}{2} \sup_{s \in (t_0,t)} \int_\Omega ( \phi_k^+ \zeta )^2 (s) \, \de x + \frac{c_0}{2} \int_{Q_{t_0,t}} \abs{\nabla (\zeta \phi_k^+ )}^2 \, \de x \, \de s \\
		& \qquad \le \frac{1}{2} \int_\Omega ( \phi_k^+ \zeta )^2 (t_0) \, \de x + C_R \int_{Q_{t_0,t}} \left( \abs{\nabla \zeta}^2 ( \phi_k^+ )^2 + (\phi_k^+)^2 \zeta \zeta_t + \abs{\zeta}^2 \right) \, \de x \, \de s.
		\end{split}
	\end{equation}
	Arguing in a similar fashion, one can easily see that inequality \eqref{holder:est} also holds with $\phi$ replaced by $-\phi$. Therefore, thanks to \cite[Chapter II, Remark 7.2]{ladyzhenskaja}, we can say that $\phi$ is an element of the class $\mathcal{B}(Q_{0,T},R,\gamma,r,0,\kappa)$ in the sense of \cite[Chapter II, Section 7]{ladyzhenskaja}, for some $\gamma > 0$, $r > 2$ and $\kappa > 0$. Then, we can apply \cite[Chapter II, Theorem 7.1]{ladyzhenskaja} to infer the existence of a $\beta \in (0,1)$ such that 
	\[ \abs{\phi (x,t) - \phi (y,s)} \le C \left( \abs{x-y}^\beta + \abs{t-s}^{\frac{\beta}{2}} \right) \quad \text{for any } (x,t), (y,s) \in \overline{\Omega} \times [0,T]. \]
	
	To get a similar estimate for $\sigma$, we argue in the same way by multiplying equation \eqref{eq:sigma2} by $\zeta^2 \sigma_k^+$ and integrating again on $Q_{t_0,t}$. Here, the situation is easier since we just have the laplacian operator for $\sigma$ and the reaction term and the source term can be treated essentially in the same way. Therefore, we can once again use \cite[Chapter II, Theorem 7.1 and Remark 7.2]{ladyzhenskaja} to conclude.
\end{proof}

\begin{remark}
	The result of Proposition \ref{phisigma:holder} implies in particular that $\phi, \sigma \in \mathcal{C}^0([0,T]; \Cbeta)$ for some $\beta \in (0,1)$, uniformly with respect to the parameters, which means that
	\[ \sup_{t \in [0,T]} \norm{(\phi (t), \sigma (t))}_{\mathcal{C}^\beta(\Omega)^2} \le C. \]
	This is what we will use in the following Theorem.
\end{remark}

Now we are ready to prove the main results of this subsection:

\begin{theorem}
	\label{global:maxreg}
		Let $\rho = 1$, $s = 2 - \frac{2}{p} > 1 + \frac{N}{p}$, with $p\in (N+2, +\infty)$. Assume also hypotheses \emph{\ref{ass:coeff}-\ref{ass:h}} and \emph{\ref{ass:j2}-\ref{ass:pinf}}. 
		
		Let $(\phi, \sigma)$ be the maximal solution to \emph{\eqref{eq:phi2}--\eqref{ic2}} in the sense of Definition \emph{\ref{max:sols}} with $\chi= 0$. Then the maximal solution is \emph{global}, i.e. the maximal existence time $t^+=t^+(\phi_0,\sigma_0,u,v)$ is equal to $T$ for any $(\phi_0,\sigma_0) \in M^{2-2/p,p}$ and for any fixed $u, v \in \Lqt \infty$.
\end{theorem}

\begin{proof}
	By Theorem \ref{local:maxsol}, let $(\phi, \sigma)$ be the maximal solution to \emph{\eqref{eq:phi2}--\eqref{ic2}} in the sense of Definition \ref{max:sols} with $\chi= 0$ on $I=[0,t^+)$. Assume by contradiction that $t^+ < T$, then since $(\phi,\sigma) \in \mathcal{C}([0,t^+); M^{2-2/p,p})$, by definition of maximal time-interval of existence, it should hold that $\lim_{t\to t^+} \norm{(\phi (t), \sigma(t))}_{W^{2-2/p,p}(\Omega)} = +\infty$. However, by Lemma \ref{smooth:eff} and Proposition \ref{phisigma:holder}, it follows that
	\[ \sup_{t \in [0,t^+)} \norm{(\phi (t), \sigma (t))}_{W^{2-2/p,p}(\Omega)} \le C \left( 1 + \sup_{t \in [0,t^+/2)} \norm{(\phi (t), \sigma (t))}_{\mathcal{C}^\beta(\Omega)^2} \right) \le C,  \]
	for $\beta \in (0,1)$, given by Proposition \ref{phisigma:holder}. 
	We recall that Lemma \ref{smooth:eff} holds for \emph{any} $\beta \in (0,1)$, thus here we choose exactly the one provided by Proposition \ref{phisigma:holder}.
	Moreover, observe that we were allowed to apply Proposition \ref{phisigma:holder} because $(\phi_0,\sigma_0)\in W^{2-2/p,p}(\Omega) \hookrightarrow L^\infty(\Omega)$. This contradicts the assumption, so the maximal solution can be continued to the whole interval $[0,T]$.
\end{proof}	

In the end, we have shown the following result:

\begin{theorem}
	\label{strong:sols}
	Assume hypotheses \emph{\ref{ass:coeff}--\ref{ass:h}}, \emph{\ref{ass:j2}--\ref{ass:u2}} and \emph{\ref{ass:pinf}}.
	Assume further that 
	\begin{equation}
		\label{hp:initaldata}
		\phi_0, \sigma_0 \in H^2(\Omega) \text{ with } 
		\begin{cases}
			\partial_{\n} \mu(0) = \partial_{\n} (AF'(\phi_0) + Ba \phi_0 - B J \ast \phi_0) = 0 \\
			\partial_{\n} \sigma_0 = 0
		\end{cases}
		\text{on $\partial \Omega$}.
	\end{equation}
	Then, there exists a unique solution $(\phi,\mu,\sigma)$ to \eqref{eq:phi2}--\eqref{ic2} such that
	\[ \phi, \mu, \sigma \in W^{1,6}(0,T; L^6(\Omega)) \cap \mathcal{C}^0([0,T]; W^{1,6}(\Omega)) \cap L^6(0,T; W^{2,6}(\Omega)) \cap \Cqt0. \]
	Moreover, there exists a constant $C > 0$, depending only on the parameters of the system, such that the following estimate holds:
	\begin{equation}
		\label{eq:strongestp6}
		\norm{ (\phi, \mu, \sigma) }_{ (W^{1,6}(0,T; L^6(\Omega)) \cap \mathcal{C}^0([0,T]; W^{1,6}(\Omega)) \cap L^6(0,T; W^{2,6}(\Omega)))^3 } \le C.
	\end{equation}
\end{theorem}

\begin{proof}
	One just needs to apply Theorem \ref{global:maxreg} with $p=6$, by observing that $H^2(\Omega) \hookrightarrow W^{2-2/p,p}(\Omega)$ if $p \in (N+2,6]$ and $N\le 3$. Then, the same regularity for $\mu$ easily follows by comparison in \eqref{eq:mu2}, since $F$ is regular and $\phi$ is bounded. Note that the regularities $\C 0 {\Wx{1,6}}$ and $\Cqt0$ come from standard embeddings.
\end{proof}

\begin{remark}
	Observe that on initial data $\phi_0$ and $\sigma_0$ we now assume \eqref{hp:initaldata}, as it is commonly done when seeking stronger solutions for Cahn-Hilliard type equations.
\end{remark}

\begin{remark}
	\label{regularity:CH}
	We also mention that the procedure used above also gives new regularity results for the standard non-local Cahn-Hilliard equation with constant mobility and regular potential, i.e.~the system
	\begin{alignat*}{2}
		& \phi_t - \Delta \mu = 0 \qquad && \hbox{in $Q_T$,} \\
		& \mu = A F'(\phi) + Ba \phi - B J \ast \phi \qquad && \hbox{in $Q_T$,} \\
		& \partial_{\n} \mu = 0 \qquad && \hbox{on $\Sigma_T$,} \\
		& \phi(0) = \phi_0 \qquad && \hbox{in $\Omega$.}
	\end{alignat*}
	Indeed, without considering the equation for $\sigma$ and by neglecting chemotaxis and forgetting the reaction and source terms, all the procedure in Sections \ref{sect:maxreg} and \ref{sect:global} can be easily repeated. Then, under hypotheses \ref{ass:j}--\ref{ass:fgrowth}, \ref{ass:j2}--\ref{ass:f2}, if $\phi_0 \in W^{2-2/p,p}(\Omega)$ for $p \in (N+2, +\infty)$ is such that $\partial_n \mu(0) = 0$ on $\partial \Omega$, one is able to prove that the unique solution to the non-local Cahn-Hilliard equation above is such that
	\[ \phi, \mu \in W^{1,p}(0,T; L^p(\Omega)) \cap \mathcal{C}^0([0,T]; W^{2-2/p,p}(\Omega)) \cap L^p(0,T; W^{2,p}(\Omega)) \cap \Cqt0. \]
\end{remark}
		
\subsection{Continuous dependence on data}

In this subsection, we prove that the strong solutions of Theorem \ref{strong:sols} depend continuously on the controls $u$ and $v$ and initial data $\phi_0, \sigma_0$. This result will be crucial in proving differentiability properties of the control-to-state operator in the next section and will strongly use the regularity estimate \eqref{eq:strongestp6}. 

\begin{theorem}
	\label{thm:contdepstrong}
	Assume hypotheses \emph{\ref{ass:coeff}--\ref{ass:h}}, \emph{\ref{ass:j2}--\ref{ass:h2}} and \emph{\ref{ass:pinf}}. Let $u_1$, $v_1$, ${\phi_0}_1$, ${\sigma_0}_1$ and $u_2$, $v_2$, ${\phi_0}_2$, ${\sigma_0}_2$ be two sets of data satisfying \emph{\ref{ass:u2}} and \emph{\eqref{hp:initaldata}} and let $(\phi_1, \mu_1, \sigma_1)$ and $(\phi_2, \mu_2, \sigma_2)$ two corresponding strong solutions as in Theorem \emph{\ref{strong:sols}}. Then, there exists a constant $K>0$, depending only on the data of the system and on the norms of $\{ (u_i, v_i, {\phi_0}_i, {\sigma_0}_i) \}_{i=1,2}$, but not on their difference, such that
	\begin{equation}
		\label{contdep:estimate}
		\begin{split}
			& \norm{\phi_1 - \phi_2}_{\HT 1 H \cap \LT \infty V \cap \LT 2 {\Hx 2}} + \norm{\mu_1 - \mu_2}_{\HT 1 H \cap \LT \infty V \cap \LT 2 W} \\
			& \qquad + \norm{\sigma_1 - \sigma_2}_{\HT 1 H \cap \LT \infty V \cap \LT 2 W} \\
			& \quad \le K \left( \norm{u_1 - u_2}_{\LT 2 H} + \norm{v_1 - v_2}_{\LT 2 H} + \norm{{\phi_0}_1 - {\phi_0}_2}_V + \norm{{\sigma_0}_1 - {\sigma_0}_2}_V \right).
		\end{split}
	\end{equation}
\end{theorem}

\begin{proof}
	Let $\phi = \phi_1 - \phi_2$, $\mu = \mu_1 - \mu_2$, $\sigma = \sigma_1 - \sigma_2$, $u = u_1 - u_2$, $v = v_1 - v_2$, $\phi_0 = {\phi_0}_1 - {\phi_0}_2$ and $\sigma_0 = {\sigma_0}_1 - {\sigma_0}_2$, then, up to adding and subtracting some terms, they solve:     
	\begin{align}
		\partial_t \phi & = \Delta \mu + P(\phi_1) (\reactiondc) + (P(\phi_1) - P(\phi_2)) (\reactiontwo) \nonumber \\
		& \quad - \hh(\phi_1) u - (\hh(\phi_1) - \hh(\phi_2)) u_2 & \text{in } Q_T,  \label{eq:phi3}\\
		\mu & = A(F'(\phi_1)- F'(\phi_2)) + Ba \phi - BJ \ast \phi & \text{in } Q_T,  \label{eq:mu3} \\
		\partial_t \sigma & = \Delta \sigma - P(\phi_1) (\reactiondc) - (P(\phi_1) - P(\phi_2)) (\reactiontwo) + v & \text{in } Q_T, \label{eq:sigma3}
	\end{align}
	paired with boundary and initial conditions:
	\begin{alignat}{2}
		& \partial_{\n} \mu = \partial_{\n} \sigma = 0 \qquad && \text{on } \Sigma_T, \label{bc3} \\
		& \phi(0) = \phi_0, \quad \sigma(0) = \sigma_0 \qquad && \text{in } \Omega. \label{ic3}
	\end{alignat}
	Now, for the first estimate, we test \eqref{eq:phi3} by $\phi$ in $H$, \eqref{eq:sigma3} by $\sigma$ in $H$ and then sum them up to obtain:
	\begin{align}
		\label{contdep:eq1}
			& \mezzo \ddt \norm{\phi}^2_H + \mezzo \ddt \norm{\sigma}^2_H + (\nabla \mu, \nabla \phi)_H + \norm{\nabla \sigma}^2_H = ( P(\phi_1) (\reactiondc), \phi - \sigma )_H  \\
			& \quad + ( (P(\phi_1) - P(\phi_2)) (\reactiontwo), \phi - \sigma )_H - (\hh(\phi_1) u - (\hh(\phi_1) - \hh(\phi_2)) u_2, \phi)_H + (v, \sigma)_H. \nonumber
	\end{align}
	Next, we start by estimating the term $(\nabla \mu, \nabla \phi)_H$. Indeed, by using equation \eqref{eq:mu3}, up to adding and subtracting some terms, we have to estimate:
	\begin{align*}
		(\nabla \mu, \nabla \phi)_H & = ( (AF''(\phi_1) + Ba) \nabla \phi, \nabla \phi)_H + A( (F''(\phi_1) - F''(\phi_2)) \nabla \phi_2, \nabla \phi )_H \\
		& \quad + B (\nabla a \, \phi, \nabla \phi)_H - B(\nabla J \ast \phi, \nabla \phi)_H.
	\end{align*}
	Hence, by using hypotheses \ref{ass:fc0}, \ref{ass:j}, H\"older, Gagliardo-Nirenberg \eqref{gn:ineq} and Young's inequalities, together with the fact that $F''$ is locally Lipschitz and $\phi_i$, $i=1,2$, is globally bounded, we infer that
	\begin{align}
		\label{contdep:est_gradphimu}
		(\nabla \mu, \nabla \phi)_H & \ge c_0 \norm{\nabla \phi}^2_H - A \norm{F''(\phi_1) - F''(\phi_2)}_{\Lx 3} \norm{\nabla \phi_2}_{\Lx 6} \norm{\nabla \phi}_H - 2B\bstar \norm{\phi}_H \norm{\nabla \phi}_H \nonumber \\
		& \ge c_0 \norm{\nabla \phi}^2_H - C \norm{\phi}_{\Lx 3} \norm{\nabla \phi_2}_{\Lx 6} \norm{\nabla \phi}_H - 2B\bstar \norm{\phi}_H \norm{\nabla \phi}_H \nonumber \\
		& \ge c_0 \norm{\nabla \phi}^2_H - C \norm{\nabla \phi_2}_{\Lx 6} \norm{\nabla \phi}^{3/2}_H \norm{\phi}^{1/2}_H - 2B\bstar \norm{\phi}_H \norm{\nabla \phi}_H \nonumber \\
		& \ge \frac{c_0}{4} \norm{\nabla \phi}^2_H - C \left(1 + \norm{\nabla \phi_2}^4_{\Lx 6} \right) \norm{\phi}^2_H. 
	\end{align}
	Then, regarding the terms on the right-hand side of \eqref{contdep:eq1}, we use again \eqref{eq:mu3}, the local Lipschitz continuity of $F'$, $P$ and $\hh$, H\"older and Young's inequalities, hypothesis \ref{ass:j}, the embedding $V \hookrightarrow \Lx 6$ and Remark \ref{rmk:normsfp_infty}, to deduce that 
	\begin{align*}
		& ( P(\phi_1) (\reactiondc), \phi - \sigma )_H \le \norm{P(\phi_1)}_{\Lx \infty} \norm{ \reactiondc }_H \norm{\phi - \sigma}_H \\
		& \quad = C \norm{ \sigma - A ( F'(\phi_1) - F'(\phi_2) ) - Ba \phi + B J \ast \phi }_H ( \norm{\phi}_H + \norm{\sigma}_H ) \\
		& \quad \le C ( \norm{\sigma}_H + C \norm{\phi}_H + 2Ba^* \norm{\phi}_H ) ( \norm{\phi}_H + \norm{\sigma}_H ) \\ 
		& \quad \le C \norm{\phi}^2_H + C \norm{\sigma}^2_H, \\ 
		& ( (P(\phi_1) - P(\phi_2)) (\reactiontwo), \phi - \sigma )_H \\
		& \quad \le \norm{P(\phi_1) - P(\phi_2)}_H \norm{\reactiontwo}_{\Lx 4} \norm{\phi - \sigma}_{\Lx 4} \\
		& \quad \le C \norm{\phi}_H \norm{\reactiontwo}_V ( \norm{\phi}_V + \norm{\sigma}_V ) \\
		& \quad \le \frac{c_0}{4} \norm{\nabla \phi}^2_H + \mezzo \norm{\nabla \sigma}^2_H + C (1 + \norm{\reactiontwo}^2_V) \norm{\phi}^2_H + C \norm{\sigma}^2_H, \\
		& (\hh(\phi_1) u - (\hh(\phi_1) - \hh(\phi_2)) u_2, \phi)_H \le C \norm{u}_H \norm{\phi}_H + \norm{u_2}_{\Lqt\infty} \norm{\hh(\phi_1) - \hh(\phi_2)}_H \norm{\phi}_H \\
		& \quad \le C \norm{u}^2_H + C \norm{\phi}^2_H, \\
		& (v,\sigma)_H \le \mezzo \norm{v}^2_H + \mezzo \norm{\sigma}^2_H.
	\end{align*}
	Then, by putting all together and integrating on $(0,t)$, for any $t \in (0,T)$, from \eqref{contdep:eq1} we arrive at the estimate:
	\begin{align*}
		& \mezzo \norm{\phi(t)}^2_H + \mezzo \norm{\sigma(t)}^2_H + \frac{c_0}{2} \int_0^t \norm{\nabla \phi}^2_H \, \de s + \mezzo \int_0^t \norm{\nabla \sigma}^2_H \, \de s \\
		& \quad \le \mezzo \norm{\phi_0}^2_H + \mezzo \norm{\sigma_0}^2_H + C \int_0^T (1 + \norm{\nabla \phi_2}^4_{\Lx 6} + \norm{\reactiontwo}^2_V) \norm{\phi}^2_H \, \de s \\
		& \qquad + C \int_0^T \norm{\sigma}^2_H \, \de s + \int_0^T \norm{u}^2_H \, \de s + \int_0^T \norm{v}^2_H \, \de s, 
	\end{align*}
	where $\norm{\nabla \phi_2}^4_{\Lx 6} \in \Lt\infty$, since $\phi_2 \in \C 0 {\Wx{1,6}}$ by Theorem \ref{strong:sols}, and also $\norm{\reactiontwo}^2_V \in \Lt \infty$, again by Theorem \ref{strong:sols}, given that $V \hookrightarrow \Wx{1,6}$. Therefore, we can apply Gronwall's inequality to deduce the following continuous dependence estimate:
	\begin{equation}
		\label{contdep:est1}
		\begin{split}
			& \norm{\phi}^2_{\LT \infty H} + \norm{\sigma}^2_{\LT \infty H} + \norm{\phi}^2_{\LT 2 V} + \norm{\sigma}^2_{\LT 2 V} \\
			& \quad \le C \left( \norm{u_1 - u_2}^2_{\LT 2 H} + \norm{v_1 - v_2}^2_{\LT 2 H} + \norm{{\phi_0}_1 - {\phi_0}_2}^2_H + \norm{{\sigma_0}_1 - {\sigma_0}_2}^2_H \right),
		\end{split}
	\end{equation}
	where $C \gs 0$ depends only on the parameters of the system and on the norms of the data $\{ (u_i, v_i, {\phi_0}_i, {\sigma_0}_i) \}_{i=1,2}$, but not on their difference. Now we argue essentially by comparison to get the same estimates also on $\mu$. Indeed, by estimating the $H$-norm of $\mu$ through \eqref{eq:mu2} and by using  hypothesis \ref{ass:j} and the local Lipschitz continuity of $F'$, we get that 
	\begin{align*}
		\norm{\mu}_H \le A \norm{F'(\phi_1) - F'(\phi_2)}^2_H + 2Ba^* \norm{\phi}_H \le C \norm{\phi}_H.  
	\end{align*}
	In the same way, we can also compute the $H$-norm of $\nabla \mu$ through \eqref{eq:mu3} and, by exploiting hypothesis \ref{ass:j}, H\"older's inequality, the embedding $V \hookrightarrow \Lx 4$, the local Lipschitz continuity of $F''$ and Remark \ref{rmk:normsfp_infty}, we see that
	\begin{align*}
		\norm{\nabla \mu}_H & \le A \norm{F''(\phi_1) \nabla \phi}_H + A \norm{(F''(\phi_1) - F''(\phi_2)) \nabla \phi_2}_H + B \norm{a \nabla \phi}_H  \\
		& \quad + B \norm{\nabla a \phi}_H + B \norm{\nabla J \ast \phi}_H \\
		& \le A \norm{F''(\phi_1)}_{\Lx \infty} \norm{\nabla \phi}_H + A \norm{F''(\phi_1) - F''(\phi_2)}_{\Lx 4}  \norm{\nabla \phi_2}_{\Lx 4} \\
		& \quad + B\bstar \norm{\nabla \phi}_H + 2B\bstar \norm{\phi}_H \\
		& \le C \norm{\phi}_V + C \norm{\phi}_{\Lx 4} \norm{\nabla \phi_2}_{\Lx 4} \\
		& \le C ( 1 + \norm{\nabla \phi_2}_{\Lx 4} ) \norm{\phi}_V \le C \norm{\phi}_V,
	\end{align*} 
	where $\norm{\nabla \phi_2}_{\Lx 4} \in \Lt\infty$, since $\phi_2 \in \C 0 {\Wx{1,6}}$ by Theorem \ref{strong:sols}  and $\Lx 6 \hookrightarrow \Lx 4$. Then, directly from \eqref{contdep:est1}, we also infer that 
	\begin{equation}
		\label{contdep:muenergy}
		\begin{split}
		& \norm{\mu}^2_{\LT \infty H \cap \LT 2 V} \le \\
		& \quad \le C \left( \norm{u_1 - u_2}^2_{\LT 2 H} + \norm{v_1 - v_2}^2_{\LT 2 H}  + \norm{{\phi_0}_1 - {\phi_0}_2}^2_H + \norm{{\sigma_0}_1 - {\sigma_0}_2}^2_H  \right).
		\end{split}
	\end{equation}
	
	For the second estimate, we now test \eqref{eq:sigma3} in $H$ by $\sigma_t - \Delta \sigma$ and we get: 
	\begin{align*}
		& \norm{\sigma_t}^2_H + \ddt \norm{\nabla \sigma}^2_H + \norm{\Delta \sigma}^2_H = - ( P(\phi_1) (\reactiondc), \sigma_t - \Delta \sigma )_H \\
		& \quad - ( (P(\phi_1) - P(\phi_2)) (\reactiontwo), \sigma_t - \Delta \sigma )_H + (u, \sigma_t - \Delta \sigma)_H.
	\end{align*}
	We can now easily estimate the terms on the right-hand side by using similar techniques to the ones used before, indeed we have that 
	\begin{align*}
		& ( P(\phi_1) (\reactiondc), \sigma_t - \Delta \sigma )_H \le \norm{P(\phi_1)}_{\Lx \infty} \norm{\reactiondc}_H \norm{\sigma_t - \Delta \sigma}_H \\
		& \quad \le \frac{1}{4} \norm{\sigma_t}^2_H + \frac{1}{4} \norm{\Delta \sigma}^2_H + C \norm{\sigma}^2_H + C \norm{\mu}^2_H, \\
		& ( (P(\phi_1) - P(\phi_2)) (\reactiontwo), \sigma_t - \Delta \sigma )_H \\
		& \quad \le \norm{P(\phi_1) - P(\phi_2)}_{\Lx 4} \norm{\reactiontwo}_{\Lx 4} \norm{\sigma_t - \Delta \sigma}_H \\
		& \quad \le \frac{1}{4} \norm{\sigma_t}^2_H + \frac{1}{4} \norm{\Delta \sigma}^2_H + C \norm{\reactiontwo}^2_V \norm{\phi}^2_V, \\
		& (v, \sigma_t - \Delta \sigma)_H \le \frac{1}{4} \norm{\sigma_t}^2_H + \frac{1}{4} \norm{\Delta \sigma}^2_H + C \norm{v}^2_H.
	\end{align*}
	Then, by also integrating on $(0,t)$, for any $t \in (0,T)$, we infer that 
	\begin{align*}
		& \frac{1}{4} \int_0^t \norm{\sigma_t}^2_H \, \de s + \norm{\nabla \sigma (t)}^2_H + \frac{1}{4} \int_0^t \norm{\Delta \sigma}^2_H \, \de t \\
		& \quad \le \norm{\sigma_0}^2_V + C \int_0^T (1 + \norm{\reactiontwo}^2_V) \norm{\phi}^2_V \, \de s + C \int_0^T \norm{\sigma}^2_H \, \de s + C \int_0^T \norm{v}^2_H \, \de s.
	\end{align*}
	Hence, since $\norm{\reactiontwo}^2_V \in \Lt \infty$ by Theorem \ref{strong:sols}, we can use \eqref{contdep:est1} and apply Gronwall's inequality to conclude that 
	\begin{equation}
		\label{contdep:sigmastrong}
		\begin{split}
			& \norm{\sigma}^2_{\HT 1 H \cap \LT \infty V \cap \LT 2 W} \\
			& \quad \le C \left(  \norm{u_1 - u_2}^2_{\LT 2 H} + \norm{v_1 - v_2}^2_{\LT 2 H}  + \norm{{\phi_0}_1 - {\phi_0}_2}^2_H + \norm{{\sigma_0}_1 - {\sigma_0}_2}^2_V  \right).
		\end{split}
	\end{equation}
	
	Next, we want to prove similar continuous dependence estimates also on $\phi$ and $\mu$. To do this, we test equation \eqref{eq:phi3} by $\mu_t - B J \ast \phi_t$ and we obtain that 
	\begin{equation}
		\label{contdep:eq2}
		\begin{split}
			& (\phi_t, \mu_t - B J \ast \phi_t)_H + \mezzo \ddt \norm{\nabla \mu}^2_H - B (\nabla \mu, \nabla J \ast \phi_t)_H \\
			&  \quad = ( P(\phi_1) (\reactiondc), \mu_t - B J \ast \phi_t )_H + ( (P(\phi_1) - P(\phi_2)) (\reactiontwo), \mu_t - B J \ast \phi_t )_H \\
			& \qquad - (\hh(\phi_1) u - (\hh(\phi_1) - \hh(\phi_2)) u_2, \mu_t - B J \ast \phi_t)_H,
		\end{split}
	\end{equation}
	where we observe that, by taking the time derivative of \eqref{eq:mu3} and adjusting some terms, we have that
	\begin{equation}
		\label{contdep:mut}
		\mu_t - B J \ast \phi_t = A F''(\phi_1) \phi_t + Ba \phi_t + A ( F''(\phi_1) - F''(\phi_2) ) {\phi_2}_t.
	\end{equation}
	The reason why we test by $\mu_t - B J \ast \phi_t$ is twofold. Firstly, we want to be able to get a positive term on $\phi_t$ out of $(\phi_t, \mu_t - B J \ast \phi_t)_H$ and, secondly, we also want to control the extra term $-B(\nabla \mu, \nabla J \ast \phi_t)_H$. Indeed, by using hypothesis \ref{ass:fc0}, H\"older and Young's inequalities, the local Lipschitz continuity of $F''$ and Remark \ref{rmk:normsfp_infty}, we infer that 
	\begin{align*}
		(\phi_t, \mu_t - B J \ast \phi_t)_H & = ( (A F''(\phi_1) + Ba) \phi_t , \phi_t )_H + A ( ( F''(\phi_1) - F''(\phi_2) ) {\phi_2}_t, \phi_t )_H \\
		& \ge c_0 \norm{\phi_t}^2_H - C \norm{F''(\phi_1) - F''(\phi_2)}_{\Lx 3} \norm{{\phi_2}_t}_{\Lx 6} \norm{\phi_t}_H \\
		& \ge c_0  \norm{\phi_t}^2_H - C \norm{\phi}_V \norm{{\phi_2}_t}_{\Lx 6}\norm{\phi_t}_H \\
		& \ge \frac{3 c_0}{4} \norm{\phi_t}^2_H - C \norm{{\phi_2}_t}^2_{\Lx 6}\norm{\phi}^2_V, 
	\end{align*}
	where we recall that $\norm{{\phi_2}_t}^2_{\Lx 6} \in \Lt 3$ because ${\phi_2}_t$ is bounded in $\LT 6 {\Lx6}$ by Theorem \ref{strong:sols}. Moreover, we can also estimate by standard means the other term:
	\[ B(\nabla \mu, \nabla J \ast \phi_t)_H \le B\bstar \norm{\nabla \mu}_H \norm{\phi_t}_H \le \frac{c_0}{8} \norm{\phi_t}^2_H + C \norm{\mu}^2_V. \]
	Next, we estimate the three terms on the right-hand side of \eqref{contdep:eq2} by using again \eqref{contdep:mut}. Indeed, by using a combination of H\"older and Young's inequalities, Sobolev embeddings, the local Lipschitz continuity of $F''$, $P$ and $\hh$ and Remark \ref{rmk:normsfp_infty}, we obtain that 
	{\allowdisplaybreaks
	\begin{align*}
		& ( P(\phi_1) (\reactiondc), \mu_t - B J \ast \phi_t )_H \\
		& \quad = ( P(\phi_1) (\reactiondc), (A F''(\phi_1) + Ba) \phi_t + A ( F''(\phi_1) - F''(\phi_2) ) {\phi_2}_t )_H \\
		& \quad \le C \norm{\reactiondc}_H \norm{\phi_t}_H + C \norm{\reactiondc}_{\Lx 4} \norm{F''(\phi_1) - F''(\phi_2)}_{\Lx 4} \norm{{\phi_2}_t }_H \\
		& \quad \le C \norm{\reactiondc}_H \norm{\phi_t}_H + C \norm{\reactiondc}_V \norm{\phi}_V \norm{{\phi_2}_t }_H \\
		& \quad \le \frac{c_0}{8} \norm{\phi_t}^2_H + C ( 1 + \norm{{\phi_2}_t }^2_H )\norm{\phi}^2_V + \norm{\sigma}^2_V + \norm{\mu}^2_V, \\
		& ( (P(\phi_1) - P(\phi_2)) (\reactiontwo), \mu_t - B J \ast \phi_t )_H \\
		& \quad = ( (P(\phi_1) - P(\phi_2)) (\reactiontwo), (A F''(\phi_1) + Ba) \phi_t + A ( F''(\phi_1) - F''(\phi_2) ) {\phi_2}_t )_H \\
		& \quad \le C \norm{P(\phi_1) - P(\phi_2)}_{\Lx 4} \norm{\reactiontwo}_{\Lx 4} \norm{\phi_t}_H \\
		& \qquad + C \norm{P(\phi_1) - P(\phi_2)}_{\Lx 6} \norm{\reactiontwo}_{\Lx 6} \norm{F''(\phi_1) - F''(\phi_2)}_{\Lx 6} \norm{{\phi_2}_t}_H \\
		& \quad \le C \norm{\phi}_V \norm{\reactiontwo}_V \norm{\phi_t}_H + C \norm{\reactiontwo}_V \norm{{\phi_2}_t}_H \norm{\phi}^2_V \\
		& \quad \le \frac{c_0}{8} \norm{\phi_t}^2_H + C \left( \norm{\reactiontwo}^2_V + \norm{\reactiontwo}_V \norm{{\phi_2}_t}_H \right) \norm{\phi}^2_V, \\
		& (\hh(\phi_1) u - (\hh(\phi_1) - \hh(\phi_2)) u_2, \mu_t - B J \ast \phi_t)_H \\
		& \quad \le (\hh(\phi_1) u - (\hh(\phi_1) - \hh(\phi_2)) u_2, (A F''(\phi_1) + Ba) \phi_t + A ( F''(\phi_1) - F''(\phi_2) ) {\phi_2}_t )_H \\
		& \quad \le C \norm{u}_H \norm{\phi_t}_H + C \norm{u}_H \norm{F''(\phi_1) - F''(\phi_2)}_{\Lx 3} \norm{ {\phi_2}_t}_{\Lx 6} + \\
		& \qquad + C \norm{\hh(\phi_1) - \hh(\phi_2)}_H  \norm{\phi_t}_H  +  \norm{\hh(\phi_1) - \hh(\phi_2)}_{\Lx 4} \norm{F''(\phi_1) - F''(\phi_2)}_{\Lx 4} \norm{{\phi_2}_t}_H \\
		& \quad \le \frac{c_0}{8} \norm{\phi_t}^2_H  + C \norm{u}^2_H + \left( 1 + \norm{ {\phi_2}_t}_{\Lx 6}^2 +  \norm{{\phi_2}_t}_H \right) \norm{\phi}^2_V.
	\end{align*} }%
	Then, by collecting all the estimates, starting from \eqref{contdep:eq2}, we get
	\begin{align*}
		& \frac{c_0}{4} \norm{\phi_t}^2_H + \mezzo \ddt \norm{\nabla \mu}^2_H \le C \norm{\mu}^2_V + C \norm{\sigma}^2_V \\
		& \quad + C \left(1 + \norm{{\phi_2}_t}^2_{\Lx 6} + \norm{\reactiontwo}^2_V + \norm{\reactiontwo}_V \norm{{\phi_2}_t}_H + \norm{{\phi_2}_t}_H \right) \norm{\phi}^2_V, 
	\end{align*}
	where we observe that $\norm{{\phi_2}_t}^2_{\Lx 6} \in \Lt 3$, $\norm{\reactiontwo}^2_V \in \Lt \infty$ and $\norm{{\phi_2}_t}_H \in \Lt 6$ by Theorem \ref{strong:sols}. Moreover, we can see that, by rewriting the final line of \eqref{contdep:est_gradphimu} and using Cauchy-Schwarz and Young's inequalities on the left-hand side, we have the estimate
	\[ \norm{\nabla \phi}^2_H \le C \norm{\nabla \mu}^2_H + C \left(1 + \norm{\nabla \phi_2}^4_{\Lx 6} \right) \norm{\phi}^2_H, \]
	which then, since $ \norm{\nabla \phi_2}^4_{\Lx 6} $ is uniformly bounded in $\Lt \infty$, implies that 
	\begin{equation}
		\label{contdep:eq3}
		\norm{\phi}^2_V \le C \norm{\nabla \mu}^2_H + C \norm{\phi}^2_H.
	\end{equation} 
	Therefore, by using \eqref{contdep:eq3} and integrating on $(0,t)$, for any $t \in (0,T)$, the previous inequality now becomes
	\begin{align*}
		& \frac{c_0}{4} \int_0^t \norm{\phi_t}^2_H \, \de s + \mezzo \norm{\nabla \mu(t)}^2_H \le \norm{\nabla \mu(0)}^2_H + \int_0^T \norm{\mu}^2_V \, \de s + \int_0^T \norm{\sigma}^2_V \, \de s \\
		& \quad + C \int_0^T \left(1 + \norm{{\phi_2}_t}^2_{\Lx 6} + \norm{\reactiontwo}^2_V + \norm{\reactiontwo}_V \norm{{\phi_2}_t}_H + \norm{{\phi_2}_t}_H \right) \norm{\nabla \mu}^2_H \, \de s \\
		& \quad + C \int_0^T \left(1 + \norm{{\phi_2}_t}^2_{\Lx 6} + \norm{\reactiontwo}^2_V + \norm{\reactiontwo}_V \norm{{\phi_2}_t}_H + \norm{{\phi_2}_t}_H \right) \norm{\phi}^2_H \, \de s,
	\end{align*}
	where the expression between the parentheses is integrable in time, due to the previous remarks. Moreover, since ${\phi_0}_i \in \Hx 2 \hookrightarrow \Lx \infty$, $i =1,2$, we can also estimate
	\begin{align*}
		\norm{\nabla \mu (0)}_H & = A \norm{F''({\phi_0}_1) - F''({\phi_0}_2) \nabla {\phi_0}_2}_H + A \norm{F''({\phi_0}_1) \nabla \phi_0}_H \\
		& \quad + B \norm{a \nabla \phi_0}_H + B \norm{\nabla a \phi_0}_H + B \norm{\nabla J \ast \phi_0}_H \\
		& \le C \norm{{\phi_0}_1}_{\Hx 2} \norm{\phi_0}_{\Lx 4} + C \norm{\phi_0}_V \le C \norm{\phi_0}_V.
	\end{align*}
	Then, by using Gronwall's inequality, together with \eqref{contdep:est1} and \eqref{contdep:muenergy}, we infer that 
	\begin{equation}
		\label{contdep:phih1h}
		\begin{split}
			& \norm{\phi}^2_{\HT 1 H} + \norm{\mu}^2_{\LT \infty V} \\ 
			& \quad \le C \left( \norm{u_1 - u_2}^2_{\LT 2 H} + \norm{v_1 - v_2}^2_{\LT 2 H}  + \norm{{\phi_0}_1 - {\phi_0}_2}^2_V + \norm{{\sigma_0}_1 - {\sigma_0}_2}^2_V   \right).
		\end{split}
	\end{equation}
	Moreover, by comparison with \eqref{contdep:mut} and \eqref{contdep:eq3}, one can also easily see that 
	\begin{equation}
		\label{contdep:philinfv}
		\begin{split}
			& \norm{\mu}^2_{\HT 1 H} + \norm{\phi}^2_{\LT \infty V} \\ 
			& \quad \le C \left( \norm{u_1 - u_2}^2_{\LT 2 H} + \norm{v_1 - v_2}^2_{\LT 2 H}  + \norm{{\phi_0}_1 - {\phi_0}_2}^2_V + \norm{{\sigma_0}_1 - {\sigma_0}_2}^2_V \right).
		\end{split}
	\end{equation}
	Next, we test \eqref{eq:phi3} by $- \Delta \mu$ in $H$ and we get:
	\begin{align*}
		\norm{\Delta \mu}^2_H & = (\phi_t, \Delta \mu)_H - (P(\phi_1) (\reactiondc), \Delta \mu)_H - ((P(\phi_1) - P(\phi_2)) (\reactiontwo), \Delta \mu)_H \\
		& \quad - (\hh(\phi_1) u - (\hh(\phi_1) - \hh(\phi_2)) u_2, \Delta \mu )_H.
	\end{align*}
	Then, by using Cauchy-Schwarz and Young's inequality, together with the local Lipschitz continuity of $P$ and $\hh$ and Remark \ref{rmk:normsfp_infty}, we infer that 
	\begin{align*}
		\norm{\Delta \mu}^2_H & \le \mezzo \norm{\Delta \mu}^2_H + \mezzo \norm{\phi_t}^2_H + C \norm{\sigma}^2_H + C \norm{\mu}^2_H \\
		& \quad + C \left(1 + \norm{\reactiontwo}^2_{\Lx \infty} \right) \norm{\phi}^2_H + C \norm{u}^2_H.
	\end{align*}
	Hence, we can integrate on $(0,T)$ and use \eqref{contdep:est1}, \eqref{contdep:muenergy} and \eqref{contdep:phih1h}, together with the fact that $\norm{\reactiontwo}^2_{\Lt\infty} \in \Lt \infty$ by Theorem \ref{strong:sols}, to deduce that 
	\begin{equation}
		\label{contdep:mul2w}
		\begin{split}
		& \norm{\mu}^2_{\LT 2 W} \\
		& \quad \le C \left( \norm{u_1 - u_2}^2_{\LT 2 H} + \norm{v_1 - v_2}^2_{\LT 2 H}  + \norm{{\phi_0}_1 - {\phi_0}_2}^2_V + \norm{{\sigma_0}_1 - {\sigma_0}_2}^2_V  \right).
		\end{split}
	\end{equation}
	Finally, for any $i,j = 1,2,3$, we apply the differential operator $\partial_{x_i x_j}$ to \eqref{eq:mu3}, which makes sense in $H$, and we test the resulting equation by $\partial_{x_i x_j} \phi$. Then, after careful rewriting of the terms arising from the derivatives of $F$, up to adding and subtracting some of them, we get:
	\begin{align*}
		& (\partial_{x_i x_j} \mu, \partial_{x_i x_j} \phi)_H = ( (AF''(\phi_1) + Ba) \, \partial_{x_i x_j} \phi, \partial_{x_i x_j} \phi )_H + ( (F''(\phi_1) - F''(\phi_2)) \partial_{x_i x_j} \phi_2, \partial_{x_i x_j} \phi )_H  \\
		& \quad + ( F'''(\phi_1) (\partial_{x_i} \phi_1 + \partial_{x_i} \phi_2) \, \partial_{x_j} \phi, \partial_{x_i x_j} \phi)_H + ( (F'''(\phi_1) - F'''(\phi_2)) \partial_{x_i} \phi_2 \, \partial_{x_j} \phi_2, \partial_{x_i x_j} \phi )_H \\
		& \quad + ( B (\partial_{x_i} a \, \partial_{x_j} \phi + \partial_{x_j} a \, \partial_{x_i} \phi), \partial_{x_i x_j} \phi )_H + (B \partial_{x_i x_j} a \, \phi, \partial_{x_i x_j} \phi)_H - B ( \partial_{x_i} (\partial_{x_j} J \ast \phi), \partial_{x_i x_j} \phi )_H.
	\end{align*}
	Hence, by using hypotheses \ref{ass:fc0}, \ref{ass:j2}, Remark \ref{admissible}, Remark \ref{rmk:normsfp_infty}, the local Lipschitz continuity of $F''$ and $F'''$, Sobolev embeddings and H\"older, Young, Gagliardo-Nirenberg and Agmon's inequalities (see \ref{gn:ineq} and \ref{agmon}), we can estimate:
	{\allowdisplaybreaks
	\begin{align*}
		& c_0 \norm{\partial_{x_i x_j} \phi}^2_H \le \frac{c_0}{4} \norm{\partial_{x_i x_j} \phi}^2_H + C \norm{\partial_{x_i x_j} \mu}^2_H + C \norm{\phi}_{\Lx \infty}  \norm{\phi_2}_{\Hx 2} \norm{\partial_{x_i x_j} \phi}_H \\
		& \qquad + C \norm{\nabla \phi_1 + \nabla \phi_2}_{\Lx 6} \norm{\nabla \phi}_{\Lx 3} \norm{\partial_{x_i x_j} \phi}_H + C \norm{\phi}_{6} \norm{\nabla \phi_2}^2_{6} \norm{\partial_{x_i x_j} \phi}_H \\
		& \qquad +  \norm{\partial_{x_i x_j} a}_{\Lx4} \norm{\phi}_{\Lx4} \norm{\partial_{x_i x_j} \phi}_H + C \norm{\nabla \phi}^2_H + C \norm{\phi}^2_H \\
		& \quad \le \frac{c_0}{4} \norm{\partial_{x_i x_j} \phi}^2_H + C \norm{\mu}^2_W + C \norm{\phi_2}_{\Hx 2} \norm{\phi}_V^{1/2} \norm{\phi}_{\Hx 2}^{3/2} \\
		& \qquad + C \norm{\nabla \phi_1 + \nabla \phi_2}_{\Lx 6} \norm{\nabla \phi}^{1/2}_H \norm{\phi}^{3/2}_{\Hx 2} + C \norm{\phi}_V \norm{\nabla \phi_2}^2_{\Lx 6} \norm{\partial_{x_i x_j} \phi}_H \\ 
		& \qquad + C \norm{\phi}_V \norm{\partial_{x_i x_j} \phi}_H + C \norm{\phi}^2_V \\
		& \quad \le  \frac{c_0}{4} \norm{\partial_{x_i x_j} \phi}^2_H +  \frac{c_0}{4} \norm{\phi}^2_{\Hx 2} + C \norm{\mu}^2_W \\
		& \qquad + C \left( 1 + \norm{\phi_2}^4_{\Hx 2} + \norm{\nabla \phi_1 + \nabla \phi_2}^4_{\Lx 6} + \norm{\nabla \phi_2}^4_{\Lx 6} \right) \norm{\phi}^2_V,
	\end{align*}
	}
	where $\norm{\phi_2}^4_{\Hx 2} + \norm{\nabla \phi_1 + \nabla \phi_2}^4_{\Lx 6} + \norm{\nabla \phi_2}^4_{\Lx 6} \in L^{3/2} (0,T)$, since, by Theorem \ref{strong:sols}, $\phi_2 \in \C 0 {\Wx{1,6}} \cap \LT 6 {\Wx{2,6}}$, and $\Hx 2 \hookrightarrow \Wx{1,6}$. Therefore, by summing on $i,j =1,2,3$, we obtain that 
	\begin{align*}
		\frac{c_0}{2} \norm{\phi}^2_{\Hx 2} \le C \norm{\mu}^2_W + C \left( 1 + \norm{\phi_2}^4_{\Hx 2} + \norm{\nabla \phi_1 + \nabla \phi_2}^4_{\Lx 6} + \norm{\nabla \phi_2}^4_{\Lx 6} \right) \norm{\phi}^2_V,
	\end{align*}
	starting from which, by integrating on $(0,T)$ and using the previous estimates \eqref{contdep:philinfv} and \eqref{contdep:mul2w}, we conclude that 
	\begin{equation}
		\label{contdep:phil2w}
		\begin{split}
		& \norm{\phi}^2_{\LT 2 {\Hx 2}} \\
		& \le C \left( \norm{u_1 - u_2}^2_{\LT 2 H} + \norm{v_1 - v_2}^2_{\LT 2 H}  + \norm{{\phi_0}_1 - {\phi_0}_2}^2_V + \norm{{\sigma_0}_1 - {\sigma_0}_2}^2_V  \right).
		\end{split}
	\end{equation}
	This concludes the proof of Theorem \ref{thm:contdepstrong}.
\end{proof}

\section{Optimal control problem}
\label{sect:oc}

From now on, we consider the initial data $\phi_0$ and $\sigma_0$, satisfying \eqref{hp:initaldata}, fixed. As an application of the strong well-posedness that we were able to prove in Theorems \ref{strong:sols} and \ref{thm:contdepstrong}, we consider the control problem (CP), which we recall below:

\bigskip
\noindent(CP) \textit{Minimise the cost functional}
\begin{equation*}
	\begin{split}
		\mathcal{J}(\phi, \sigma, u, v) & = \, \frac{\alpha_{\Omega}}{2} \int_{\Omega} |\phi(T) - \phi_{\Omega}|^2 \,\de x  + \frac{\alpha_Q}{2} \int_{0}^{T} \int_{\Omega} |\phi - \phi_Q|^2 \,\de x \,\de t \\
		& \quad + \frac{\beta_{\Omega}}{2} \int_{\Omega} |\sigma(T) - \sigma_{\Omega}|^2 \,\de x  + \frac{\beta_Q}{2} \int_{0}^{T} \int_{\Omega} |\sigma - \sigma_Q|^2 \,\de x \,\de t \\ 
		& \quad + \frac{\alpha_u}{2} \int_{0}^{T} \int_{\Omega} |u|^2 \,\de x \,\de t + \frac{\beta_v}{2} \int_{0}^{T} \int_{\Omega} |v|^2 \,\de x \,\de t,
	\end{split}
\end{equation*}
\textit{subject to the control constraints}
\begin{equation*} 
	\begin{split}
	& u \in \Uad := \{ u \in L^{\infty}(Q_T) \mid u_{\text{min}} \le u \le u_{\text{max}} \text{ a.e.~in } Q_T \}, \\
	& v \in \Vad := \{ v \in L^{\infty}(Q_T) \mid v_{\text{min}} \le v \le v_{\text{max}} \text{ a.e.~in } Q_T \},
	\end{split}
\end{equation*}
\textit{and to the state system \eqref{eq:phi2}-\eqref{ic2}}.
\bigskip

Regarding the parameters at play, we make the following hypotheses:
\begin{enumerate}[font = \bfseries, label = C\arabic*., ref=\bf{C\arabic*}]
	\item\label{C1} $\alpha_\Omega, \alpha_Q, \beta_\Omega, \beta_Q, \beta_u \ge 0$, but not all equal to $0$.
	\item\label{C2} $\phi_\Omega, \sigma_\Omega \in L^2(\Omega)$ and $\phi_Q, \sigma_Q \in L^2(Q_T)$.
	\item\label{C3} $u_{\text{min}}, u_{\text{max}}, v_{\text{min}}, v_{\text{max}} \in L^\infty(Q_T)$, with $u_{\text{min}} \le u_{\text{max}}$ and $v_{\text{min}} \le v_{\text{max}}$ a.e. in $\Omega$.
	\item\label{ass:ph3} $P, \hh \in \mathcal{C}^2(\R) \cap L^\infty(\R)$.
	\item\label{ass:initial3} $\phi_0, \sigma_0 \in H^2(\Omega)$ with $\partial_{\n} (AF'(\phi_0) + Ba \phi_0 - B J \ast \phi_0) = \partial_{\n} \sigma_0 = 0$ on $\partial \Omega$.
\end{enumerate}

\begin{remark}
	For modelling reasons, in practice one generally takes $\hh$ to be non-negative and $u_{\text{min}} \ge 0$, since the radiotherapy $u$ should only act to decrease the tumour proliferation.
\end{remark}

By Theorem \ref{strong:sols}, we know that for any $(u,v) \in \Uad \times \Vad$ there exists a unique strong solution $(\phi, \mu, \sigma) \in \mathbb{X}$ to \eqref{eq:phi2}--\eqref{ic2}, where 
\begin{align*}
	\mathbb{X} & := \left( W^{1,6}(0,T; L^6(\Omega)) \cap \mathcal{C}^0([0,T]; W^{1,6}(\Omega)) \cap L^6(0,T; W^{2,6}(\Omega)) \right)^3,
\end{align*}
therefore the optimal control problem (CP) is well-defined. Our goal is to prove existence of an optimal control and then find the first-order necessary optimality conditions. 
We stress that to prove such optimality conditions, we need to study the differentiability of the control-to-state operator. 
Hence, a strong continuous dependence estimate like \eqref{contdep:estimate} is necessary and we recall that, to prove it, we heavily relied on the global maximal regularity results.
We first begin with the following existence result for optimal controls.

\begin{theorem}
	\label{thm:excont}
	Assume hypotheses \emph{\ref{ass:coeff}--\ref{ass:h}}, \emph{\ref{ass:j2}--\ref{ass:f2}} and \emph{\ref{C1}}--\emph{\ref{ass:initial3}}. Then the optimal control problem (CP) admits at least one solution $(\ub, \vb) \in \Uad \times \Vad$, such that if $(\phib, \mub, \sigmab)$ is the solution to \eqref{eq:phi2}--\eqref{ic2} associated to $(\ub, \vb)$, one has that
	\begin{equation}
		\mathcal{J}(\phib, \sigmab, \ub, \vb) = \min_{(u,v) \,\in \,\Uad \times \Vad} \, \mathcal{J}(\phi, \sigma, u, v).
	\end{equation}
\end{theorem}

\begin{proof}
	The argument is standard and relies on the direct method of Calculus of Variations, therefore we omit it for the sake of brevity. For more details on the procedure, we refer the interested reader to \cite[Theorem 4.2]{F2023_viscous}. 
\end{proof}

\subsection{Linearised system}

We now want to study Fr\'echet-differentiability properties of the control-to-state operator, which maps any $(u,v) \in \Uad \times \Vad$ into the corresponding solution of the state system. The first step consists in deriving the linearised version of system \eqref{eq:phi2}-\eqref{ic2} and in proving its well-posedness, since it is generally the \emph{ansatz} for the expression of the Fr\'echet derivative. Indeed, we fix a state $(\phib, \mub, \sigmab) \in \mathbb{X}$ corresponding to $(\ub, \vb) \in \Uad \times \Vad$ and linearise near $(\ub, \vb)$:
\[ \phi = \phib + \xi, \, \mu = \mub + \eta, \, \sigma = \sigmab + \rho, \, u = \ub + h, v = \vb + k, \]
with $(h,k) \in \Lqt\infty^2$. Then, by approximating the non-linearities at the first order of their Taylor expansion, we see that $(\xi, \eta, \rho)$ satisfy the equations:
\begin{align}
	& \partial_t \xi - \Delta \eta = P'(\phib) (\reactionbar) \xi + P(\phib) (\reactionlin) - \hh'(\phib) \ub \, \xi - \hh(\phib) h && \text{in } Q_T,  \label{eq:xi} \\
	& \eta = AF''(\phib) \xi + Ba \xi - BJ \ast \xi && \text{in } Q_T,  \label{eq:eta} \\
	& \partial_t \rho - \Delta \rho = - P'(\phib) (\reactionbar) \xi - P(\phib) (\reactionlin) + k && \text{in } Q_T, \label{eq:rho}
\end{align}
together with boundary and initial conditions:
\begin{alignat}{2}
	& \partial_{\n} \eta = \partial_{\n} \rho = 0 \qquad && \text{on } \Sigma_T, \label{bcl} \\
	& \xi(0) = 0, \quad \rho(0) = 0 \qquad && \text{in } \Omega. \label{icl}
\end{alignat}

\begin{theorem}
	\label{thm:linearised}
	Assume hypotheses \emph{\ref{ass:coeff}--\ref{ass:h}}, \emph{\ref{ass:j2}--\ref{ass:f2}} and \emph{\ref{ass:ph3}--\ref{ass:initial3}}. Let $(\phib, \mub, \sigmab) \in \mathbb{X}$ be the strong solution to \eqref{eq:phi2}--\eqref{ic2}, corresponding to $(\ub, \vb) \in \Uad \times \Vad$. Then, for any $(h,k) \in L^2(0,T;H) \times L^2(0,T;H)$, the linearised system \eqref{eq:xi}--\eqref{icl} admits a unique weak solution, which is uniformly bounded in the following spaces
	\begin{align*}
		& \xi \in H^1(0,T;V^*) \cap \mathcal{C}^0([0,T]; H) \cap L^2(0,T;V), \\
		& \eta \in L^\infty(0,T;H) \cap L^2(0,T;V), \\
		& \rho \in H^1(0,T;V^*) \cap \mathcal{C}^0([0,T]; H) \cap L^2(0,T;V),
	\end{align*}
	and fulfils \eqref{eq:xi}--\eqref{icl} in variational form, i.e.~it satisfies
	\begin{align*}
		& \duality{\xi_t, w}_V + (\nabla \eta, \nabla w)_H = (P'(\phib)(\reactionbar) \xi + P(\phib)(\reactionlin) - \hh'(\phib) \ub \, \xi - \hh(\phib) h, w)_H, \\
		& (\eta,w)_H = (AF'(\phib) \xi + Ba \xi - BJ \ast \xi,w)_H, \\
		& \duality{\rho_t,w}_V + (\nabla \rho, \nabla w)_H = - (P'(\phib)(\reactionbar) \xi + P(\phib)(\reactionlin), w)_H + (k,w)_H,
	\end{align*}
	for a.e. $t \in (0,T)$ and for any $w \in V$, and $\xi(0) = 0$, $\rho(0) = 0$.
\end{theorem}

\begin{proof}
	We proceed formally, but we recall that the argument can be made rigorous by employing a Faedo-Galerkin discretisation scheme, with discrete spaces made of eigenvectors of the operator $\mathcal{N}$. Then, being the system linear, it is a standard matter to pass to the limit in the discretisation framework and recover a weak solution with the expected regularities. For the main estimate, we test equation \eqref{eq:xi} by $\xi$, \eqref{eq:rho} by $\rho$ and sum them up, to obtain:
	\begin{align*}
		& \mezzo \ddt \norm{\xi}^2_H + \mezzo \ddt \norm{\rho}^2_H + (\nabla \eta, \nabla \xi)_H + \norm{\nabla \rho}^2_H \\
		& \quad \le ( P'(\phib)(\reactionbar) \xi, \xi - \rho )_H + ( P(\phib)(\reactionlin), \xi - \rho )_H - (\hh'(\phib) \ub \, \xi + \hh(\phib) h, \xi)_H + (k,\rho)_H.
	\end{align*}
	Then, by using Remark \ref{rmk:normsfp_infty}, Cauchy-Schwarz and Young's inequalities and by recalling that $\norm{\reactionbar}_{\Lx \infty} \in \Lt \infty$ by Theorem \ref{strong:sols}, we estimate the right-hand side as 
	\begin{align*}
		& \mezzo \ddt \norm{\xi}^2_H + \mezzo \ddt \norm{\rho}^2_H + (\nabla \eta, \nabla \xi)_H + \norm{\nabla \rho}^2_H \\
		& \quad \le C \left( 1 + \norm{\reactionbar}^2_{\Lx \infty} \right) \norm{\xi}^2_H + C \norm{\eta}^2_H + C \norm{\rho}^2_H + C \norm{h}^2_H + C \norm{k}^2_H.
	\end{align*}
	Next, to close the estimate, we further test \eqref{eq:eta} by $- \Delta \xi$, which is possible within the discretisation framework, and integrate by parts, by recalling that, due to the fact that the discrete spaces are made of functions satisfying homogeneous Neumann boundary conditions, no extra boundary terms appear. Then, by exploiting H\"older, Young and Gagliardo-Nirenberg \eqref{gn:ineq} inequalities, together with Remark \ref{rmk:normsfp_infty} and hypotheses \ref{ass:fc0} and \ref{ass:j}, we infer that 
	\begin{align*}
		(\nabla \eta, \nabla \xi)_H & = (AF'''(\phib) \nabla \phib \, \xi, \nabla \xi)_H + ( (AF''(\phib) + Ba) \nabla \xi, \nabla \xi )_H \\
		& \quad + (B \nabla a \, \xi, \nabla \xi)_H - (B \nabla J \ast \xi, \nabla \xi)_H \\
		& \ge c_0 \norm{\nabla \xi}^2_H - C \norm{\nabla \phib}_{\Lx 6} \norm{\xi}_{\Lx 3} \norm{\nabla \xi}_H - 2B\bstar \norm{\xi}_H \norm{\nabla \xi}_H \\
		& \ge c_0 \norm{\nabla \xi}^2_H - C \norm{\nabla \phib}_{\Lx 6} \norm{\xi}^{1/2}_H \norm{\xi}^{3/2}_V - 2B\bstar \norm{\xi}_H \norm{\nabla \xi}_H \\
		& \ge \frac{c_0}{2} \norm{\nabla \xi}^2_H + C \left( 1 + \norm{\nabla \phib}^4_{\Lx 6} \right) \norm{\xi}^2_H,
	\end{align*}
	where $\norm{\nabla \phib}^4_{\Lx 6} \in \Lt\infty$ by Theorem \ref{strong:sols} and Sobolev embeddings. Therefore, by putting all together, we have the estimate: 
	\begin{align*}
		& \mezzo \ddt \norm{\xi}^2_H + \mezzo \ddt \norm{\rho}^2_H + \frac{c_0}{2} \norm{\nabla \xi}^2_H + \norm{\nabla \rho}^2_H \\
		& \quad \le C \left( 1 + \norm{\nabla \phib}^4_{\Lx 6} + \norm{\reactionbar}^2_{\Lx \infty} \right) \norm{\xi}^2_H + C \norm{\eta}^2_H + C \norm{\rho}^2_H + C \norm{h}^2_H + C \norm{k}^2_H.
	\end{align*}
	Now observe that, by comparison in equation \eqref{eq:eta}, thanks to Remark \ref{rmk:normsfp_infty} and hypothesis \ref{ass:j}, one can easily see that
	\begin{equation}
		\label{linear:etaxi}
		 \norm{\eta}^2_H \le C \norm{\xi}^2_H. 
	\end{equation}
	Hence, by also integrating on $(0,t)$, for any $t \in (0,T)$, from the previous inequality we deduce that
	\begin{align*}
		& \mezzo \norm{\xi (t)}^2_H + \mezzo \norm{\rho(t)}^2_H + \frac{c_0}{2} \int_0^t \norm{\nabla \xi}^2_H \, \de s + \int_0^t \norm{\nabla \rho}^2_H \, \de s \\
		& \quad \le C \int_0^T \left( 1 + \norm{\nabla \phib}^4_{\Lx 6} + \norm{\reactionbar}^2_{\Lx \infty} \right) \norm{\xi}^2_H \, \de s + C \int_0^T \norm{\rho}^2_H + \norm{h}^2_H + \norm{k}^2_H \, \de s,
	\end{align*}
	which, by Gronwall's lemma, implies the following uniform estimate:
	\begin{equation}
		\label{linear:est1}
		\norm{\xi}^2_{\LT \infty H \cap \LT 2 V} + \norm{\rho}^2_{\LT \infty H \cap \LT 2 V} \le C \left( \norm{h}^2_{\LT 2 H} + \norm{k}^2_{\LT 2 H} \right),
	\end{equation}
	with $C \gs 0$ depending only on the parameters of the system. 
	Next, by testing \eqref{eq:eta} by $- \Delta \eta$, integrating by parts and using H\"older and Young's inequalities, we obtain that
	\begin{align*}
		\norm{\nabla \eta}^2_H & = (AF'''(\phib) \nabla \phib \, \xi, \nabla \eta)_H + ( (AF''(\phib) + Ba) \nabla \xi, \nabla \eta )_H \\
		& \quad + (B \nabla a \, \xi, \nabla \eta)_H - (B \nabla J \ast \xi, \nabla \eta)_H \\
		& \le \mezzo \norm{\nabla \eta}^2_H + C \underbrace{\norm{\nabla \phib}^2_{\Lx 4}}_{\in \Lt \infty} \norm{\xi}^2_{\Lx 4} + C \norm{\xi}^2_V \le \mezzo \norm{\nabla \eta}^2_H + C \norm{\xi}^2_V,
	\end{align*}
	which, by integrating on $(0,T)$ and applying \eqref{linear:est1}, together with \eqref{linear:etaxi}, implies that 
	\[ \norm{\eta}^2_{\LT \infty H \cap \LT 2 V} \le C \left( \norm{h}^2_{\LT 2 H} + \norm{k}^2_{\LT 2 H} \right). \]
	Finally, by comparison in \eqref{eq:xi} and \eqref{eq:rho}, it also follows that 
	\[ \norm{\xi}^2_{\HT 1 {V^*}} + \norm{\rho}^2_{\HT 1 {V^*}} \le C \left( \norm{h}^2_{\LT 2 H} + \norm{k}^2_{\LT 2 H} \right). \]
	In the end, by also using standard embeddings of Lebesgue-Bochner spaces on a Hilbert triplet, we have shown that there exists a constant $C \gs 0$, depending only on the parameters of the system, such that
	\begin{equation}
		\label{linear:estweak}
		\begin{split}
			& \norm{\xi}^2_{\HT 1 {V^*} \cap \C 0 H \cap \LT 2 V} + \norm{\eta}^2_{\LT \infty H \cap \LT 2 V} \\
			& \quad +  \norm{\rho}^2_{\HT 1 {V^*} \cap \C 0 H \cap \LT 2 V} \le C \left( \norm{h}^2_{\LT 2 H} + \norm{k}^2_{\LT 2 H} \right).
		\end{split}
	\end{equation}
	With this estimate, it is a standard matter to pass to the limit in the discretisation and show the existence of a weak solution to \eqref{eq:xi}--\eqref{icl}. Moreover, due to the linearity of the system, this same estimate also gives uniqueness of the solution. This concludes the proof of Theorem \ref{thm:linearised}.
\end{proof}

\begin{remark}
	We observe that the linearised system \eqref{eq:xi}--\eqref{icl} is exactly the same as the abstract linearised system \eqref{eq:abslin} with $\vpsi = (\phib, \sigmab)^\top$, $\vec{\upxi} = (\xi, \rho)^\top$, $\vec{f} = (-\hh(\phib) h, k)^\top$ and $\vec{g} = \vec{0}$. This means that, since $(h,k)$ can be taken in $\Lqt\infty^2$ and $\phib$ can be embedded into $\Cqt0$, one could follow the same argument used in the first part of the proof of Theorem \ref{local:maxsol}, to say that there exists a unique maximal solution with regularity 
	\[ \xi, \rho \in W^{1,p}(0,T; \Lx p) \cap L^p(0,T;\Wx{2,p}), \quad \text{for any $p \gs N+2$.} \]
	This can clearly be done in place of the previous proof, but the actual regularity that we need on the linearised system to study the optimal control problem is way less that the one guaranteed by maximal regularity theory. This is the reason we also provided the proof above. 
\end{remark}

\subsection{Differentiability of the control-to-state operator}

In order to study the \emph{control-to-state operator} $\mathcal{S}$, which associates to any control $(u,v) \in \Uad \times \Vad$ the corresponding solution of the system \eqref{eq:phi2}--\eqref{ic2}, we introduce the following spaces:
\begin{align*}
	\mathbb{Y} & := (H^1(0,T;H) \cap L^\infty(0,T;V) \cap L^2(0,T;\Hx 2))^3, \\
	\mathbb{W} & := (H^1(0,T;V^*) \cap \C 0 H \cap L^2(0,T;V)) \times (L^\infty(0,T;H) \cap L^2(0,T;V)) \\
	& \qquad \times (H^1(0,T;V^*) \cap \C 0 H \cap L^2(0,T;V)).
\end{align*}
Observe that the space of strong solutions $\mathbb{X}$ is continuously embedded into $\mathbb{Y}$, which is exactly the space where we proved the continuous dependence estimates. Indeed, from Theorem \ref{strong:sols} and Theorem \ref{thm:contdepstrong} we respectively know that
\[ \mathcal{S}: L^\infty(Q_T)^2 \to \mathbb{X} \quad \hbox{is well-defined and}  \]
\[ \mathcal{S}: L^\infty(Q_T)^2 \to \mathbb{Y} \quad \hbox{is locally Lipschitz-continuous.}  \]
Now, for $R>0$, we fix an open set $\mathcal{U}_R \times \mathcal{V}_R \subseteq L^\infty(Q_T)^2$ such that $\Uad \times \Vad \subseteq \mathcal{U}_R \times \mathcal{V}_R$. Indeed, by hypothesis \ref{C3}, we can take:
\begin{align*}
	& \mathcal{U}_R := \{ u \in L^{\infty}(Q_T) \mid \norm{u}_{L^\infty(Q_T)} < M_u + R \},
	& \mathcal{V}_R := \{ v \in L^{\infty}(Q_T) \mid \norm{v}_{L^\infty(Q_T)} < M_v + R \},
\end{align*}
where $M_u = \norm{u_{\text{max}}}_\infty$ and $M_v = \norm{v_{\text{max}}}_\infty$.
Note that, in $\mathcal{U}_R \times \mathcal{V}_R$, the continuous dependence estimate of Theorem \ref{thm:contdepstrong} holds with $K$ depending only on $R$ and the fixed data of the system.
Our aim is to show that $\mathcal{S}: \mathcal{U}_R \times \mathcal{V}_R \to \mathbb{W}$ is also Fr\'echet-differentiable in the larger space $\mathbb{W}$.
Indeed, we can prove the following theorem:

\begin{theorem}
	\label{thm:frechet}
	Assume hypothesis \emph{\ref{ass:coeff}--\ref{ass:h}}, \emph{\ref{ass:j2}--\ref{ass:f2}} and \emph{\ref{ass:ph3}--\ref{ass:initial3}}. Then $\mathcal{S}: \mathcal{U}_R \times \mathcal{V}_R \to \mathbb{W}$ is Fr\'echet-differentiable, i.e. for any $(\ub, \vb) \in \mathcal{U}_R \times \mathcal{V}_R$ there exists a unique Fr\'echet-derivative $D\mathcal{S}(\ub, \vb) \in \mathcal{L}(L^\infty(Q_T)^2, \mathbb{W})$ such that:
	\begin{equation}
		\label{frechet:diff}
		\frac{ \norm{ \mathcal{S}(\ub+h, \vb +k) - \mathcal{S}(\ub, \vb) - D\mathcal{S}(\ub, \vb)[h,k]  }_{\mathbb{W}} }{ \norm{(h,k)}_{L^2(Q_T)^2} } \to 0 \quad \text{as } \norm{(h,k)}_{L^2(Q_T)^2} \to 0.
	\end{equation}
	Moreover, for any $(h,k) \in L^\infty(Q_T)^2$, the Fr\'echet-derivative at $(\ub, \vb)$ in $(h,k)$, which we denote by $D\mathcal{S}(\ub, \vb)[h,k]$, is defined as the solution $(\xi, \eta, \rho)$ to the linearised system \eqref{eq:xi}--\eqref{icl} corresponding to $(\phib, \mub, \sigmab) = \mathcal{S}(\ub,\vb)$, with data $h$ and $k$. 
\end{theorem}

\begin{remark}
	Note that, by Theorem \ref{thm:linearised}, $D\mathcal{S}(\ub, \vb)$ as defined above actually belongs to the space of continuous linear operators $\mathcal{L}(L^\infty(Q_T)^2, \mathbb{W})$. Observe also that \eqref{frechet:diff} shows Fr\'echet-differentiability with respect to the $L^2(Q_T)$ norm, but clearly, since $L^\infty(Q_T) \hookrightarrow L^2(Q_T)$, this also implies Fr\'echet-differentiability in the correct space.
\end{remark}

\begin{proof}
	We observe that it is sufficient to prove the result for any small enough perturbation $(h,k)$, i.e.~we fix $\Lambda > 0$ and consider only perturbations such that
	\begin{equation}
		\label{h:bound}
		\norm{(h,k)}_{L^2(Q_T)^2} \le \Lambda.
	\end{equation} 
	Now, we fix $\ub$, $\vb$, $h$ and $k$ as above and consider
	\begin{align*}
		& (\phi, \mu, \sigma) := \mathcal{S}(\ub+h, \vb +k), \\
		& (\phib, \mub, \sigmab) := \mathcal{S}(\ub, \vb), \\
		& (\xi, \eta, \rho) \text{ as the solution to \eqref{eq:xi}--\eqref{icl} with respect to } (h,k).
	\end{align*}
	In order to show Fr\'echet-differentiability, then, it is enough to show that there exists a constant $C>0$, depending only on the parameters of the system and possibly on $\Lambda$, and an exponent $s > 2$ such that
	\[ \norm{ (\phi, \mu, \sigma) - (\phib, \mub, \sigmab) - (\xi, \eta, \rho) }^2_{\mathbb{W}} \le C \norm{(h,k)}^s_{L^2(Q_T)^2}. \]
	To do this, we introduce the additional variables
	\begin{align*}
		& \psi := \phi - \phib - \xi \in H^1(0,T;V^*) \cap \C 0 H \cap  L^2(0,T; V), \\
		& \zeta  := \mu - \mub - \eta \in L^\infty(0,T;H) \cap L^2(0,T;V), \\
		& \theta := \sigma - \sigmab - \rho \in H^1(0,T;V^*) \cap \C 0 H \cap  L^2(0,T; V), 
	\end{align*}
	which by Theorems \ref{strong:sols} and \ref{thm:linearised} enjoy the regularities shown above. Then, this is equivalent to showing that
	\begin{equation}
		\label{frechet:aim}
		\norm{ (\psi, \zeta, \theta) }_{\mathbb{W}}^2 \le C \norm{(h,k)}^s_{L^2(Q_T)^2}.
	\end{equation}
	Moreover, by inserting the equations solved by the variables in the definitions of $\psi$, $\zeta$ and $\theta$ and exploiting the linearity of the involved differential operators, we infer that these new variables formally satisfy the equations:
	\begin{alignat}{2}
		& \partial_t \psi - \Delta \zeta = Q^h - U^h \,\,\, && \text{in } Q_T,  \label{eq:psi}\\
		& \zeta =  AF^h + Ba \psi - BJ \ast \psi \quad && \text{in } Q_T,  \label{eq:zeta} \\
		& \partial_t \theta - \Delta \theta = - Q^h \quad && \text{in } Q_T, \label{eq:theta}
	\end{alignat}
	together with boundary and initial conditions:
	\begin{alignat}{2}
		& \partial_{\n} \zeta = \partial_{\n} \theta = 0 \qquad && \text{on } \Sigma_T, \label{bcd} \\
		& \psi(0) = 0, \quad \theta(0) = 0 \qquad && \text{in } \Omega, \label{icd}
	\end{alignat}
	where:
	\begin{align*}
		F^h & = F'(\phi) - F'(\phib) - F''(\phib)\xi, \\
		Q^h & = P(\phi)(\reaction) - P(\phib)(\reactionbar) - P(\phib)(\reactionlin) - P'(\phib) (\reactionbar) \xi, \\
		U^h & = \hh(\ub + h) - \hh(\phib) \ub - \hh(\phib) - \hh(\phib) \ub \xi.
	\end{align*}
	Note that, to be precise, system \eqref{eq:psi}--\eqref{icd} has to be understood in weak sense, i.e.~through a variational formulation, since only weak regularity is available for the linearised variables $(\xi, \eta, \rho)$. 
	Before going on, we can rewrite in a better way the terms $F^h$ and $Q^h$, by using the following version of Taylor's theorem with integral remainder for any real function $f \in \mathcal{C}^2$ at a point $x_0 \in \R$:
	\[ f(x) = f(x_0) + f'(x_0) (x-x_0) + \left( \int_0^1 (1-z) f''(x_0 + z(x-x_0)) \, \de z \right) (x-x_0)^2. \]
	Indeed, with straightforward calculations one can see that
	\begin{align*}
		& F^h = F''(\phib) \psi + R_1^h (\phi - \phib)^2, \\
		& U^h = \hh'(\phib) \psi \ub + (\hh(\phi) - \hh(\phib)) h + R_2^h (\phi - \phib)^2 \ub, 
	\end{align*}
	and also, up to adding and subtracting some additional terms, that
	\begin{align*}
		& Q^h = P(\phib) (\reactiondiff) + P'(\phib) (\reactionbar) \, \psi  \\
		& \qquad + (P(\phi) - P(\phib))[ (\sigma - \sigmab) - (\mu - \mub) ] + R_3^h (\reactionbar) (\phi - \phib)^2,
	\end{align*}
	where
	\begin{gather*}
		R_1^h = \int_0^1 (1-z) F'''(\phib + z(\phi - \phib)) \, \de z, \quad R_2^h = \int_0^1 (1-z) \hh''(\phib + z(\phi - \phib)) \, \de z, \\
		R_3^h = \int_0^1 (1-z) P''(\phib + z(\phi - \phib)) \, \de z. 
	\end{gather*}
	We observe that, exactly as in  \cite[Proof of Theorem 4.4]{F2023_viscous}, by exploiting the strong regularity of $\phi$ and $\phib$ given by Theorem \ref{strong:sols}, together with hypotheses \ref{ass:f2} and \ref{ass:ph3}, we can show that there exists a constant $C_\Lambda \gs 0$, depending only on the parameters of the system and possibly on $\Lambda$, such that 
	\begin{equation}
		\label{remainder}
		\norm{R_1^h}_{L^\infty(Q_T)}, \norm{R_2^h}_{L^\infty(Q_T)}, \norm{R_3^h}_{L^\infty(Q_T)} \le C_\Lambda \quad \text{and} \quad \norm{\nabla R_1^h}_{L^\infty(0,T;L^6(\Omega))} \le C_{\Lambda}.
	\end{equation} 
	To show \eqref{frechet:aim}, we now proceed by performing a priori estimates on the system \eqref{eq:psi}--\eqref{icd}; however, note that, due to the low regularity, these should be done through a proper Faedo-Galerkin discretisation scheme, by passing to the limit. Nevertheless, here we stick to formal estimates to give the idea of the procedure and leave the discretisation details to the interested reader. Indeed, the main estimate is done by testing \eqref{eq:psi} by $\psi$, \eqref{eq:theta} by $\theta$ and summing them up to obtain:
	\begin{align*}
		\mezzo \ddt \norm{\psi}^2_H + \mezzo \ddt \norm{\theta}^2_H + (\nabla \zeta, \nabla \psi)_H + \norm{\nabla \theta}^2_H = (Q^h, \psi - \theta)_H - (U^h,\psi)_H.
	\end{align*}  
	By computing $\nabla \zeta$ through equation \eqref{eq:zeta} and by using the expression of $F^h$, hypotheses \ref{ass:j} and \ref{ass:fc0}, Remark \ref{rmk:normsfp_infty}, \eqref{remainder} and H\"older, Gagliardo-Nirenberg  \eqref{gn:ineq} and Young's inequalities, together with Sobolev embeddings, we infer that
	\begin{align*}
		(\nabla \zeta, \nabla \psi)_H & = ((AF''(\phib) + Ba) \nabla \psi, \nabla \psi)_H + (A F'''(\phib) \nabla \phib \, \psi, \nabla \psi)_H + (A \nabla R_1^h (\phi-\phib)^2, \nabla \psi)_H \\
		& \quad + (AR_1^h 2(\phi-\phib)(\nabla \phi - \nabla \phib), \nabla \psi)_H + (B \nabla a \, \psi, \nabla \psi)_H - (B \nabla J \ast \psi, \nabla \psi)_H  \\
		& \ge c_0 \norm{\nabla \psi}^2_H - A \norm{F'''(\phib)}_{\Lx \infty} \norm{\nabla \phib}_{\Lx 6} \norm{\psi}_{\Lx 3} \norm{\nabla \psi}_H \\
		& \quad - A \norm{\nabla R_1^h}_{\Lx 6} \norm{\phi - \phib}^2_{\Lx 6} \norm{\nabla \psi}_H \\
		& \quad - 2A \norm{R_1^h} \norm{\phi - \phib}_{\Lx 4} \norm{\nabla \phi - \nabla \phib}_{\Lx 4} \norm{\nabla \psi}_H - 2B\bstar \norm{\psi}_H \norm{\nabla \psi}_H \\
		& \ge c_0 \norm{\nabla \psi}^2_H - C \norm{\nabla \phib}_{\Lx 6} \norm{\psi}^{1/2}_H \norm{\psi}^{3/2}_V - C_\Lambda \norm{\phi - \phib}^2_{V} \norm{\nabla \psi}_H \\
		& \quad - C_\Lambda \norm{\phi - \phib}_V \norm{\phi - \phib}_{\Hx 2} \norm{\nabla \psi}_H - C \norm{\psi}_H \norm{\nabla \psi}_H \\
		& \ge \frac{c_0}{2} \norm{\nabla \psi}^2_H - C \left( 1 + \norm{\nabla \phib}^4_{\Lx 6} \right) \norm{\psi}^2_H \\
		& \quad - C_\Lambda \norm{\phi - \phib}^4_V - C_\Lambda \norm{\phi - \phib}^2_V \norm{\phi - \phib}^2_{\Hx 2},
	\end{align*} 
	where $\norm{\nabla \phib}^4_{\Lx 6} \in \Lt \infty$ by Theorem \ref{strong:sols}. Next, to estimate the right-hand side, we use the definition of $Q^h$, together with the local Lipschitz continuity of $P$, Remark \ref{rmk:normsfp_infty}, \eqref{remainder}, H\"older and Young's inequalities and the Sobolev embedding $V \hookrightarrow \Lx 4$. Indeed, we have that
	\begin{align*}
		(Q^h, \psi - \theta)_H & = (P(\phib) (\reactiondiff), \psi - \theta)_H + (P'(\phib) (\reactionbar) \, \psi, \psi - \theta)_H  \\
		& \quad + ((P(\phi) - P(\phib))[ (\sigma - \sigmab) - (\mu - \mub) ], \psi - \theta)_H \\
		& \quad + (R_2^h (\reactionbar) (\phi - \phib)^2, \psi - \theta)_H \\
		& \le C \norm{\psi}^2_H + C \norm{\theta}^2_H + C \abs{(\zeta, \psi - \theta)_H} + C \norm{\reactionbar}^2_{\Lx \infty} \norm{\psi}^2_H \\
		& \quad + C \norm{\phi - \phib}^2_{V} \left( \norm{\sigma - \sigmab}^2_{V} + \norm{\phi - \phib}^2_{V} + \norm{\mu - \mub}^2_{V} \right) \\
		& \quad + C_\Lambda \norm{\reactionbar}^2_{\Lx \infty} \norm{\phi - \phib}^4_{V}, 
	\end{align*}
	where $\norm{\reactionbar}^2_{\Lx \infty} \in \Lt \infty$ by Theorem \ref{strong:sols}.
	Moreover, to estimate the remaining term $\abs{(\zeta, \psi - \theta)_H}$, we use \eqref{eq:zeta} and similar techniques to see that
	\begin{align*}
		\abs{(\zeta, \psi - \theta)_H} & = \abs{ (AF''(\phib) \psi, \psi)_H + (AR_1^h (\phi - \phib)^2, \psi)_H + (Ba \psi, \psi)_H - B (J \ast \psi, \psi)_H } \\
		& \le C \norm{\psi}^2_H + C_\Lambda \norm{\phi - \phib}^4_H.
	\end{align*}
	Finally, for the last term we use the local Lipschitz continuity of $\hh$, Remark \ref{rmk:normsfp_infty} and \eqref{remainder}, together with the embedding $V \hookrightarrow \Lx4$ and H\"older and Young's inequalities, yielding
	\begin{align*}
		(U^h, \psi)_H & = (\hh'(\phib) \psi \ub + (\hh(\phi) - \hh(\phib)) h + R_2^h (\phi - \phib)^2 \ub, \psi)_H \\
		& \le C \norm{\psi}^2_H + \norm{ \hh(\phi) - \hh(\phib) }_{\Lx4} \norm{h}_H \norm{\psi}_{\Lx4} + C \norm{\phi - \phib}^2_{\Lx4} \norm{\psi}_H \\
		& \le \frac{c_0}{4} \norm{\psi}^2_V + C \norm{\psi}^2_H + C \norm{h}^2_H \norm{\phi - \phib}^2_V + C \norm{\phi - \phib}^4_V.
	\end{align*}
	{\allowdisplaybreaks Therefore, by putting all together and integrating on $(0,t)$, for any $t \in (0,T)$, we arrive at the inequality:
	\begin{align*}
		& \norm{\psi(t)}^2_H + \norm{\theta(t)}^2_H + \frac{c_0}{4} \int_0^t \norm{\nabla \psi}^2_H \, \de s + \int_0^t \norm{\nabla \theta}^2_H \, \de s \\
		& \quad \le C \int_0^T \left(  1 + \norm{\nabla \phib}^4_{\Lx 6} + \norm{\reactionbar}^2_{\Lx \infty} \right) \norm{\psi}^2_H \, \de s + C \int_0^T \norm{\theta}^2_H \, \de s  \\
		& \qquad + C \int_0^T \norm{\phi -  \phib}^4_V \, \de t + C \norm{\phi - \phib}^2_{\LT \infty V} \int_0^T  \norm{\sigma - \sigmab}^2_{V} + \norm{\phi - \phib}^2_{V} + \norm{\mu - \mub}^2_{V} \, \de s \\
		& \qquad + C \norm{\phi - \phib}^2_{\LT \infty V} \int_0^T \norm{h}^2_H \, \de s + C \norm{\phi - \phib}^2_{\LT \infty V} \int_0^T \norm{\phi - \phib}^2_{\Hx 2} \, \de s.
	\end{align*} }%
	Next, we apply the continuous dependence result given by Theorem \ref{thm:contdepstrong} on the terms depending on the differences between $(\phi, \mu, \sigma)$ and $(\phib, \mub, \sigmab)$, to deduce that
	\begin{align*}
		& \norm{\psi(t)}^2_H + \norm{\theta(t)}^2_H + \frac{c_0}{4} \int_0^t \norm{\nabla \psi}^2_H \, \de s + \int_0^t \norm{\nabla \theta}^2_H \, \de s \le C \int_0^T \norm{\theta}^2_H \, \de s \\
		& \quad + C \int_0^T \left(  1 + \norm{\nabla \phib}^4_{\Lx 6} + \norm{\reactionbar}^2_{\Lx \infty} \right) \norm{\psi}^2_H \, \de s + C \norm{h}^4_{L^2(Q_T)} + C \norm{k}^4_{L^2(Q_T)}. 
	\end{align*}
	Then, by applying Gronwall's lemma, we obtain the estimate:
	\begin{equation}
		\label{frechet:est1}
		\norm{\psi}^2_{\LT \infty H \cap \LT 2 V} + \norm{\theta}^2_{\LT \infty H \cap \LT 2 V} \le C \norm{h}^4_{L^2(Q_T)} + C \norm{k}^4_{L^2(Q_T)}.
	\end{equation}
	Moreover, by comparison in \eqref{eq:zeta} and by using \eqref{frechet:est1}, we can easily see that also 
	\begin{equation}
		\label{frechet:est2}
		\norm{\zeta}^2_{\LT \infty H} \le C \norm{h}^4_{L^2(Q_T)} + C \norm{k}^4_{L^2(Q_T)}.
	\end{equation}
	Additionally, by testing \eqref{eq:zeta} by $- \Delta \zeta$, integrating by parts and performing similar estimates to the ones done when studying $(\nabla \zeta, \nabla \psi)_H$ above, we can infer that
	\begin{align*}
		\norm{\nabla \zeta}^2_H & \le \mezzo \norm{\nabla \zeta}^2_H + C \left(  1 + \norm{\nabla \phib}^4_{\Lx 6} \right) \norm{\psi}^2_H + C \norm{\nabla \psi}^2_H \\
		& \quad + C \norm{\phi - \phib}^4_V + C \norm{\phi - \phib}^2_V \norm{\phi - \phib}^2_{\Hx  2}.
	\end{align*} 
	Hence, by integrating on $(0,T)$ and using \eqref{frechet:est1} and Theorem \ref{thm:contdepstrong}, we obtain that 
	\begin{equation}
		\label{frechet:est3}
		\norm{\zeta}^2_{\LT 2 V} \le C \norm{h}^4_{L^2(Q_T)} + C \norm{k}^4_{L^2(Q_T)}.
	\end{equation}
	Finally, by comparison in \eqref{eq:psi} and \eqref{eq:theta} and by using \eqref{frechet:est1} and \eqref{frechet:est3}, we can also easily infer that 
	\begin{equation}
		\label{frechet:est4}
		\norm{\psi}^2_{\HT 1 {V^*}} + \norm{\theta}^2_{\HT 1 {V^*}} \le C \norm{h}^4_{L^2(Q_T)} + C \norm{k}^4_{L^2(Q_T)}.
	\end{equation}
	Therefore, by putting together \eqref{frechet:est1}, \eqref{frechet:est2}, \eqref{frechet:est3} and \eqref{frechet:est4}, we realise that we have actually shown \eqref{frechet:aim} with $s = 4 \gs 2$; thus the proof is concluded.
\end{proof}

\subsection{Adjoint system and optimality conditions}

In order to write down the necessary conditions of optimality in a form which is suitable for applications, we now introduce the adjoint system to the optimal control problem (CP). Indeed, we fix an optimal state $(\phib, \mub, \sigmab) = \mathcal{S}(\ub, \vb)$. Then, by using the formal Lagrangian method with adjoint variables $(p,q,r)$, one can find that the adjoint system, which is formally solved by these variables, has the following form:
\begin{alignat}{2}
	& - \partial_t p + AF''(\phib) q + Baq - BJ \ast q - P'(\phib) (\reactionbar)(p-r) \nonumber \\ 
	& \qquad + p \hh'(\phib) \ub = \alpha_Q (\phib - \phi_Q) \qquad && \text{in } Q_T,  \label{eq:p}\\
	& - q - \Delta p + P(\phib)(p-r) = 0 \qquad && \text{in } Q_T,  \label{eq:q} \\
	& - \partial_t r - \Delta r - P(\phib) (p-r) = \beta_Q (\sigmab - \sigma_Q) \qquad && \text{in } Q_T, \label{eq:r}
\end{alignat}
together with the following boundary and final conditions:
\begin{alignat}{2}
	& \partial_{\n} p = \partial_{\n} r = 0 \qquad && \text{on } \Sigma_T, \label{bca} \\
	& p(T) = \alpha_\Omega (\phib(T) - \phi_\Omega), \quad r(T) = \beta_\Omega (\sigmab(T) - \sigma_\Omega) \qquad && \text{in } \Omega. \label{fca}
\end{alignat}
First, we prove the well-posedness of this adjoint system in the following theorem.

\begin{theorem}
	\label{thm:adjoint}
	Assume hypotheses \emph{\ref{ass:coeff}--\ref{ass:h}}, \emph{\ref{ass:j2}--\ref{ass:f2}} and \emph{\ref{C1}--\ref{ass:initial3}}. Let $(\phib, \mub, \sigmab) \in \mathbb{X}$ be the strong solution to \eqref{eq:phi2}--\eqref{ic2}, corresponding to $(\ub, \vb) \in \Uad \times \Vad$. Then, the adjoint system \eqref{eq:p}--\eqref{fca} admits a unique weak solution such that
	\begin{align*}
		& p \in H^1(0,T;V^*) \cap L^\infty(0,T; H) \cap L^2(0,T;V), \\
		& q \in L^2(0,T;V^*), \\
		& r \in H^1(0,T;V^*) \cap L^\infty(0,T; H) \cap L^2(0,T;V),
	\end{align*}
	which fulfils \eqref{eq:p}--\eqref{fca} in variational formulation, i.e.~it satisfies
	\begin{align}
		& \duality{-\partial_t p, w}_V + \duality{q, (AF''(\phib) + Ba) w}_V - B \duality{q, J \ast w}_V - (P'(\phib) (\reactionbar)(p-r),w)_H \nonumber \\
		& \qquad  + (p \hh'(\phib) \ub, w)_H = (\alpha_Q (\phib - \phi_Q),w)_H \label{varform:p} \\
		& \duality{q,w}_V = (\nabla p, \nabla w)_H + (P(\phib)(p-r),w)_H \label{varform:q} \\
		& \duality{- \partial_t r, w}_V + (\nabla r, \nabla w)_H - (P(\phib)(p-r),w)_H = (\beta_Q (\sigmab - \sigma_Q),w)_H \label{varform:r}
	\end{align}
	for a.e.~$t \in (0,T)$ and for any $w \in V$.
\end{theorem}

\begin{remark}
	When writing the variational formulation \eqref{varform:p}, we used hypothesis \ref{ass:j} on the symmetry of the kernel $J$ to move the convolution operator onto the test function. Indeed, by a simple change of variables, one can formally see that
	\begin{align*}
		\intom (J \ast q) \, w \, \de x & = \intom \left( \intom J(x-y) q(y) \, \de y \right) w(x) \, \de x \\
		& = \intom \left( \intom J(y-x) w(x) \, \de x \right) q(y) \, \de y = \intom q \, (J \ast w) \, \de x.
	\end{align*}
\end{remark}

\begin{proof}
	We observe that \eqref{eq:p}--\eqref{fca} is a backward linear parabolic system, therefore it is not difficult to prove the existence of a solution through a Faedo-Galerkin discretisation scheme. Here we proceed with formal a priori estimates for the sake of brevity. Indeed, for the main estimate, we test \eqref{varform:p} with $p$, \eqref{varform:r} with $r$ and we sum them up to obtain:
	\begin{equation}
		\label{adjoint:test}
		\begin{split}
		& - \mezzo \ddt \norm{p}^2_H - \mezzo \ddt \norm{r}^2_H + \duality{q, (AF''(\phib) + Ba) p}_V - B \duality{q, J \ast p}_V + \norm{\nabla r}^2_H \\
		& \quad = (P(\phib)(p-r), r)_H + (P'(\phib)(\reactionbar)(p-r), p)_H + (p \hh'(\phib) \ub, p)_H \\
		& \qquad + (\alpha_Q (\phib - \phi_Q),p)_H + (\beta_Q (\sigmab - \sigma_Q),r)_H.
		\end{split}
	\end{equation}
	Next, we estimate the two terms involving the duality with $q$ by using equation \eqref{varform:q}, together with hypotheses \ref{ass:j} and \ref{ass:fc0}, Remark \ref{rmk:normsfp_infty} and H\"older, Gagliardo-Nirenberg \eqref{gn:ineq} and Young's inequalities. Indeed, we infer that
	\begin{align*}
		& \duality{q, (AF''(\phib) + Ba) p}_V - B \duality{q, J \ast p}_V = \\
		& \quad = (\nabla p, (AF''(\phib) + Ba) \nabla p)_H + (\nabla p, (AF'''(\phib) \nabla \phib + B \nabla a) p)_H  + (P(\phib)(p-r), p) \\
		& \qquad - B(\nabla p, \nabla J \ast p)_H - B(P(\phib)(p-r), J \ast p)_H \\
		& \quad \ge c_0 \norm{\nabla p}^2_H -  C \norm{\nabla p}_H \norm{\nabla \phib}_{\Lx 6} \norm{p}_{\Lx 3} - 2B\bstar \norm{\nabla p}_H \norm{p}_H - C (\norm{p}_H + \norm{r}_H) \norm{p}_H \\
		& \quad \ge c_0 \norm{\nabla p}^2_H -  C \norm{\nabla \phib}_{\Lx 6} \norm{p}^{1/2}_H \norm{p}^{3/2}_V- 2B\bstar \norm{\nabla p}_H \norm{p}_H - C (\norm{p}_H + \norm{r}_H) \norm{p}_H \\
		& \quad \ge \frac{c_0}{2} \norm{\nabla p}^2_H - C \left(1 + \norm{\nabla \phib}^4_{\Lx 6} \right) \norm{p}^2_H - C \norm{r}^2_H,
	\end{align*}
	where $\norm{\nabla \phib}^4_{\Lx 6} \in \Lt \infty$ by Theorem \ref{strong:sols} and Sobolev embeddings. Moreover, we also estimate the right-hand side of \eqref{adjoint:test} by using Remark \ref{rmk:normsfp_infty}, hypothesis \ref{C2} and Cauchy-Schwarz and Young's inequalities as follows:
	\begin{align*}
		& (P(\phib)(p-r), r)_H + (P'(\phib)(\reactionbar)(p-r), p)_H + (p \hh'(\phib) \ub, p)_H \\
		& \qquad + (\alpha_Q (\phib - \phi_Q),p)_H + (\beta_Q (\sigmab - \sigma_Q),r)_H \\
		& \quad \le C \left(1 + \norm{\reactionbar}^2_{\Lx \infty} \right) \norm{p}^2_H + C \norm{r}^2_H + C \norm{\alpha_Q (\phib - \phi_Q)}^2_H + C \norm{\beta_Q (\sigmab - \sigma_Q)}^2_H, 
	\end{align*}
	where $\norm{\reactionbar}^2_{\Lx \infty} \in \Lt\infty$ by Theorem \ref{strong:sols}.
	Then, by putting all together and integrating on $(t,T)$, for any $t \in (0,T)$, from \eqref{adjoint:test}, we deduce that 
	\begin{align*}
		& \mezzo \norm{p(t)}^2_H + \mezzo \norm{r(t)}^2_H + \frac{c_0}{2} \int_t^T \norm{\nabla p}^2_H \, \de s + \int_t^T \norm{\nabla r}^2_H \, \de s \\
		& \quad \le \norm{\alpha_\Omega (\phib(T) - \phi_\Omega)}^2_H + \norm{\beta_\Omega (\sigmab(T) - \sigma_\Omega)}^2_H + C \int_0^T \norm{r}^2_H \, \de s \\
		& \qquad + C \int_0^T \left(1 + \norm{\nabla \phib}^4_{\Lx 6} + \norm{\reactionbar}^2_{\Lx \infty} \right) \norm{p}^2_H \, \de s \\
		& \qquad + C \int_0^T \norm{\alpha_Q (\phib - \phi_Q)}^2_H \, \de s + C \int_0^T \norm{\beta_Q (\sigmab - \sigma_Q)}^2_H \, \de s.
	\end{align*}
	Therefore, by using Gronwall's lemma, together with hypothesis \ref{C2}, we conclude that 
	\begin{equation}
		\label{adjoint:mainest}
		\begin{split}
			& \norm{p}^2_{\LT \infty H \cap \LT 2 V} + \norm{r}^2_{\LT \infty H \cap \LT 2 V} \\
			& \quad \le C \Big( \norm{\alpha_\Omega (\phib(T) - \phi_\Omega)}^2_H + \norm{\beta_\Omega (\sigmab(T) - \sigma_\Omega)}^2_H \\ 
			& \qquad \qquad + \int_0^T \norm{\alpha_Q (\phib - \phi_Q)}^2_H \, \de s + \int_0^T \norm{\beta_Q (\sigmab - \sigma_Q)}^2_H \, \de s \Big).
		\end{split}
	\end{equation}
	Moreover, by comparison in \eqref{varform:p}, \eqref{varform:q} and \eqref{varform:r}, we also easily see that \eqref{adjoint:mainest} implies the following estimate:
	\begin{equation}
		\label{adjoint:dualest}
		\begin{split}
			& \norm{p}^2_{\HT 1 {V^*}} + \norm{q}^2_{\LT 2 {V^*}} + \norm{r}^2_{\HT 1 {V^*}} \\
			& \quad \le C \Big( \norm{\alpha_\Omega (\phib(T) - \phi_\Omega)}^2_H + \norm{\beta_\Omega (\sigmab(T) - \sigma_\Omega)}^2_H \\ 
			& \qquad \qquad + \int_0^T \norm{\alpha_Q (\phib - \phi_Q)}^2_H \, \de s + \int_0^T \norm{\beta_Q (\sigmab - \sigma_Q)}^2_H \, \de s \Big).
		\end{split}
	\end{equation}
	All these estimates can, then, be repeated in a proper discretisation framework and, by passing to the limit, one can prove the existence of a solution with the sought regularities. Moreover, being the system linear, estimates \eqref{adjoint:mainest} and \eqref{adjoint:dualest} also imply the uniqueness of the solution. This concludes the proof of Theorem \ref{thm:adjoint}. 
\end{proof}

To conclude, with the adjoint variables, we can finally determine and then simplify the first-order necessary conditions. Indeed, we have the following result:

\begin{theorem}
	\label{thm:optcond}
	Assume hypotheses  \emph{\ref{ass:coeff}--\ref{ass:fgrowth}}, \emph{\ref{ass:j2}--\ref{ass:f2}} and \emph{\ref{C1}--\ref{ass:initial3}}. Let $(\ub, \vb) \in \Uad \times \Vad$ be an optimal control for \emph{(CP)} and let $(\phib, \mub, \sigmab) = \mathcal{S}(\ub, \vb) \in \mathbb{X}$ be the corresponding optimal state, i.e.~the solution of \eqref{eq:phi2}--\eqref{ic2} with such $(\ub, \vb)$. Let also $(p,q,r)$ be the adjoint variables to $(\phib, \sigmab, \mub)$, i.e.~the solutions to the adjoint system \eqref{eq:p}--\eqref{fca}. Then, they satisfy the following variational inequality, which holds for any $(u,v) \in \Uad \times \Vad$:
	\begin{equation}
		\label{var:ineq}
		\int_0^T \int_\Omega (- \hh(\phib) p + \alpha_u \ub )(u - \ub) \, \de x  \, \de t + \int_0^T \int_\Omega ( r + \beta_v \vb) (v - \vb) \, \de x  \, \de t \ge 0. 
	\end{equation}
\end{theorem}

{\allowdisplaybreaks
\begin{proof}
	First observe that the cost functional $\mathcal{J}$ is convex and Fr\'echet-differentiable in the space $\mathcal{C}^0([0,T];H) \times \mathcal{C}^0([0,T];H) \times L^2(Q_T) \times L^2(Q_T)$. Next, in Theorem \ref{thm:frechet} we showed that the solution operator $\mathcal{S}$ is Fr\'echet-differentiable from $\mathcal{U}_R \times \mathcal{V}_R \subseteq L^\infty(Q_T)^2$ to $\mathbb{W}$. Consequently, since by standard results $\LT 2 V \cap \HT 1 {V^*}$ is embedded with continuity in  $\C 0 H$, we also have that the operator $(\mathcal{S}_1, \mathcal{S}_3)$ that selects the first and third components of $\mathcal{S}$ is Fr\'echet-differentiable from $\mathcal{U}_R \times \mathcal{V}_R$ to $(\C 0 H)^2$. Therefore, we can consider the \emph{reduced cost functional} $f: \Lqt \infty^2 \to \R$, defined as 
	\[ f(u,v) := \mathcal{J}(\mathcal{S}_1(u,v), \mathcal{S}_3(u,v), u,v),  \]  
	which, by the chain rule, is Fr\'echet-differentiable in $\mathcal{U}_R \times \mathcal{V}_R$. 
	
	At this point, we can rewrite our optimal control problem (CP) through the reduced cost functional as the minimisation problem
	\[ \argmin_{(u,v) \in \, \Uad \times \Vad} f(u,v). \]
	Then, if $(\ub, \vb)$ is optimal, since $\Uad \times \Vad$ is convex and $f$ is Fr\'echet-differentiable, it has to satisfy the necessary optimality condition
	\[ f'(\ub, \vb) [ (u - \ub, v - \vb) ] \ge 0 \quad \text{for any } (u,v) \in \Uad \times \Vad. \]
	Hence, by computing explicitly the derivative of $f$, we get that for any $(u,v) \in \Uad \times \Vad$ 
	\begin{align*}
		& \int_\Omega \alpha_\Omega (\phib(T) -\phi_\Omega) \xi(T) \, \de x  + \int_0^T \int_\Omega \alpha_Q (\phib - \phi_Q) \xi \, \de x  \, \de t  \\
		& + \int_\Omega \beta_\Omega (\sigmab(T) -\sigma_\Omega) \rho(T) \, \de x  + \int_0^T \int_\Omega \beta_Q (\sigmab - \sigma_Q) \rho \, \de x \, \de t  \\
		& + \int_0^T \int_\Omega  \alpha_u \ub (u - \ub) \, \de x  \, \de t + \int_0^T \int_\Omega \beta_v \vb (v - \vb) \,\de x  \, \de t \ge 0,
	\end{align*}
	where $\xi = D \mathcal{S}_1(\ub, \vb)[u - \ub, v - \vb]$ and $\rho = D \mathcal{S}_3(\ub, \vb)[u - \ub, v - \vb]$ are the components of the solution $(\xi, \eta, \rho)$ to the linearised system \eqref{eq:xi}--\eqref{icl} in $(\phib, \mub, \sigmab)$ corresponding to $h = u - \ub$ and $k = v - \vb$. Note that in what follows we are going to write all the integral terms as if they were in strong form, however keep in mind that, in our regularity setting, all products involving time-derivatives, laplacians and $q$ have to be intended as duality products.
	
	Now observe that the right-hand sides and the final conditions of the adjoint system appear in this inequality, therefore by substituting equations \eqref{eq:p}, \eqref{eq:r} and \eqref{fca} in the previous expression, we find that for any $(u,v) \in \Uad \times \Vad$ 
	\begin{align*}
		& \int_\Omega p(T) \xi(T) \, \de x  + \int_0^T \int_\Omega \Big( - \partial_t p  + AF''(\phib) q + Baq - BJ \ast q \\ 
		& \qquad - P'(\phib) (\reactionbar)(p-r) + p \hh'(\phib) \ub \Big) \xi \, \de x  \, \de t + \int_\Omega r(T) \rho(T) \, \de x  \\
		& \quad + \int_0^T \int_\Omega \big(- \partial_t r - \Delta r - P(\phib) (p-r) \big) \rho \, \de x  \, \de t \\
		& \quad + \int_0^T \int_\Omega  \alpha_u \ub (u - \ub) \, \de x  \, \de t + \int_0^T \int_\Omega \beta_v \vb (v - \vb) \,\de x  \, \de t \ge 0.
	\end{align*}
	Now we integrate by parts in time, by using also the initial conditions \eqref{icl} on the linearised system, and in space, by using the boundary conditions \eqref{bca} and \eqref{bcl}, and, after cancellations, we find that equivalently for any $(u,v) \in \Uad \times \Vad$ 
	\begin{align*}
		& \int_0^T \int_\Omega \Big( p \xi_t + AF''(\phib) \xi q + Ba \xi q - B(J \ast \xi) q - P'(\phib) (\reactionbar)(p-r) \xi + p \hh'(\phib) \ub \, \xi \Big) \, \de x  \, \de t \\ 
		& \quad + \int_0^T \int_\Omega \big( \rho_t r - \Delta \rho \, r - P(\phib) (p-r) \rho \big) \, \de x  \, \de t \\
		& \quad + \int_0^T \int_\Omega  \alpha_u \ub (u - \ub) \, \de x  \, \de t + \int_0^T \int_\Omega \beta_v \vb (v - \vb) \,\de x  \, \de t \ge 0,
	\end{align*}
	where we also used the symmetry of the kernel $J$. By factoring out $p$, $q$ and $r$ respectively, we can rewrite the previous inequality as
	\begin{align*}
		& \int_0^T \int_\Omega p \left( \xi_t - P(\phib)\rho - P'(\phib) (\reactionbar) \xi + \hh'(\phib) \ub \, \xi \right) \, \de x  \, \de t \\ 
		& \quad + \int_0^T \int_\Omega q \left( AF''(\phib) \xi + Ba \xi - BJ\ast \xi \right) \, \de x  \, \de t \\
		& \quad + \int_0^T \int_\Omega r \left( \rho_t - \Delta \rho + P(\phib)\rho + P'(\phib) (\reactionbar) \xi \right) \, \de x  \, \de t \\
		& \quad +  \int_0^T \int_\Omega  \alpha_u \ub (u - \ub) \, \de x  \, \de t + \int_0^T \int_\Omega \beta_v \vb (v - \vb) \,\de x  \, \de t \ge 0.
	\end{align*}
	Finally, we use equation \eqref{eq:q} and again integration by parts to also get that 
	\begin{align*}
		0 & = \int_0^T \int_\Omega \left( - q - \Delta p + P(\phib)(p-r) \right) \eta \, \de x  \, \de t = \int_0^T \int_\Omega - \eta q  - \Delta \eta \, p +   P(\phib)(p-r) \eta \, \de x  \, \de t.
	\end{align*}
	Then, by adding this to the previous inequality, we at last infer that for any $(u,v) \in \Uad \times \Vad$ 
	\begin{align*}
		& \int_0^T \int_\Omega p \left( \xi_t - \Delta \eta - P(\phib)(\rho - \eta) - P'(\phib) (\reactionbar) \xi + \hh'(\phib) \ub \, \xi \right) \, \de x  \, \de t \\ 
		& \quad + \int_0^T \int_\Omega q \left( - \eta + AF''(\phib) \xi + Ba \xi - BJ\ast \xi \right) \, \de x  \, \de t \\
		& \quad + \int_0^T \int_\Omega r \left( \rho_t - \Delta \rho + P(\phib)(\rho -\eta) + P'(\phib) (\reactionbar) \xi \right) \, \de x  \, \de t \\
		& \quad +  \int_0^T \int_\Omega  \alpha_u \ub (u - \ub) \, \de x  \, \de t + \int_0^T \int_\Omega \beta_v \vb (v - \vb) \,\de x  \, \de t \ge 0.
	\end{align*}
	To conclude, we notice that the expressions enclosed in the parentheses are exactly the equations \eqref{eq:xi}, \eqref{eq:eta}, \eqref{eq:rho} of the linearised system, up to their source terms. Hence, by substituting those into our inequality, we find that for any $(u,v) \in \Uad \times \Vad$  
	\begin{align*}
		& \int_0^T \int_\Omega - p \, \hh(\phib) (u - \ub) \, \de x  \, \de t + \int_0^T \int_\Omega r (v -\vb) \, \de x  \, \de t \\ 
		& \quad +  \int_0^T \int_\Omega  \alpha_u \ub (u - \ub) \, \de x  \, \de t + \int_0^T \int_\Omega \beta_v \vb (v - \vb) \,\de x  \, \de t \ge 0,
	\end{align*}
	which is exactly \eqref{var:ineq}. This concludes the proof of Theorem \ref{thm:optcond}.
\end{proof} 
}

\begin{remark}
	Observe that, since $\Uad \times \Vad$ is closed and convex, \eqref{var:ineq} means that, if $\alpha_u >0$ and $\beta_v > 0$, the optimal control $(\ub, \vb)$ is exactly the $L^2(Q_T)^2$-orthogonal projection of $( \alpha_u^{-1} \hh(\phib) \, p,  - \beta_v^{-1} r )$ onto $\Uad \times \Vad$. 
	In particular, it can be shown that, due to the structure of $\Uad \times \Vad$, the above $L^2(Q_T)^2$-projection has the explicit form:
	\begin{align*}
		& \ub(x,t) = \min \left\{ u_{\text{max}}(x,t), \max \left\{ \alpha_u^{-1} \hh(\phib) \, p(x,t), u_{\text{min}}(x,t) \right\} \right\} \quad \text{for a.e. } (x,t) \in Q_T, \\
		& \vb(x,t) = \min \left\{ v_{\text{max}}(x,t), \max \left\{ - \beta_v^{-1} r(x,t), v_{\text{min}}(x,t) \right\} \right\} \quad \text{for a.e. } (x,t) \in Q_T.
	\end{align*}
	To get this kind of explicit form, it is crucial for $\Uad$ and $\Vad$ to be described by box constraints in $\Lqt\infty$. 
	We stress that our new maximal regularity strategy allowed us to get highly regular solutions by only assuming such constraints on the controls.
	In this sense, the result presented here can be thought as a partial improvement to the optimality conditions proved in \cite{F2023_viscous}, where we also needed the additional $\HT 1 H$-regularity on the control $u$.  
\end{remark}

\section*{Acknowledgements}

The author wishes to express his gratitude to professors Cecilia Cavaterra and Elisabetta Rocca for several fruitful discussions and suggestions, without which this work would not have been possible. 
The author also wishes to thank the anonymous referees, who carefully read the manuscript and provided many comments that improved the quality of the paper.
This research activity has been performed in the framework of the MIUR-PRIN Grant 2020F3NCPX ``Mathematics for industry 4.0 (Math4I4)'' and the GNAMPA (Gruppo Nazionale per l'Analisi Matematica, la Probabilit\`a e le loro Applicazioni) of INdAM (Istituto Nazionale di Alta Matematica).





\end{document}